\documentclass[11pt]{amsart} 

\usepackage{amssymb,hyperref,cite}
\usepackage[margin=1in]{geometry}
\usepackage{tensor,enumitem}
\usepackage[all,cmtip]{xy}
\usepackage{graphicx}

\newcommand{\subgrp}[1]{\langle #1 \rangle}
\newcommand{\set}[1]{\left\{ #1 \right\}}

\newcommand{\bs}[1]{\boldsymbol{#1}}
\newcommand{\wt}[1]{\widetilde{ #1}}
\newcommand{\ol}[1]{\overline{#1}}

\newcommand{\op}{\textup{op}}
\newcommand{\ev}{\textup{ev}}
\newcommand{\odd}{\textup{odd}}
\newcommand{\ve}{\varepsilon}
\newcommand{\dbar}{\ol{d}}
\newcommand{\Ebar}{\ol{E}}

\newcommand{\simrightarrow}{\stackrel{\sim}{\rightarrow}}

\DeclareMathOperator{\coker}{coker}
\DeclareMathOperator{\Dist}{Dist}
\DeclareMathOperator{\End}{End}
\DeclareMathOperator{\Ext}{Ext}
\DeclareMathOperator{\bbExt}{\mathbb{E}xt}
\DeclareMathOperator{\opH}{H}
\DeclareMathOperator{\Hom}{Hom}
\DeclareMathOperator{\id}{id}
\DeclareMathOperator{\im}{im}
\DeclareMathOperator{\Lie}{Lie}

\DeclareMathOperator{\str}{str}
\DeclareMathOperator{\svec}{svec}
\DeclareMathOperator{\Tot}{Tot}
\DeclareMathOperator{\tr}{tr}
\DeclareMathOperator{\Yext}{Yext}

\newcommand{\gotimes}{\tensor[^g]{\otimes}{}}

\renewcommand{\mod}{\, \textup{mod}\, }

\newcommand{\Hbul}{\opH^\bullet}

\newcommand{\F}{\mathbb{F}}
\newcommand{\G}{\mathbb{G}}
\newcommand{\N}{\mathbb{N}}
\newcommand{\Z}{\mathbb{Z}}

\newcommand{\cp}{\mathcal{P}}
\newcommand{\cv}{\mathcal{V}}

\newcommand{\bsa}{\bs{A}}
\newcommand{\bsc}{\bs{c}}

\newcommand{\bsD}{\bs{D}}
\newcommand{\bsd}{\bs{d}}
\newcommand{\bsdelta}{\bs{\Delta}}
\newcommand{\bsE}{\bs{E}}
\newcommand{\bse}{\bs{e}}
\newcommand{\bsg}{\bs{\Gamma}}
\newcommand{\bsi}{\bs{I}}
\newcommand{\bsir}{\bsi^{(r)}}
\newcommand{\bsilr}{{\bsi_\ell}^{(r)}}
\newcommand{\bsirone}{{\bsi_1}^{(r)}}
\newcommand{\bsijzero}{{\bsi_0}^{(j)}}
\newcommand{\bsirzero}{{\bsi_0}^{(r)}}
\newcommand{\bsK}{\bs{K}}
\newcommand{\bsk}{\bs{\kappa}}
\newcommand{\bsl}{\bs{\Lambda}}
\newcommand{\bso}{\bs{\Omega}}
\newcommand{\bsp}{\bs{\cp}}
\newcommand{\bsPi}{\bs{\Pi}}

\newcommand{\bss}{\bs{S}}
\newcommand{\bssigma}{\bs{\Sigma}}
\newcommand{\bst}{\bs{T}}
\newcommand{\bsv}{\bs{\cv}}
\newcommand{\bsvone}{\bsv_{\ol{1}}}
\newcommand{\bsvzero}{\bsv_{\ol{0}}}

\newcommand{\bibsp}{\textup{bi-}\bsp}

\newcommand{\g}{\mathfrak{g}}
\newcommand{\gl}{\mathfrak{gl}}
\newcommand{\fS}{\mathfrak{S}}
\newcommand{\fsvec}{\mathfrak{svec}}

\newcommand{\glmn}{\gl(m|n)}
\newcommand{\glone}{\gl(m|n)_{\ol{1}}}
\newcommand{\glzero}{\gl(m|n)_{\ol{0}}}

\newcommand{\gone}{\g_{\ol{1}}}
\newcommand{\gzero}{\g_{\ol{0}}}
\newcommand{\Vone}{V_{\ol{1}}}
\newcommand{\Voner}{{\Vone}^{(r)}}
\newcommand{\Vzero}{V_{\ol{0}}}
\newcommand{\Vzeror}{{\Vzero}^{(r)}}

\numberwithin{equation}{subsection}

\newtheorem{theorem}{Theorem}[subsection]
\newtheorem{proposition}[theorem]{Proposition}

\newtheorem{corollary}[theorem]{Corollary}
\newtheorem{lemma}[theorem]{Lemma}
\newtheorem{problem}[theorem]{Problem}

\newtheorem*{theorem*}{Theorem}
\newtheorem*{corollary*}{Corollary}

\theoremstyle{definition}
\newtheorem{definition}[theorem]{Definition}

\newtheorem{example}[theorem]{Example}
\newtheorem{remark}[theorem]{Remark}

\title{Cohomological finite-generation for finite supergroup schemes}

\author{Christopher M.\ Drupieski}
\address{Department of Mathematical Sciences, DePaul University, Chicago, IL 60614, USA}
\email{cdrupies@depaul.edu}

\thanks{This work was supported in part by a faculty development grant from the DePaul University College of Science and Health, and by an AMS--Simons Travel Grant.}


\allowdisplaybreaks

\begin{document}

\begin{abstract}
In this paper we compute extension groups in the category of strict polynomial super\-functors and thereby exhibit certain ``universal extension classes'' for the general linear supergroup. Some of these classes restrict to the universal extension classes for the general linear group exhibited by Friedlander and Suslin, while others arise from purely super phenomena. We then use these extension classes to show that the cohomology ring of a finite supergroup scheme---equivalently, of a finite-dimensional cocommutative Hopf superalgebra---over a field is a finitely-generated algebra. Implications for the rational cohomology of the general linear supergroup are also discussed.
\end{abstract}

\maketitle


\section{Introduction}

\subsection{Main results}

Let $k$ be a field of positive characteristic $p$. Friedlander and Suslin introduced the category $\cp$ of strict polynomial functors over $k$ as part of their investigation into the cohomology of finite $k$-group schemes \cite{Friedlander:1997}. They calculated the Yoneda algebra in $\cp$ of the $r$-th Frobenius twist of the identity functor, and thereby exhibited certain ``universal extension classes'' for the general linear group. These extension classes enabled them to show that the cohomology ring of a finite $k$-group scheme, or equivalently of a finite-dimensional cocommutative $k$-Hopf algebra, is a finitely-generated $k$-algebra. The purpose of this article is to extend all of these results to the world of $\Z/2\Z$-graded vector spaces and group schemes. In particular, we prove:

\begin{theorem*}[\ref{thm:fgpgeq3}, \ref{thm:fgp=0}]
Let $k$ be a field, and let $G$ be a finite $k$-supergroup scheme, or equivalently a finite-dimensional cocommutative $k$-Hopf superalgebra. Then the cohomology ring $\Hbul(G,k)$ is a finitely-generated $k$-superalgebra, and for each finite-dimensional $G$-supermodule $M$, the cohomology group $\Hbul(G,M)$ is a finitely-generated $\Hbul(G,k)$-supermodule.
\end{theorem*}

Recall that a superspace is a $\Z/2\Z$-graded vector space. The category $\fsvec$ of $k$-super\-spaces, consisting of the $k$-superspaces as objects and arbitrary linear maps between them as morphisms, admits a tensor product operation with a braiding $T: V \otimes W \rightarrow W \otimes V$ defined by
\[
T(v \otimes w) = (-1)^{\ol{v} \cdot \ol{w}} w \otimes v.
\]
Here $\ol{v},\ol{w} \in \Z_2 :=\Z/2\Z = \{ \ol{0},\ol{1} \}$ denote the degrees of homogeneous elements $v \in V$ and $w \in W$. Then a $k$-Hopf superalgebra is a Hopf algebra object in the category $\fsvec$. Similarly, an affine $k$-super\-group scheme is an affine group scheme object in $\fsvec$. Each ordinary $k$-Hopf algebra can be viewed as a $k$-Hopf superalgebra concentrated in degree $\ol{0}$. Conversely, if $A$ is a $k$-Hopf superalgebra, then $A$ is a subalgebra of the (typically non-cocommutative) ordinary $k$-Hopf algebra $A \# k\Z_2$ (the smash product of $A$ with the group algebra $k\Z_2$), and the cohomology ring of $A \# k\Z_2$ identifies with the $\ol{0}$-graded component of $\Hbul(A,k)$. Thus, our main theorem can be viewed as a generalization of the Friedlander--Suslin finite-generation result to a wider class of Hopf algebra objects or to a wider class of ordinary Hopf algebras. In particular, our theorem provides additional supporting evidence for a conjecture of Etingof and Ostrik \cite[Conjecture 2.18]{Etingof:2004}, which asserts that the cohomology ring of a finite tensor category should be a finitely-generated algebra. Reducing the gradings mod~$2$, our theorem also applies to $\Z$-graded cocommutative Hopf algebras (as defined, e.g., by Milnor and Moore \cite{Milnor:1965}) and group schemes.

\begin{corollary*}
Let $A$ be a finite-dimensional cocommutative graded Hopf algebra (in the sense of Milnor and Moore) over the field $k$. Then the cohomology ring $\Hbul(A,k)$ is a finitely-generated $k$-algebra.
\end{corollary*}


The main focus of this article is the calculation of extension groups in the category $\bsp$ of strict polynomial superfunctors, recently defined by Axtell \cite{Axtell:2013}, and the application of those results to the cohomology of the general linear supergroup $GL(m|n)$. We write $\bsp_d$ for Axtell's category of degree-$d$ ``strict polynomial functors of type I.'' Axtell shows for $m,n \geq d$ that $\bsp_d$ is equivalent to the category of finite-dimensional left supermodules for the Schur superalgebra $S(m|n,d)$, and hence also equivalent to the category of degree-$d$ polynomial representations for the supergroup $GL(m|n)$. Axtell also defines a category of ``strict polynomial functors of type II,'' which for $n \geq d$ is equivalent to the category of finite-dimensional left supermodules for the Schur superalgebra $Q(n,d)$ of type $Q$ defined by Brundan and Kleshchev \cite[\S4]{Brundan:2002}. We don't specifically consider the cohomology of these ``type II'' functors at this time, but we expect that to do so would yield some interesting results related to the representation theory of the supergroup $Q(n)$.

A degree-$d$ homogeneous strict polynomial superfunctor $F$ can be restricted to the subcategory $\bsvzero$ of purely even finite-dimensional $k$-superspaces or to the subcategory $\bsvone$ of purely odd finite-dimensional $k$-superspaces. These two restrictions are each then naturally degree-$d$ homogeneous strict polynomial functors in the sense of \cite{Friedlander:1997}, though they are rarely isomorphic. Now suppose that $k$ is a (perfect) field of characteristic $p > 2$. Given a $k$-superspace $V$ and a positive integer $r$, write $V^{(r)}$ for the $k$-superspace obtained by twisting the $k$-module structure on $V$ by the Frobenius morphism $\lambda \mapsto \lambda^{p^r}$. Since the decomposition $V = \Vzero \oplus \Vone$ of a superspace $V$ into its even and odd subspaces is not functorial (it is not compatible with the composition of odd linear maps), there is no corresponding direct sum decomposition of the identity functor $\bsi: V \mapsto V$. On the other hand, for each $r \geq 1$ the $r$-th Frobenius twist functor $\bsir: V \mapsto V^{(r)}$ decomposes as a direct sum $\bsir = \bsirzero \oplus \bsirone$ in $\bsp_{p^r}$ with $\bsirzero(V) = \Vzeror$ and $\bsirone(V) = \Voner$. This can be thought of as a functor analogue of the fact that the Frobenius morphism for $GL(m|n)$ has image in the underlying purely even subgroup $GL_m \times GL_n$ of $GL(m|n)$. In general, there does not appear to be a natural way to extend ordinary strict polynomial functors to the structure of strict polynomial superfunctors, though Frobenius twists of ordinary strict polynomial functors can be lifted to $\bsp$ in several ways; see Section \ref{subsec:Frobenius}.

The decomposition $\bsir = \bsirzero \oplus \bsirone$ leads to a matrix ring decomposition
\begin{equation} \label{eq:matrixring}
\renewcommand*{\arraystretch}{1.5}
\begin{pmatrix}
\Ext_{\bsp}^\bullet(\bsi_0^{(r)},\bsi_0^{(r)}) & \Ext_{\bsp}^\bullet(\bsi_1^{(r)},\bsi_0^{(r)}) \\
\Ext_{\bsp}^\bullet(\bsi_0^{(r)},\bsi_1^{(r)}) & \Ext_{\bsp}^\bullet(\bsi_1^{(r)},\bsi_1^{(r)})
\end{pmatrix}
\end{equation}
of the Yoneda algebra $\Ext_{\bsp}^\bullet(\bsir,\bsir)$. One of our main results (Theorem \ref{thm:Yonedaalgebra}) is the description of this algebra in terms of certain extension classes
\[
\left.
\begin{aligned}
\bse_i' &\in \Ext_{\bsp}^{2p^{i-1}}(\bsi_0^{(r)},\bsi_0^{(r)}) \\
\bse_i'' &\in \Ext_{\bsp}^{2p^{i-1}}(\bsi_1^{(r)},\bsi_1^{(r)})
\end{aligned}
\right\} \text{ for $1 \leq i \leq r$, and }
\left\{
\begin{aligned}
\bsc_r &\in \Ext_{\bsp}^{p^r}(\bsi_1^{(r)},\bsi_0^{(r)}), \\
\bsc_r^\Pi &\in \Ext_{\bsp}^{p^r}(\bsi_0^{(r)},\bsi_1^{(r)}).
\end{aligned}
\right.
\]
For example, we show that $\Ext_{\bsp}^\bullet(\bsirzero,\bsirzero)$ is a commutative algebra generated by $\bse_1',\ldots,\bse_r'$ and $\bsc_r \cdot \bsc_r^\Pi$ subject only to the relations $(\bse_1')^p = \cdots = (\bse_{r-1}')^p = 0$ and $(\bse_r')^p = \bsc_r \cdot \bsc_r^\Pi$. Additionally, the restriction functor $F \mapsto F|_{\bsvzero}$ induces a surjective map from $\Ext_{\bsp}^\bullet(\bsirzero,\bsirzero)$ to the extension algebra $\Ext_{\cp}^\bullet(I^{(r)},I^{(r)})$ calculated by Friedlander and Suslin. Some notable differences from the classical case include the fact that $\Ext_{\bsp}^\bullet(\bsir,\bsir)$ is noncommutative and contains non-nilpotent elements. We also observe by degree consideration that precomposition with $\bsir$ does not induce an injective map on $\Ext$-groups in $\bsp$; cf.\ \cite[Corollary 1.3]{Franjou:1999}.

\subsection{Organization of the paper}

The paper is organized as follows: In Sections \ref{sec:preliminaries} and \ref{sec:cohomology} we give definitions and basic results that are needed for doing cohomology calculations in the category $\bsp$. Readers acquainted with ordinary strict polynomial functors may find much of this material quite familiar, though some attention should be paid to make the transition from ordinary to $\Z_2$-graded objects. In particular, $\bsp$ is not an abelian category, so care should be taken in defining the relevant $\Ext$-groups (Sections \ref{subsec:projectivesinjectives}--\ref{subsec:cohomology}), in defining operations on $\Ext$-groups (Sections \ref{subsec:cohomology}--\ref{subsec:externalproducts}), and in interpreting homogeneous elements as equivalence classes of $n$-extensions (Section \ref{subsec:extnextensions}).

In Section \ref{sec:yonedaalgebra} we compute the Yoneda algebra $\Ext_{\bsp}^\bullet(\bsir,\bsir)$. Our strategy parallels the inductive approach in \cite{Friedlander:1997} (in turn based on that in \cite{Franjou:1994}) using hypercohomology spectral sequences. For the base case of induction, we consider a super analogue $\bso$ of the de~Rham complex functor, which satisfies a super Cartier isomorphism: $\Hbul(\bso_{pn}) \cong {\bso_n}^{(1)}$. This isomorphism does not, however, preserve the cohomological degree, and this failure more-or-less directly leads to the existence of the extension classes $\bsc_r$ and $\bsc_r^\Pi$. After the base case, the remaining steps in the induction argument rely on the ordinary de Rham complex functor and follow quite closely the arguments in \cite{Friedlander:1997}. One of the key results toward describing the multiplicative structure of the algebra $\Ext_{\bsp}^\bullet(\bsir,\bsir)$ is Lemma \ref{lem:differential}, but the proof of the lemma is deferred until Section \ref{subsec:fgpgeq3}, after we have had a chance to analyze how certain extension classes restrict to the Frobenius kernel of $GL(m|n)$.

Finally, in Section \ref{sec:applications} we present our main applications to the cohomology of the general linear supergroup $GL(m|n)$ and to the cohomology of finite supergroup schemes. In particular, we show that the restrictions to $GL(m|n)$ of certain distinguished extension classes $\bse_r$, $\bse_r^\Pi$, $\bsc_r$, and $\bsc_r^\Pi$ provide in a natural way the extension classes for $GL(m|n)$ conjectured in \cite[\S5.4]{Drupieski:2013b}. Combined with our previous results in \cite{Drupieski:2013b}, this proves cohomological finite-generation for finite supergroup schemes over fields of characteristic $p > 2$. The case $p=2$ reduces to the case of ordinary finite $k$-group schemes (since then every finite $k$-supergroup scheme \emph{is} a finite $k$-group scheme), while the case $p=0$ follows swiftly (though not trivially) from a structure theorem of Kostant; see Section \ref{subsec:fgp=0}. We also show in Section \ref{subsec:implications} that the restriction of $\bse_r$ to $GL(m|n)$ naturally produces a nonzero class in the rational cohomology group $\opH^2(GL(m|n),k)$. This stands in contrast to the well-known fact that if $G$ is a reductive algebraic group, then $\opH^i(G,k) = 0$ for all $i > 0$. It also demonstrates, in contrast to the classical case, that the embedding of the category of polynomial representations for $GL(m|n)$ into the category of all rational representations for $GL(m|n)$ does not induce isomorphisms on extension groups. We are not aware of any other nontrivial calculations of $\opH^i(GL(m|n),k)$ in the literature, though Brundan and Kleshchev have shown that $\opH^1(Q(n),k)$ is a one-dimensional odd superspace \cite[Corollary 7.8]{Brundan:2003a}. An obvious open problem is to calculate the complete structure of the rational cohomology ring $\Hbul(GL(m|n),k)$.

\subsection{Conventions}

Except when indicated, $k$ will denote a perfect field of positive characteristic $p > 2$, and we will follow the notation, terminology, and conventions laid out in \cite[\S2]{Drupieski:2013b}. In particular, we assume that the reader is familiar with the standard sign conventions of ``super'' linear algebra. All vector spaces are $k$-vector spaces, and all unadorned tensor products denote tensor products over $k$. Set $\Z_2 = \Z/2\Z = \{ \ol{0},\ol{1} \}$, and write $V = \Vzero \oplus \Vone$ for the decomposition of a superspace $V$ into its even and odd subspaces. Given a homogeneous element $v \in V$, write $\ol{v} \in \Z_2$ for the $\Z_2$-degree of $v$. Except when indicated, all isomorphisms will arise from even linear maps; we typically reserve the symbol ``$\simeq$'' for isomorphisms arising from odd linear maps. Set $\N = \set{0,1,2,\ldots}$.

\subsection{}

An earlier draft of this manuscript, posted on the arXiv from August 2014 to August 2015 and discussed by the author in a number of venues, asserted in error that the differential in Lemma \ref{lem:differential} is trivial in degree $s = 2p^r$. Consequently, the earlier draft misidentified the multiplicative structure of the algebra $\Ext_{\bsp}^\bullet(\bsir,\bsir)$. The author discovered the error as a result of work in progress with Jonathan Kujawa, investigating the structure of support varieties for restricted Lie superalgebras. Just as Suslin, Friedlander, and Bendel \cite{Suslin:1997,Suslin:1997a} applied the functor cohomology calculations of Friedlander and Suslin \cite{Friedlander:1997} to investigate cohomological support varieties for finite group schemes, so too do we hope to apply the results of this paper to investigate cohomological support varieties for graded group schemes.

\section{Preliminaries} \label{sec:preliminaries}

\subsection{Strict polynomial superfunctors} \label{subsec:definitions}

Let $\fsvec$ be the category whose objects are the $k$-super\-spaces and whose morphisms are the $k$-linear maps between superspaces. The category $\fsvec$ is naturally enriched over itself. Let $\bsv$ be the full subcategory of $\fsvec$ whose objects are the finite-dimensional $k$-superspaces, and let $\fsvec_\ev$ and $\bsv_\ev$ be the underlying even subcategories having the same objects as $\fsvec$ and $\bsv$, respectively, but only the even linear maps as morphisms.\footnote{In \cite{Drupieski:2013b}, $\fsvec$ is denoted $\svec_k$, while $\fsvec_\ev$ is denoted $\fsvec_k$.} Then $\fsvec_\ev$ and $\bsv_\ev$ are abelian categories. Given $V,W \in \bsv$, let $T: V \otimes W \rightarrow W \otimes V$ be the supertwist map, which is defined on homogeneous simple tensors by $T(v \otimes w) = (-1)^{\ol{v} \cdot \ol{w}} w \otimes v$.\footnote{From now on, whenever we state a formula in which a homogeneous degree has been specified, we mean that the formula is true as written for homogeneous elements, and that it extends linearly to non-homogeneous elements.} For each $n \in \N$, there exists a right action of the symmetric group $\fS_n$ on $V^{\otimes n}$ such that the transposition $(i,i+1) \in \fS_n$ acts via the linear map $(1_V)^{\otimes(i-1)} \otimes T \otimes (1_V)^{\otimes(n-i-1)}$. Set $\bsg^n(V) = (V^{\otimes n})^{\fS_n}$. Then $\bsg^n: V \mapsto \bsg^n(V)$ is an endofunctor on $\bsv_\ev$. If $V = \Vzero$, then $\bsg^n(V) = \Gamma^n(V)$, where $\Gamma^n$ denotes the ordinary $n$-th divided power (i.e., symmetric tensor) functor.

For each $A,B \in \bsv$, the supertwist map induces an isomorphism $A^{\otimes n} \otimes B^{\otimes n} \cong (A \otimes B)^{\otimes n}$, which is an isomorphism of $\fS_n$-modules if we consider $A^{\otimes n} \otimes B^{\otimes n}$ as a right $\fS_n$-module via the diagonal map $\fS_n \rightarrow \fS_n \times \fS_n$. This means that $\bsg^n(A) \otimes \bsg^n(B)$ is naturally a subspace of $\bsg^n(A \otimes B)$. In particular, if $\phi: A \otimes B \rightarrow C$ is an even linear map, then there exists an induced even linear map
\begin{equation} \label{eq:gammamap}
\bsg^n(\phi): \bsg^n(A) \otimes \bsg^n(B) \rightarrow \bsg^n(C).
\end{equation}
Now given $n \in \N$, define $\bsg^n \bsv$ to be the category whose objects are the same as those in $\bsv$, whose morphisms are defined by $\Hom_{\bsg^n \bsv}(V,W) = \bsg^n \Hom_k(V,W)$, and in which the composition of morphisms is induced as in \eqref{eq:gammamap} by the composition of linear maps in $\bsv$. Alternatively, there exists by \cite[Lemma 3.1]{Axtell:2013} a natural isomorphism $\bsg^n \Hom_k(V,W) \cong \Hom_{k\fS_n}(V^{\otimes n},W^{\otimes n})$. Then composition in $\bsg^n \bsv$ can be viewed as the composition of $k\fS_n$-module homomorphisms.

\begin{definition} \label{def:superfunctor}
Let $n \in \N$. A \emph{homogeneous strict polynomial superfunctor of degree $n$} is an even linear functor $F: \bsg^n \bsv \rightarrow \bsv$, i.e., a covariant functor $F: \bsg^n \bsv \rightarrow \bsv$ such that for each $V,W \in \bsv$, the function $F_{V,W} : \bsg^n \Hom_k(V,W) \rightarrow \Hom_k(F(V),F(W))$ is an even linear map. Given degree-$n$ homogeneous strict polynomial superfunctors $F$ and $G$, a homomorphism $\eta: F \rightarrow G$ consists for each $V \in \bsv$ of a map $\eta(V) \in \Hom_k(F(V), G(V))$ such that for each $\phi \in \Hom_{\bsg^n \bsv}(V,W)$ one has
\[
\eta(W) \circ F(\phi) = (-1)^{\ol{\eta} \cdot \ol{\phi}} G(\phi) \circ \eta(V).
\]
We denote by $\bsp_n$ the category whose objects are the homogeneous strict polynomial superfunctors of degree $n$ and whose morphisms are the homomorphisms between those functors. The category $\bsp$ of arbitrary strict polynomial superfunctors is defined to be the category $\prod_{n \in \N} \bsp_n$.
\end{definition}

\begin{remark} \label{rem:strictstructure}
Let $\bsg \bsv$ be the category whose objects are the same as those in $\bsv$, whose morphisms are defined by $\Hom_{\bsg \bsv}(V,W) = \prod_{n \in \N} \Hom_{\bsg^n \bsv}(V,W)$, and in which composition of morphisms is defined componentwise using the composition laws in each $\bsg^n \bsv$. Then each $F \in \bsp$ is naturally an even linear functor $F: \bsg \bsv \rightarrow \fsvec$ as follows: Let $F = \bigoplus_{n \in \N} F^n$ be the \emph{polynomial decomposition} of $F$, i.e., the decomposition of $F$ into the sum of its homogeneous components $F^n \in \bsp_n$. Similarly, write $\phi = \prod_{n \in \N} \phi_n$ for the decomposition of $\phi \in \Hom_{\bsg \bsv}(V,W)$. Now $F: \bsg \bsv \rightarrow \fsvec$ is defined on objects by $F(V) = \bigoplus_{n \in \N} F^n(V)$, and on morphisms by $F(\phi) = \prod_{n \in \N} F^n(\phi_n)$, i.e., $F(\phi)$ acts on the summand $F^n(V)$ of $F(V)$ by the linear map $F^n(\phi_n): F^n(V) \rightarrow F^n(W)$.

Conversely, let $F: \bsg \bsv \rightarrow \fsvec$ be an even linear functor. Then for each $V \in \bsg \bsv$, the function $F_{V,V}: \Hom_{\bsg \bsv}(V,V) \rightarrow \Hom_k(F(V),F(V))$ makes $F(V)$ into a left $\Hom_{\bsg \bsv}(V,V)$-module, and the decomposition $\Hom_{\bsg \bsv}(V,V) = \prod_{n \in \N} \Hom_{\bsg^n \bsv}(V,V)$ leads to an expression $\id_V = \prod_{n \in \N} \id_{V,n}$ of the identity morphism as an infinite sum of orthogonal commuting idempotents. Set $F^n(V) = \id_{V,n}(F(V))$, and given $\phi \in \Hom_{\bsg \bsv}(V,W)$, define $F^n(\phi) : F^n(V) \rightarrow F^n(W)$ to be the map ${\id_{W,n}} \circ {F(\phi)} \circ {\id_{V,n}}$. Then $F^n: \bsg^n \bsv \rightarrow \fsvec$ is an even linear functor. If $F^n(V) \in \bsv$ for each $V \in \bsv$, then $F^n \in \bsp_n$. If also $F(V) = \bigoplus_{n \in \N} F^n(V)$ for each $V \in \bsv$, then $F = \bigoplus_{n \in \N} F^n \in \bsp$.
\end{remark}

\begin{remark}
Let $V,W \in \bsv$, and let $\phi \in \Hom_k(V,W)_{\ol{0}}$. Then for $n \in \N$, $\phi^{\otimes n} \in  \bsg^n \Hom_k(V,W)$. Thus, it follows that a homogeneous strict polynomial superfunctor $F \in \bsp_n$ defines an ordinary functor $\bsv_\ev \rightarrow \bsv_\ev$ that acts on objects by $V \mapsto F(V)$ and on morphisms by $\phi \mapsto F(\phi^{\otimes n})$. Throughout the paper, we will often state remarks (such as this one) in the context of homogeneous superfunctors, and then leave it to the reader to consider how those remarks can be extended to the non-homogeneous case, and vice versa.
\end{remark}

The category $\bsp$ is not an abelian category, though the underlying even subcategory $\bsp_\ev$, having the same objects as $\bsp$ but only the even homomorphisms, is an abelian category in which kernels and cokernels are computed ``pointwise'' in the category $\bsv$. More generally, if $\eta$ is a homogeneous homomorphism in $\bsp$, then the kernel, cokernel, and image of $\eta$ are again objects in $\bsp$. Only the even homomorphisms in $\bsp$ are genuine natural transformations between functors.

Observe that $\bsv_\ev = \bsvzero \oplus \bsvone$, where $\bsvzero$ is the full subcategory of $\bsv_\ev$ having as objects just the purely even superspaces, i.e., the superspaces $V$ with $V = \Vzero$, and $\bsvone$ is the full subcategory of $\bsv_\ev$ having as objects just the purely odd superspaces (and as morphisms just the even linear maps between them). In turn, $\bsvzero$ and $\bsvone$ are each isomorphic to the category $\cv$ of arbitrary finite-dimensional $k$-vector spaces. Replacing $\bsv$ by $\bsvzero$ or $\bsvone$ in the definition of $\bsg^n \bsv$, one obtains categories $\bsg^n(\bsvzero)$ and $\bsg^n(\bsvone)$ that are each isomorphic to the ordinary analogue $\Gamma^n \cv$ of $\bsg^n \bsv$.

Let $V,W \in \bsvzero$. Since $\Hom_k(V,W)$ is a purely even space, there exists a natural identification $\Hom_{\bsg^n \bsv}(V,W) = \Hom_{\Gamma^n \cv}(V,W)$. Thus, it makes sense to consider the restriction of $F \in \bsp_n$ to the category $\bsg^n(\bsvzero) \cong \Gamma^n \cv$. Similarly, we can consider the restriction of $F$ to $\bsg^n(\bsvone) \cong \Gamma^n \cv$. By abuse of notation, we denote these restrictions by $F|_{\bsvzero}$ and $F|_{\bsvone}$, respectively. In either case, we can then consider $F: \Gamma^n \cv \rightarrow \bsv$ as a functor to $\cv$ by forgetting the $\Z_2$-gradings on objects in $\bsv$. Then $F|_{\bsvzero}$ and $F|_{\bsvone}$ are homogeneous strict polynomial functors of degree $n$ in the sense of \cite{Friedlander:1997}; cf.\ the comments following \cite[3.7]{Friedlander:2003}. In particular, the restriction map $F \mapsto F|_{\bsvzero}$ defines an exact linear functor from $\bsp$ to the category $\cp$ of ordinary strict polynomial functors, which we call \emph{restriction from $\bsp$ to $\cp$}. In this paper we will consider strict polynomial superfunctors that restrict to well-known ordinary strict polynomial functors. In these situations we often use a boldface symbol for the strict polynomial superfunctor and a non-boldface version of the same symbol for the functor's restriction to $\bsvzero$. In particular, we consider superfunctors that restrict to the divided power algebra functor $\Gamma$ (i.e., the symmetric tensor functor), the symmetric algebra functor $S$, and the exterior algebra functor $\Lambda$ (which is also isomorphic to the anti-symmetric tensor functor). We write $\cp_n$ and $\cp$ for the ordinary analogues of $\bsp_n$ and $\bsp$, respectively.

\begin{remark}
In \cite[\S5]{Axtell:2013}, Axtell provides an alternate definition for the category $\bsp_d$ that more closely matches the original definition given by Friedlander and Suslin. The definition given above, stated by Axtell in \cite[\S3.3]{Axtell:2013}, follows the expositions of Kuhn \cite[\S3.1]{Kuhn:1998} and Pirashvili \cite[\S4.1]{Pirashvili:2003}.
\end{remark}

\subsection{Constructions} \label{subsec:constructions}

Given $F \in \bsp_m$ and $G \in \bsp_n$, one can construct, in exactly the same manner as for ordinary strict polynomial functors,
\begin{enumerate}[label=(\thesubsection.\arabic*)]
\setcounter{enumi}{\value{equation}}
\item the direct sum $F \oplus G \in \bsp$, $V \mapsto F(V) \oplus G(V)$,
\item the composite $F \circ G \in \bsp_{mn}$, $V \mapsto F(G(V))$, and
\item the tensor product $F \otimes G \in \bsp_{m+n}$, $V \mapsto F(V) \otimes G(V)$.
\setcounter{equation}{\value{enumi}}
\end{enumerate}
For example, for each $U \in \bsv$, the inclusion of the Young subgroup $(\fS_n)^{\times m}$ into $\fS_{mn}$ induces an inclusion $\bsg^{mn}(U) \hookrightarrow \bsg^m(\bsg^n(U))$; this inclusion is used to define the action of $F \circ G$ on morphisms. Similarly, the inclusion of the Young subgroup $\fS_m \times \fS_n$ into $\fS_{m+n}$ induces for each $U \in \bsv$ an inclusion $\bsg^{m+n}(U) \hookrightarrow \bsg^m(U) \otimes \bsg^n(U)$, which is used to define the action of $F \otimes G$ on morphisms. We leave further details of these constructions to the reader.

For $U \in \bsv$, set $U^* = \Hom_k(U,k)$, the $k$-linear dual. Given $F \in \bsp_n$, the dual $F^\# \in \bsp_n$ is defined on objects by $F^\#(V) = F(V^*)^*$. On morphisms, $(F^\#)_{V,W}$ is the composite map
\[
\bsg^n \Hom_k(V,W) \cong \bsg^n \Hom_k(W^*,V^*) \stackrel{F_{W^*,V^*}}{\longrightarrow} \Hom_k(F(W^*),F(V^*)) \cong \Hom_k(F^\#(V),F^\#(W)),
\]
where the first and last isomorphisms are induced by sending a linear map $\psi$ to its transpose $\psi^*$. This construction is due to Kuhn \cite[\S3.4]{Kuhn:1994}, and $F^\#$ is sometimes called the `Kuhn dual' of $F$.

Recall that $V$ is naturally isomorphic to its double dual $(V^*)^*$ via the map $\Phi(V): V \rightarrow (V^*)^*$ that is defined for $v \in V$ and $g \in V^*$ by $\Phi(V)(v)(g) = (-1)^{\ol{v} \cdot \ol{g}} g(v)$. Let $\Phi: \bsi \simrightarrow \bsi^\#$ be the natural transformation lifting $\Phi(V)$. Then $F$ identifies with $F^{\#\#}$ via the composite isomorphism
\[
F = F \circ \bsi \stackrel{F \circ \Phi}{\longrightarrow} F \circ \bsi^\# = \bsi \circ F \circ \bsi^\# \stackrel{\Phi \circ (F \circ \bsi^\#)}{\longrightarrow} \bsi^\# \circ F \circ \bsi^\#.
\]
Now let $F,G \in \bsp_n$, and let $\eta \in \Hom_{\bsp_n}(F,G)$. Then $\eta^\# \in \Hom_{\bsp_n}(G^\#,F^\#)$ is defined by $\eta^\#(V) = \eta(V^*)^*$. If $\sigma \in \Hom_{\bsp_n}(G,H)$, then $(\sigma \circ \eta)^\# = (-1)^{\ol{\sigma} \cdot \ol{\eta}} \eta^\# \circ \sigma^\#$. Thus, $(-)^\#$ defines an equivalence of categories $\bsp_n \simeq \bsp_n^{\op,-}$, where $\bsp_n^{\op,-}$ denotes the category with the same objects and morphisms as the opposite category of $\bsp_n$, but in which the composition law has been modified so that $\eta \circ_{\op} \sigma$ is now equal to $(-1)^{\ol{\sigma} \cdot \ol{\eta}} \eta \circ_{\op} \sigma$; cf.\ \cite[\S3.4]{Axtell:2013}.

\subsection{Examples} \label{subsec:examples}

Throughout this section, let $m,n \in \N$, let $V,W \in \bsv$, and let $\phi \in  \bsg^n \Hom_k(V,W)$. We identify $\phi$ with an element of $\Hom_{k\fS_n}(V^{\otimes n},W^{\otimes n})$. Let $\set{x_1,\ldots,x_s}$ be a basis for $\Vzero$ and let $\set{y_1,\ldots,y_t}$ be a basis for $\Vone$.

\subsubsection{Tensor products}

Let $U \in \bsv$. Then the functors $V \mapsto V \otimes U$ and $V \mapsto U \otimes V$ are objects in $\bsp_1$. These functors act on morphisms by sending $\phi \in \Hom_k(V,W)$ to the tensor products of maps $\phi \otimes \id_U$ and ${\id_U} \otimes {\phi}$, respectively. Then the supertwist map $T: V \otimes U \rightarrow U \otimes V$ lifts to an isomorphism $(- \otimes U) \cong (U \otimes -)$. Taking $U = k$, we get the identity functor $\bsi: \bsv \rightarrow \bsv$.

\subsubsection{Parity change}

Let $k^{0|1}$ be a one-dimensional purely odd superspace. Then the parity change functor $\bsPi$ is the functor $k^{0|1} \otimes -$. It acts on objects by reversing the $\Z_2$-grading, and on morphisms by $\bsPi(\phi) = (-1)^{\ol{\phi}} \phi$, i.e., if $\phi: V \rightarrow W$ is an even linear map, then $\bsPi(\phi): \bsPi(V) \rightarrow \bsPi(W)$ is equal to $\phi$ as a map between the underlying spaces, while if $\phi$ is odd, then $\bsPi(\phi) = -\phi$.

In contradiction to our stated convention on the use of boldface and non-boldface versions of the same symbol, set $\Pi = - \otimes k^{0|1}$. Then $\Pi$ acts on objects by $\Pi(V) = \bsPi(V)$, and acts on morphisms by $\Pi(\phi) = \phi$, i.e., $\Pi(\phi)$ is (always) equal to $\phi$ as a map between the underlying spaces.

Write $\id_{V \rightarrow \bsPi(V)}$ for the identity map on $V$ considered as an odd linear map $V \rightarrow \bsPi(V)$. Then $\id_{V \rightarrow \bsPi(V)}$ lifts for each $F \in \bsp$ to an odd isomorphism $\id_{F \rightarrow \bsPi \circ F}: F \simeq \bsPi \circ F$. More generally, for each $F,G \in \bsp$ there exist odd isomorphisms
\begin{equation} \label{eq:paritychangebsphom}
\begin{aligned}
\Hom_{\bsp}(F,G) &\simrightarrow \Hom_{\bsp}(F,\bsPi \circ G), & \eta &\mapsto {\id_{G \rightarrow \bsPi \circ G}} \circ \eta, & \text{and} \\
\Hom_{\bsp}(F,G) &\simrightarrow \Hom_{\bsp}(\bsPi \circ F,G), & \eta &\mapsto \eta \circ \id_{\bsPi \circ F \rightarrow F}
\end{aligned}
\end{equation}
that are natural with respect to even homomorphisms in either variable. Similar comments apply to the odd isomorphism $F \simeq \Pi \circ F$ induced by the superspace map $V \rightarrow V$, $v \mapsto (-1)^{\ol{v}}v$.

\subsubsection{Tensor powers}

The $n$-th tensor power functor $\otimes^n \in \bsp_n$ is defined on objects by $(\otimes^n)(V) = V^{\otimes n}$ and on morphisms by $(\otimes^n)(\phi) = \phi$. Equivalently, $\otimes^n = \bsi^{\otimes n}$. Set $\bst = \bigoplus_{n \in \N} \otimes^n$. Then $\bst(V)$ is the tensor superalgebra on $V$. Since $(\otimes^m) \otimes (\otimes^n) = \otimes^{m+n}$, it follows that the multiplication map on $\bst(V)$ lifts to a natural transformation $m_{\bst}: \bst \otimes \bst \rightarrow \bst$. Also, given $\sigma \in \fS_n$, it follows that the right action of $\sigma$ on $V^{\otimes n}$ lifts to a natural transformation $\otimes^n \rightarrow \otimes^n$, which we also denote $\sigma$.

\subsubsection{Symmetric powers}

The $n$-th symmetric power functor $\bss^n \in \bsp_n$ is defined on objects by
\[
\bss^n(V) = (V^{\otimes n})_{\fS_n} = (V^{\otimes n})/\subgrp{z-(z.\sigma): z \in V^{\otimes n}, \sigma \in \fS_n}.
\]
On morphisms, $\bss^n(\phi) : \bss^n(V) \rightarrow \bss^n(W)$ is the linear map that is naturally induced by $\phi$.

Set $\bss = \bigoplus_{n \in \N} \bss^n$. Then $\bss(V)$ is the symmetric superalgebra on $V$. It is the free commutative superalgebra generated by the superspace $V$. As an algebra, $\bss(V)$ is isomorphic to the ordinary tensor product of algebras $S(\Vzero) \otimes \Lambda(\Vone)$; cf.\ the algebra denoted $S_s(V)$ in \cite[\S2.3]{Drupieski:2013b}. Since the product in $\bss(V)$ is induced by the product in $\bst(V)$, multiplication in $\bss(V)$ lifts to a natural transformation $m_{\bss}: \bss \otimes \bss \rightarrow \bss$. The following set of monomials forms a basis for $\bss(V)$:
\begin{equation} \label{eq:bssbasis}
\{ x_1^{a_1} \cdots x_s^{a_s} y_1^{b_1} \cdots y_t^{b_t} : a_i,b_j \in \N, 0 \leq b_j \leq 1 \}.
\end{equation}

\subsubsection{Exterior powers}

Write $\sigma \mapsto (-1)^{\sigma}$ for the one-dimensional sign representation of $\fS_n$. The $n$-th exterior power functor $\bsl^n \in \bsp_n$ is defined on objects by
\[
\bsl^n(V) = (V^{\otimes n})/\subgrp{z - (-1)^\sigma (z.\sigma): z \in V^{\otimes n}, \sigma \in \fS_n}.
\]
On morphisms, $\bsl^n(\phi): \bsl^n(V) \rightarrow \bsl^n(W)$ is the linear map that is naturally induced by $\phi$.

Set $\bsl = \bigoplus_{n \in \N} \bsl^n$. Then $\bsl(V)$ is the exterior superalgebra on $V$. In the terminology of \cite[\S2.1]{Drupieski:2013b}, $\bsl(V)$ is the free graded-commutative graded super\-algebra generated by $V$ when $V$ is considered as a $\Z$-graded superspace concentrated in degree $1$. As an algebra, $\bsl(V)$ is isomorphic to the graded tensor product of superalgebras $\Lambda(\Vzero) \gotimes S(\Vone)$.\footnote{If $A$ and $B$ are graded superalgebras, then multiplication of homogeneous simple tensors in $A \gotimes B$ is defined by $(a \otimes b)(c \otimes d) = (-1)^{\ol{b} \cdot \ol{c}} (-1)^{\deg(b) \cdot \deg(c)} (ac \otimes bd)$, where $\deg(x)$ denotes the $\Z$-degree of $x$. The $\Z$-degree of $a \otimes b$ is defined by $\deg(a \otimes b) = \deg(a)+\deg(b)$.} As for $\bss(V)$, the multiplication map in $\bsl(V)$ lifts to a natural transformation $m_{\bsl}: \bsl \otimes \bsl \rightarrow \bsl$, and a basis for $\bsl(V)$ is given by the set
\begin{equation} \label{eq:bslbasis}
\{ x_1^{a_1} \cdots x_s^{a_s} y_1^{b_1} \cdots y_t^{b_t} : a_i,b_j \in \N, 0 \leq a_i \leq 1 \}
\end{equation}
of monomials in $\bsl(V)$. The algebra $\bsl(V)$ identifies with the algebra denoted $\Lambda_s(V)$ in \cite[\S2.3]{Drupieski:2013b}.

\subsubsection{Divided powers} \label{subsubsec:dividedpowers}

The $n$-th divided power functor $\bsg^n \in \bsp_n$ is defined on objects by
\[
\bsg^n(V) = (V^{\otimes n})^{\fS_n} = \{ z \in V^{\otimes n} : z.\sigma = z \text{ for all $\sigma \in \fS_n$} \}
\]
and on morphisms by $\bsg^n(\phi) = \phi|_{\bsg^n(V)}$. Set $\bsg = \bigoplus_{n \in \N} \bsg^n$.

Let $J \subseteq \fS_{m+n}$ be a set of right coset representatives for the Young subgroup $\fS_m \times \fS_n$ of $\fS_{m+n}$. Then $\sum_{\sigma \in J} \sigma$ defines a natural transformation $(\otimes^m) \otimes (\otimes^n) \rightarrow \otimes^{m+n}$. This natural transformation restricts to a natural transformation $\bsg^m \otimes \bsg^n \rightarrow \bsg^{m+n}$ that does not depend on the choice of $J$. Summing over all $m,n \in \N$, there exists a natural transformation $m_{\bsg}: \bsg \otimes \bsg \rightarrow \bsg$. Given $v \in \Vzero$ and $n \in \N$, set $\gamma_n(v) = v \otimes \cdots \otimes v \in \bsg^n(V)$.\footnote{The definition of $\gamma_n(v)$ also makes sense if $v \in \Vone$ and $n = 1$, but $\gamma_n(v) \notin \bsg^n(V)$ if $v \in \Vone$ and $n \geq 2$.} By an adaptation of the arguments in \cite[IV.5]{Bourbaki:2003}, one can show that the product $m_{\bsg}(V): \bsg(V) \otimes \bsg(V) \rightarrow \bsg(V)$ makes $\bsg(V)$ into a commutative superalgebra, and that the set
\begin{equation} \label{eq:bsgbasis}
\{ \gamma_{a_1}(x_1) \cdots \gamma_{a_s}(x_s) y_1^{b_1} \cdots y_t^{b_t}: a_i,b_j \in \N, 0 \leq b_j \leq 1 \}.
\end{equation}
of monomials in $\bsg(V)$ is a basis for $\bsg(V)$. In particular, $\bsg(V)$ is isomorphic to the tensor product of algebras $\Gamma(\Vzero) \otimes \Lambda(\Vone)$ (cf.\ \cite[(4) and Lemma 4.1]{Axtell:2013}), and is generated as an algebra by the subspace $\Vone \subseteq \bsg^1(V)$ together with the elements of the form $\gamma_{p^e}(v)$ for $v \in \Vzero$ and $e \geq 0$.

\subsubsection{Alternating powers}

The $n$-th alternating power functor $\bsa^n \in \bsp_n$ is defined on objects by
\[
\bsa^n(V) = \{ z \in V^{\otimes n} : z.\sigma = (-1)^\sigma z \text{ for all $\sigma \in \fS_n$} \}
\]
and on morphisms by $\bsa^n(\phi) = \phi|_{\bsa^n(V)}$. Set $\bsa = \bigoplus_{n \in \N} \bsa^n$.

Again let $J$ be a set of right coset representatives for $\fS_m \times \fS_n$ in $\fS_{m+n}$. Then $\sum_{\sigma \in J} (-1)^\sigma \sigma$ restricts to a natural transformation $\bsa^m \otimes \bsa^n \rightarrow \bsa^{m+n}$ that does not depend on the choice of $J$. Summing over all $m,n \in \N$, there exists a natural transformation $m_{\bsa}: \bsa \otimes \bsa \rightarrow \bsa$. Now arguing as for $\bsg$, one can show that the product $m_{\bsa}(V)$ makes $\bsa(V)$ into a graded-commutative graded superalgebra, with the grading induced by considering $V$ as a graded superspace concentrated in degree $1$. One can also show that the set
\begin{equation} \label{eq:bsabasis}
\set{x_1^{a_1} \cdots x_s^{a_s} \gamma_{b_1}(y_1) \cdots \gamma_{b_t}(y_t): a_i,b_j \in \N, 0 \leq a_i \leq 1}.
\end{equation}
of monomials in $\bsa(V)$ is a basis for $\bsa(V)$, and hence that $\bsa(V)$ is isomorphic as an algebra to the graded tensor product of algebras $\Lambda(\Vzero) \gotimes \Gamma(\Vone)$. In particular, $\bsa(V)$ is generated as an algebra by the subspace $\Vzero \subseteq \bsa^1(V)$ together with the elements of the form $\gamma_{p^e}(v)$ for $v \in \Vone$ and $e \geq 0$. 

\subsection{Bisuperfunctors} \label{subsec:bisuperfunctors}

Let $\bsv \times \bsv$ denote the direct product of the category $\bsv$ with itself. Thus, the objects of $\bsv \times \bsv$ are pairs $(V,W)$ with $V,W \in \bsv$, while a morphism $\phi: (V,W) \rightarrow (V',W')$ in $\bsv \times \bsv$ is a linear map $\phi: V \oplus W \rightarrow V' \oplus W'$ such that $\phi(V) \subseteq V'$ and $\phi(W) \subseteq W'$. Now given $d,e \in \N$, define $\bsg_d^e(\bsv \times \bsv)$ to be the category with the same objects as $\bsv \times \bsv$, but in which morphisms are defined by
\begin{equation} \label{eq:bsgdehom}
\Hom_{\bsg_d^e(\bsv \times \bsv)}((V,W),(V',W')) = \Hom_{\bsg^d \bsv}(V,V') \otimes \Hom_{\bsg^e \bsv}(W,W'),
\end{equation}
and in which the composition law is induced by the composition laws in $\bsg^d(\bsv)$ and $\bsg^e(\bsv)$. In other words, if $\phi_d$ and $\psi_d$ are composable morphisms in $\bsg^d \bsv$, and if $\phi_e$ and $\psi_e$ are composable morphisms in $\bsg^e \bsv$, then $(\phi_d \otimes \phi_e) \circ (\psi_d \otimes \psi_e) = (-1)^{\ol{\phi_e} \cdot \ol{\psi_d}} (\phi_d \circ \psi_d) \otimes (\phi_e \circ \psi_e)$.

\begin{definition} \label{def:bisuperfunctor}
Let $d,e \in \N$. A \emph{strict polynomial bisuperfunctor of bidegree $(d,e)$} is an even linear functor $F: \bsg_d^e(\bsv \times \bsv) \rightarrow \bsv$. Given two such functors $F$ and $G$, a homomorphism $\eta: F \rightarrow G$ consists for each $(V,W) \in \bsv \times \bsv$ of a linear map $\eta(V,W): F(V,W) \rightarrow G(V,W)$ such that for each $\phi \in \Hom_{\bsg_d^e(\bsv \times \bsv)}((V,W),(V',W'))$, one has
\begin{equation} \label{eq:bihom}
\eta(V',W') \circ F(\phi) = (-1)^{\ol{\eta} \cdot \ol{\phi}} G(\phi) \circ \eta(V,W).
\end{equation}
We denote by $\bsp_d^e$ the category whose objects are the strict polynomial bisuperfunctors of bidegree $(d,e)$ and whose morphisms are the homomorphisms between those functors. Given $n \in \N$, the category $\bsp(n)$ of \emph{strict polynomial bisuperfunctors of total degree $n$} is the category $\prod_{d+e = n} \bsp_d^e$, and the category $\bibsp$ of arbitrary strict polynomial bisuperfunctors is the category $\prod_{n \in \N} \bsp(n)$.
\end{definition}

Given $F \in \bsp_d^e$, $F' \in \bsp_r^s$, $G \in \bsp_m$, and $H \in \bsp_n$, one can construct:
\begin{enumerate}[label=(\thesubsection.\arabic*)]
\setcounter{enumi}{\value{equation}}
\item the internal direct sum $F \oplus F' \in \bsp_d^e \oplus \bsp_r^s \subset \bibsp$, $(V,W) \mapsto F(V,W) \oplus F'(V,W)$,
\item the external direct sum $G \boxplus H \in \bsp_m^0 \oplus \bsp_0^n \subset \bibsp$, $(V,W) \mapsto G(V) \oplus H(W)$, \label{eq:externaldirectsum}
\item the composite $G \circ F \in \bsp_{dm}^{em}$, $(V,W) \mapsto G(F(V,W))$,
\item the composite $F \circ (G,H) \in \bsp_{dm}^{en}$, $(V,W) \mapsto F(G(V),H(W))$,
\item the internal tensor product $F \otimes F' \in \bsp_{d+r}^{e+s}$, $(V,W) \mapsto F(V,W) \otimes F'(V,W)$,
\item the external tensor product $G \boxtimes H \in \bsp_m^n$, $(V,W) \mapsto G(V) \otimes H(W)$, and
\item the dual $F^\# \in \bsp_d^e$, $(V,W) \mapsto F(V^*,W^*)^*$.
\setcounter{equation}{\value{enumi}}
\end{enumerate}
We leave the details of these constructions to the reader.

Define $\iota_m^0: \bsp_m \rightarrow \bsp_m^0$ and $\iota_0^n: \bsp_n \rightarrow \bsp_0^n$ by $\iota_m^0(G)(V,W) = G(V)$ and $\iota_0^n(H)(V,W) = H(W)$. Then $\iota_m^0$ and $\iota_0^n$ induce isomorphisms $\bsp_m \cong \bsp_m^0$ and $\bsp_n \cong \bsp_0^n$, and one can immediately check for $G \in \bsp_m$ and $H \in \bsp_n$ that $G \boxplus H = \iota_m^0(G) \oplus \iota_0^n(H)$ and $G \boxtimes H = \iota_m^0(G) \otimes \iota_0^n(H)$. Set $\bsi_{1,0} = \iota_1^0(\bsi)$ and $\bsi_{0,1} = \iota_0^1(\bsi)$. Then $\iota_m^0(G) = G \circ \bsi_{1,0}$, $\iota_0^n(H) = H \circ \bsi_{0,1}$, and the direct sum functor $\bssigma \in \bibsp$,
\[
\bssigma: (V,W) \mapsto V \oplus W,
\]
is equal to the internal direct sum $\bsi_{1,0} \oplus \bsi_{0,1}$.

\begin{remark}
Given $n \in \N$, define $\bsg^n(\bsv \times \bsv)$ to be the category that is obtained by replacing $\bsv$ with $\bsv \times \bsv$ in the definition of $\bsg^n \bsv$. In particular, if $(V,W),(V',W') \in \bsv \times \bsv$, then
\begin{equation} \label{eq:pairhom}
\Hom_{\bsg^n(\bsv \times \bsv)}((V,W),(V',W')) = \bsg^n[\Hom_k(V,V') \oplus \Hom_k(W,W')].
\end{equation}
The exponential isomorphism for $\bsg$ (discussed next in Section \ref{subsec:Palgebras}) induces an isomorphism
\[
\Hom_{\bsg^n(\bsv \times \bsv)}((V,W),(V',W')) \cong \prod_{d+e=n} \Hom_{\bsg_d^e(\bsv \times \bsv)}((V,W),(V',W'))
\]
that is compatible with the composition of morphisms. Taking $(V,W) = (V',W')$, this yields a decomposition in $\bsg^n(\bsv \times \bsv)$ of the identity morphism $\id_{(V,W)}$ as a sum of commuting orthogonal idempotents. Now let $F: \bsg^n(\bsv \times \bsv) \rightarrow \bsv$ be an even linear functor. Then by reasoning similar to that in Remark \ref{rem:strictstructure}, it follows that the decomposition of $\id_{(V,W)}$ for each $(V,W) \in \bsv \times \bsv$ induces a decomposition $\bigoplus_{d+e=n} F_d^e$ of $F$ with $F_d^e \in \bsp_d^e$. In other words, $F \in \bsp(n)$. Conversely, each $F \in \bsp(n)$ naturally defines an even linear functor $F: \bsg^n(\bsv \times \bsv) \rightarrow \bsv$. Thus, $\bsp(n)$ identifies with the category of even linear functors $F: \bsg^n(\bsv \times \bsv) \rightarrow \bsv$ whose morphisms satisfy \eqref{eq:bihom}.
\end{remark}

\subsection{\texorpdfstring{$\bsp$}{P}-algebras and exponential superfunctors} \label{subsec:Palgebras}

\begin{definition} \textup{(cf.\ \cite[\S3]{Touze:2014})} \label{definition:Palgebra}
A functor $A \in \bsp$ is a \emph{$\bsp$-algebra} if there exist even homomorphisms $k \rightarrow A$ and $m_A: A \otimes A \rightarrow A$ such that for each $V \in \bsv$ the induced maps make $A(V)$ into a $k$-superalgebra. We say that $A$ is \emph{commutative} if each $A(V)$ is then a commutative superalgebra. We say that $A$ is \emph{graded} if there exists a decomposition $A = \bigoplus_{n \in \Z} A^n$ such that $m_A$ restricts for each $m,n \in \Z$ to a homomorphism $A^m \otimes A^n \rightarrow A^{m+n}$, and we say that $A$ is \emph{graded-commutative} if each $A(V)$ is then a graded-commutative graded superalgebra with respect to the grading $A(V) = \bigoplus_{n \in \Z} A^n(V)$. If $A$ and $B$ are (graded) $\bsp$-algebras, then $\eta \in \Hom_{\bsp}(A,B)$ is a (\emph{graded}) \emph{$\bsp$-algebra homomorphism} if each $\eta(V)$ is a homomorphism of (graded) superalgebras.
\end{definition}

The functors $\bss$ and $\bsg$ are examples of commutative $\bsp$-algebras. Any $\bsp$-algebra can be made into a graded $\bsp$-algebra via its \emph{polynomial grading}, i.e., the grading defined by the functor's polynomial decomposition. The polynomial gradings make $\bsl$ and $\bsa$ into graded-commutative $\bsp$-algebras. Now let $A$ and $B$ be graded $\bsp$-algebras. We denote by $A \gotimes B$ the graded $\bsp$-algebra such that for each $V \in \bsv$, $(A \gotimes B)(V)$ is equal to the graded tensor product of superalgebras $A(V) \gotimes B(V)$. Similarly, if $A$ and $B$ are graded $\bsp$-coalgebras, then the graded tensor product $A \gotimes B$ is naturally a graded $\bsp$-coalgebra. (The reader can fill in the coalgebra analogue of Definition \ref{definition:Palgebra}.) If the $\Z$-grading on either $A$ or $B$ is trivial (or purely even), then we may denote $A \gotimes B$ simply as $A \otimes B$.

\begin{definition}
Let $A \in \bsp$. Then $A$ is a \emph{graded $\bsp$-bialgebra} if $A$ is both a graded $\bsp$-algebra and a graded $\bsp$-coalgebra, and if the coproduct $\Delta_A: A \rightarrow A \gotimes A$ is a homomorphism of graded $\bsp$-algebras. (If the grading on $A$ is purely even, then we may drop the adjective \emph{graded}.)
\end{definition}

Given a (graded) $\bsp$-algebra $A$, we can consider the composite bisuperfunctor homomorphism
\begin{equation} \label{eq:mu}
\mu_A : A \boxtimes A = (A \circ \bsi_{1,0}) \otimes (A \circ \bsi_{0,1}) \rightarrow (A \circ \bssigma) \otimes (A \circ \bssigma) = (A \otimes A) \circ \bssigma \rightarrow A \circ \bssigma
\end{equation}
induced by the inclusions $\bsi_{1,0} \hookrightarrow \bssigma$, $\bsi_{0,1} \hookrightarrow \bssigma$ and by the product $m_A: A \otimes A \rightarrow A$.

\begin{definition} \label{definition:exponential}
A graded $\bsp$-algebra $A$ is an \emph{exponential superfunctor} if $\mu_A$ is an isomorphism of strict polynomial bisuperfunctors, i.e., if for each $V,W \in \bsv$, the composite map
\[
A(V) \otimes A(W) \rightarrow A(V \oplus W) \otimes A(V \oplus W) \rightarrow A(V \oplus W),
\]
induced by the inclusions $V \hookrightarrow V \oplus W$ and $W \hookrightarrow V \oplus W$ and by multiplication in $A(V \oplus W)$, lifts to an isomorphism of strict polynomial bisuperfunctors $A \boxtimes A \cong A \circ \bssigma$.
\end{definition}

\begin{example}
The functors $\bss$, $\bsg$, $\bsl$, and $\bsa$ are exponential superfunctors. More generally, if $A$ and $B$ are exponential superfunctors, then so is $A \gotimes B$.
\end{example}

Let $\bsD: \bsv \rightarrow \bsv \times \bsv$ be the diagonal functor $V \mapsto (V,V)$, and let $\bsdelta: \bsi \rightarrow \bsD$ be the natural transformation that lifts the usual diagonal map on vector spaces. Let $A$ be an exponential super\-functor. Composing the isomorphism $A \boxtimes A \cong A \circ \bssigma$ with $\bsD$, we obtain an isomorphism $A \otimes A \cong A \circ \bsD$ of polynomial superfunctors. Define $\Delta_A$ to be the composite homomorphism
\begin{equation} \label{eq:definecoproduct}
\Delta_A: A = A \circ \bsi \stackrel{A \circ \bsdelta}{\longrightarrow} A \circ \bsD \stackrel{(\mu_A \circ \bsD)^{-1}}{\longrightarrow} A \otimes A.
\end{equation}
Then $\Delta_A$ defines a coassociative coproduct on $A$, making $A$ into a (graded) $\bsp$-coalgebra.

Now suppose in addition for each $V,W \in \bsv$ that $\mu_A$ induces an isomorphism of graded super\-algebras $A(V) \gotimes A(W) \cong A(V \oplus W)$. Then $\Delta_A$ is a homomorphism of graded $\bsp$-algebras (because it is the composition of two algebra homomorphisms), and hence endows $A$ with the structure of a graded $\bsp$-bialgebra. In particular, $\bss$ and $\bsg$ are (trivially graded) commutative cocommutative $\bsp$-bialgebras, and $\bsl$ and $\bsa$ are graded-commutative graded-cocommutative graded $\bsp$-bialgebras. Using the descriptions in Section \ref{subsec:examples} for $\bsg$, $\bss$, $\bsl$ and $\bsa$ in terms of $S$, $\Gamma$, and $\Lambda$, one can check that the coproducts restrict to the usual coproducts for $S$, $\Gamma$, and $\Lambda$. Specifically, the coproducts on $S(V)$ and $\Lambda(V)$ are induced (for $v \in V$) by the map $v \mapsto v \otimes 1 + 1 \otimes v$, and the coproduct on $\Gamma(V)$ is induced by the map $\gamma_n(v) \mapsto \sum_{i+j=n} \gamma_i(v) \otimes \gamma_j(v)$; for details see \cite[III.11.1]{Bourbaki:1998a} and \cite[IV.5.7]{Bourbaki:2003}.

\begin{remark} \label{rem:lambda}
Given a graded $\bsp$-coalgebra $A$, there exists a bisuperfunctor homomorphism
\begin{equation} \label{eq:lambda}
\lambda_A : A \circ \bssigma \rightarrow (A \otimes A) \circ \bssigma = (A \circ \bssigma) \otimes (A \circ \bssigma) \rightarrow (A \circ \bsi_{1,0}) \otimes (A \circ \bsi_{0,1}) = A \boxtimes A
\end{equation}
induced by the coproduct $\Delta_A: A \rightarrow A \otimes A$ and the projections $\bssigma \rightarrow \bsi_{1,0}$ and $\bssigma \rightarrow \bsi_{0,1}$. If $A$ is an exponential superfunctor with coproduct defined by \eqref{eq:definecoproduct}, then $\lambda_A = \mu_A^{-1}$. Conversely, suppose $A$ is a graded $\bsp$-coalgebra and $\lambda_A$ is an isomorphism. Let $+: \bsD \rightarrow \bsi$ be the natural transformation defined by $+(v_1,v_2) = v_1+v_2$. Then the composite
\begin{equation} \label{eq:defineproduct}
A \otimes A \stackrel{(\lambda_A \circ \bsD)^{-1}}{\longrightarrow} A \circ \bsD \stackrel{A(+)}{\longrightarrow} A,
\end{equation}
defines an associative product $m_A$ making $A$ into a graded $\bsp$-algebra, and the homomorphism $\mu_A$ defined in terms of \eqref{eq:defineproduct} is equal to $\lambda_A^{-1}$.
\end{remark}

\subsection{Duality isomorphisms} \label{subsec:duality}

As described in Section \ref{subsec:Palgebras}, $\bsg$ is a commutative cocommutative $\bsp$-bialgebra. Then by duality so is $\bsg^\#$. Since $\bsg^1 = \bsi$, we have $(\bsg^\#)^1 = \bsi^\#$. Then the isomorphism $\bsi \cong \bsi^\#$ together with the multiplication map on $\bsg^\#$ defines a natural transformation
\[
\otimes^n = \bsi^{\otimes n} \stackrel{\sim}{\longrightarrow} (\bsi^\#)^{\otimes n} \longrightarrow (\bsg^\#)^n
\]
that factors through $\bss^n$. Summing over all $n \in \N$, it follows that there exists a unique algebra homomorphism $\theta: \bss \rightarrow \bsg^\#$ that extends the identification $\bss^1 = \bsi \cong \bsi^\# = (\bsg^\#)^1$. Now by an adaptation of the arguments in \cite[IV.5.11]{Bourbaki:2003}, one can show that $\theta$ is an isomorphism. Explicitly, let $\set{x_1,\ldots,x_s}$ and $\set{y_1,\ldots,y_t}$ be bases for $\Vzero$ and $\Vone$, and let $\set{x_1^*,\ldots,x_s^*}$ and $\set{y_1^*,\ldots,y_t^*}$ be the dual bases. Then $\theta(V)(x_1^{a_1} \cdots x_s^{a_s} y_1^{b_1} \cdots y_t^{b_t})$ evaluates to either $1$ or $-1$ (depending on the values of $b_1,\ldots,b_t$) on the basis monomial $\gamma_{a_1}(x_1^*) \cdots \gamma_{a_n}(x_n^*)(y_1^*)^{b_1} \cdots (y_t^*)^{b_t}$, and evaluates to $0$ on all other basis monomials in $\bsg(V^*)$ of the form \eqref{eq:bsgbasis}. The isomorphism $\bss \cong \bsg^\#$ is compatible with the coproducts on $\bss$ and $\bsg^\#$, so $\bss \cong \bsg^\#$ as $\bsp$-bialgebras. By duality, $\bss^\# \cong \bsg$ as $\bsp$-bialgebras as well. An entirely parallel argument shows that $\bsl \cong \bsa^\#$ and $\bsa \cong \bsl^\#$ as graded $\bsp$-bialgebras.


\begin{remark} \label{rem:symmetrizationmaps}
There exist unique algebra homomorphisms $\bss \rightarrow \bsg$ and $\bsl \rightarrow \bsa$ extending the identifications $\bss^1 = \bsi = \bsg^1$ and $\bsl^1 = \bsi = \bsa^1$. Given $V \in \bsv$, the induced maps $\bss(\Vone) \rightarrow \bsg(\Vone)$ and $\bsl(\Vzero) \rightarrow \bsa(\Vzero)$ are surjective, hence isomorphisms by dimension comparison. If $k$ is a field of characteristic zero, then the maps $\bss \rightarrow \bsg$ and $\bsl \rightarrow \bsa$ are isomorphisms.
\end{remark}

\subsection{Frobenius twists} \label{subsec:Frobenius}

Given $r \geq 1$ and $V \in \bsv$, write $\varphi: k \rightarrow k$ for the $p^r$-power map $\lambda \mapsto \lambda^{p^r}$, and set $V^{(r)} = k \otimes_{\varphi} V$. In other words, $V^{(r)}$ is equal to $V$ as an additive group, but a scalar $\lambda \in k$ acts on $V^{(r)}$ the way that $\lambda^{p^{-r}}$ acts on $V$. The $r$-th Frobenius twist functor $\bsir \in \bsp_{p^r}$ is defined on objects by $\bsir(V) = V^{(r)}$. To describe the action of $\bsir$ on morphisms, first observe for each $V \in \bsv$ that the $p^r$-power map defines an algebra homomorphism $\bss(V)^{(r)} \rightarrow \bss(V)$ that is natural with respect to even linear maps on $V$. By abuse of notation, we also denote this homomorphism by $\varphi$. If $z = z_{\ol{0}} + z_{\ol{1}}$ is the decomposition of $z$ into its even and odd components, then $z^p = (z_{\ol{0}})^p$, so $\varphi$ has image in the subalgebra $\bss(\Vzero)$ of $\bss(V)$. One can check that if $A,B \in \bsv$, then there exists a commutative diagram in which the horizontal arrows are the natural algebra maps:
\begin{equation} \label{eq:prpowercompatible}
\vcenter{\xymatrix{
\bss(A \otimes B)^{(r)} \ar@{->}[r] \ar@{->}[d]_{\varphi} & \bss(A)^{(r)} \otimes \bss(B)^{(r)} \ar@{->}[d]^{\varphi \otimes \varphi} \\
\bss(A \otimes B) \ar@{->}[r] & \bss(A) \otimes \bss(B).
}}
\end{equation}
By duality, there exists for each $V \in \bsv$ an algebra homomorphism $\varphi^\#: \bsg(V) \rightarrow \bsg(V)^{(r)}$, which we refer to as the \emph{dual Frobenius morphism}, that is also natural with respect to even linear maps on $V$. Explicitly, $\varphi^\#$ acts on the generators for $\bsg(V)$ described in Section \ref{subsubsec:dividedpowers} by
\begin{equation} \label{eq:dualFrobenius}
\varphi^\#(z) = \begin{cases}
0 & \text{if $z \in \Vone \subseteq \bsg^1(V)$}, \\
\gamma_{p^{e-r}}(v) & \text{if $z = \gamma_{p^e}(v)$ for some $e \in \N$, $v \in \Vzero$, and $e \geq r$,} \\
0 & \text{if $z = \gamma_{p^e}(v)$ for some $e \in \N$, $v \in \Vzero$, and $e < r$.}
\end{cases}
\end{equation}
Then $\varphi^\#$ is determined by its restriction to the subalgebra $\bsg(\Vzero)$ of $\bsg(V)$, and has image in the subalgebra $\bsg(\Vzero)^{(r)}$ of $\bsg(V)^{(r)}$. Now given $V,W \in \bsv$, $\bsir_{V,W}$ is the linear map
\begin{equation} \label{eq:IrVW}
\bsg^{p^r} \Hom_k(V,W) \stackrel{\varphi^\#}{\longrightarrow} [\bsg^1 \Hom_k(V,W)_{\ol{0}}]^{(r)} = \Hom_k(V^{(r)},W^{(r)})_{\ol{0}}.
\end{equation}
Dualizing \eqref{eq:prpowercompatible}, and using the fact that $\varphi^\#$ is natural with respect to even linear maps, it follows that $\bsir_{V,W}$ is compatible with the composition of morphisms. From now on, for $r \geq 1$ and $F \in \bsp_n$, set $F^{(r)} = F \circ \bsir \in \bsp_{p^rn}$. Since \eqref{eq:IrVW} has image in the space of even linear maps, it follows for $r \geq 1$ that there exist subfunctors $\bsirzero$ and $\bsirone$ of $\bsir$ such that $\bsirzero(V) = \Vzeror$, $\bsirone(V) = \Voner$, and $\bsir = \bsirzero \oplus \bsirone$. Additionally,
\begin{equation} \label{eq:conjugatebsir}
\bsi_1^{(r)} = \Pi \circ \bsi_0^{(r)} \circ \Pi.
\end{equation}

\begin{lemma}
Let $F \in \bsp_n$, and let $r \geq 1$. Then $F \circ \bsirzero$ and $F \circ \bsirone$ are summands of $F \circ \bsir$.
\end{lemma}

\begin{proof}
We prove that $F \circ \bsirzero$ is a summand of $F \circ \bsir$; the proof for $F \circ \bsirone$ is entirely analogous. Given $V \in \bsv$, let $\iota_V: \Vzero \rightarrow V$ and $\pi_V: V \rightarrow \Vzero$ be the natural (but non-functorial) inclusion and projection maps, respectively. These are even linear maps, and $\pi_V \circ \iota_V$ is the identity on $\Vzero$. Set $\bs{\iota_V} = (\iota_V)^{\otimes p^rn} \in \bsg^{p^rn} \Hom_k(\Vzero,V)$, and set $\bs{\pi_V} = (\pi_V)^{\otimes p^r n} \in \bsg^{p^rn} \Hom_k(V,\Vzero)$. Then the composite $\bs{\pi_V} \circ \bs{\iota_V} \in \bsg^{p^rn}\Hom_k(\Vzero,\Vzero)$ is the identity map on $\Vzero$ in the category $\bsg^{p^rn} \bsv$. Now applying the functor $F \circ \bsir$,  we get maps
\[
(F \circ \bsir)(\Vzero) \stackrel{(F \circ \bsir)(\bs{\iota_V})}{\longrightarrow} (F \circ \bsir)(V) \stackrel{(F \circ \bsir)(\bs{\pi_V})}{\longrightarrow} (F \circ \bsir)(\Vzero)
\]
whose composite is the identity. But $(F \circ \bsir)(\Vzero) = (F \circ \bsirzero)(V)$, so to prove that $F \circ \bsirzero$ is a summand of $F \circ \bsir$, it suffices to show that
\begin{align*}
(F \circ \bsir)(\bs{\iota_V}) &: (F \circ \bsi_0^{(r)})(V) \rightarrow (F \circ \bsir)(V) \quad \text{and} \\
(F \circ \bsir)(\bs{\pi_V}) &: (F \circ \bsir)(V) \rightarrow (F \circ \bsi_0^{(r)})(V)
\end{align*}
lift to natural transformations. We verify this for $(F \circ \bsir)(\bs{\iota_V})$, the argument for $(F \circ \bsir)(\bs{\pi_V})$ being entirely analogous. Let $W \in \bsv$, and let $\phi \in \bsg^{p^r n}\Hom_k(V,W)$. We must show that
\begin{equation} \label{eq:naturality}
(F \circ \bsir)(\phi) \circ (F \circ \bsir)(\bs{\iota_V}) = (F \circ \bsir)(\bs{\iota_W}) \circ (F \circ \bsi_0^{(r)})(\phi).
\end{equation}
We may assume that $\phi \in \bsg^{p^rn}[\Hom_k(V,W)_{\ol{0}}]$, since otherwise \eqref{eq:dualFrobenius} implies that $(F \circ \bsir)(\phi) = 0$ and $(F \circ \bsirzero)(\phi) = 0$. By definition, if $\psi \in \bsg^{p^r} \Hom_k(V,W)$, then
\[
\bsi_0^{(r)}(\psi) = \bsir( (\pi_W)^{\otimes p^r} \circ \psi \circ (\iota_V)^{\otimes p^r}).
\]
Then it follows that $(F \circ \bsirzero)(\phi) = (F \circ \bsir)(\bs{\pi_W} \circ \phi \circ \bs{\iota_V})$. Now since $\phi \in \bsg^{p^rn}[\Hom_k(V,W)_{\ol{0}}]$, it follows that $\phi \circ \bs{\iota_V} = \bs{\iota_W} \circ (\bs{\pi_W} \circ \phi \circ \bs{\iota_V})$, and hence that \eqref{eq:naturality} is true.
\end{proof}

\begin{remark}
In general, if $F \in \bsp_n$ and $r \geq 1$, then $F \circ \bsir \neq (F \circ \bsirzero) \oplus (F \circ \bsirone)$.
\end{remark}

Now let $A \in \bsp$ be an exponential superfunctor. Then
\[
A^{(r)}  = A \circ (\bsirzero \oplus \bsirone) \cong (A \circ \bsirzero) \gotimes (A \circ \bsirone)
\]
as $\bsp$-algebras. This observation applies in particular to $\bss$, $\bsg$, $\bsl$, and $\bsa$. For $n \in \N$, set
\begin{equation} \label{eq:twistcomponents}
\begin{aligned}
S_0^{n(r)} &= \bss^n \circ \bsi_0^{(r)}, &
\Lambda_0^{n(r)} &= \bsa^n \circ \bsi_0^{(r)}, &
\Gamma_0^{n(r)} &= \bsg^n \circ \bsi_0^{(r)}, \\
\Lambda_1^{n(r)} &= \bss^n \circ \bsi_1^{(r)}, &
\Gamma_1^{n(r)} &= \bsa^n \circ \bsi_1^{(r)}, &
S_1^{n(r)} &= \bsl^n \circ \bsi_1^{(r)}.
\end{aligned}
\end{equation}
For each $V \in \bsv$, one has ${S_0^n}^{(r)}(V) = S^n({\Vzero}^{(r)})$ and ${\Lambda_1^n}^{(r)}(V) = \Lambda^n({\Vone}^{(r)})$ as abstract vector spaces. As a superspace, ${S_0^n}^{(r)}(V)$ is always a purely even superspace, while ${\Lambda_1^n}^{(r)}(V)$ is a purely even (resp.\ odd) superspace if $n$ is even (resp.\ odd). Similar identifications and comments apply to the other composite functors defined in \eqref{eq:twistcomponents}.

Suppose for the moment that $F \in \cp_n$ is an ordinary homogeneous strict polynomial functor. Let $r \geq 1$, and let $F^{(r)} \in \cp_{p^rn}$ be the ordinary $r$-th Frobenius twist of $F$. We can lift $F^{(r)}$ to the structure of a homogeneous strict polynomial superfunctor in several different ways. First, since \eqref{eq:IrVW} has image in the space of even linear maps, it follows that $\bsirzero$ and $\bsirone$ define even linear functors $\bsg^{p^rn} \bsv \rightarrow \bsg^n (\bsvzero)$ and $\bsg^{p^rn} \bsv \rightarrow \bsg^n (\bsvone)$, respectively. Then making the natural identifications $\bsvzero \cong \cv$ and $\bsvone \cong \cv$, we can consider $F$ as an even linear functor $\bsg^n(\bsvzero) \rightarrow \bsvzero$ or as an even linear functor $\bsg^n(\bsvone) \rightarrow \bsvone$. Next, the inclusions $\bsvzero \hookrightarrow \bsv$ and $\bsvone \hookrightarrow \bsv$ are even linear functors. We denote the composite even linear functors by
\begin{align}
F \circ \bsi_0^{(r)} &: \bsg^{p^rn}(\bsv) \rightarrow \bsg^n(\bsvzero) \stackrel{F}{\rightarrow} \bsvzero \hookrightarrow \bsv, \quad \text{and} \label{eq:evenlift} \\
F \circ \bsi_1^{(r)} &: \bsg^{p^rn}(\bsv) \rightarrow \bsg^n(\bsvone) \stackrel{F}{\rightarrow} \bsvone \hookrightarrow \bsv. \label{eq:oddlift}
\end{align}
Finally, we can optionally compose \eqref{eq:evenlift} or \eqref{eq:oddlift} with the parity change functor $\Pi$. (Composing with $\bsPi$ has the same result as composing with $\Pi$.) This gives four different ways we can lift $F^{(r)} \in \cp_{p^rn}$ to an element of $\bsp_{p^rn}$. For example, the four liftings of $I^{(r)}$ obtained in this fashion are $\bsirzero$, $\bsirone$, $\Pi \circ \bsirzero$, and $\Pi \circ \bsirone$. More generally, if $\bs{F^{(r)}}$ denotes one of the four liftings of $F^{(r)}$ to $\bsp$ described above, then the other three liftings are $\bs{F^{(r)}} \circ \Pi$, $\Pi \circ \bs{F^{(r)}}$, and $\Pi \circ \bs{F^{(r)}} \circ \Pi$.

One can now show that each of the functors defined in \eqref{eq:twistcomponents} is a lifting of either $S^{n(r)}$, $\Lambda^{n(r)}$, or $\Gamma^{n(r)}$ to the category $\bsp$. In particular, one can use this observation to see that
\begin{equation} \label{eq:conjugateiso}
\left. \begin{aligned}
S_1^{n(r)} &\cong S_0^{n(r)} \circ \Pi \\
\Lambda_1^{n(r)} &\cong \Lambda_0^{n(r)} \circ \Pi \\
\Gamma_1^{n(r)} &\cong \Gamma_0^{n(r)} \circ \Pi
\end{aligned} \right\rbrace \text{ if $n$ is even, and} \qquad
\left. \begin{aligned}
S_1^{n(r)} &\cong \Pi \circ S_0^{n(r)} \circ \Pi \\
\Lambda_1^{n(r)} &\cong \Pi \circ \Lambda_0^{n(r)} \circ \Pi \\
\Gamma_1^{n(r)} &\cong \Pi \circ \Gamma_0^{n(r)} \circ \Pi
\end{aligned} \right\rbrace \text{ if $n$ is odd.} 
\end{equation}
From now on, if $F \in \cp_n$ and $r \geq 1$, then unless stated otherwise we will use \eqref{eq:evenlift} to consider $F^{(r)}$ as a homogeneous strict polynomial superfunctor. Then for all $F \in \cp_n$, one has $F^{(r)}|_{\bsvzero} \cong F^{(r)}$ as ordinary strict polynomial functors. On the other hand, if $F \in \bsp_n$, then using \eqref{eq:evenlift} to lift $F|_{\bsvzero} \in \cp_n$ to $\bsp_n$ does not in general result in a strict polynomial superfunctor that is isomorphic to the original functor $F$.

\section{Cohomology of strict polynomial superfunctors} \label{sec:cohomology}

\subsection{Projectives and injectives} \label{subsec:projectivesinjectives}

Recall that $\bsp_\ev$ is an abelian category. We say that $P \in \bsp$ is \emph{projective} if the functor $\Hom_{\bsp}(P,-): \bsp_\ev \rightarrow \fsvec_\ev$ is exact, and that $Q \in \bsp$ is \emph{injective} if $\Hom_{\bsp}(-,Q): \bsp_\ev \rightarrow \fsvec_\ev$ is exact. Given $F,G \in \bsp$, one has $\Hom_{\bsp}(F,G)_{\ol{0}} = \Hom_{\bsp_\ev}(F,G)$, while the odd isomorphisms in \eqref{eq:paritychangebsphom} restrict to odd isomorphisms
\begin{equation} \label{eq:oddtoevenhom}
\begin{aligned}
\Hom_{\bsp}(F,G)_{\ol{1}} &\simeq \Hom_{\bsp}(F,\bsPi \circ G)_{\ol{0}} = \Hom_{\bsp_\ev}(F,\bsPi \circ G), \text{ and} \\
\Hom_{\bsp}(F,G)_{\ol{1}} &\simeq \Hom_{\bsp}(\bsPi \circ F,G)_{\ol{0}} = \Hom_{\bsp_\ev}(\bsPi \circ F,G)
\end{aligned}
\end{equation}
that are natural with respect to even homomorphisms in either variable. Then it follows that a functor is projective (resp.\ injective) in the above-defined sense if and only if it is projective (resp.\ injective) as an object in the abelian category $\bsp_\ev$.

Given $d \in \N$ and $V \in \bsv$, set $\bsg^{d,V} = \bsg^d \Hom_k(V,-)$ and set $\bss^{d,V} = \bss^d(V \otimes -) \cong (\bsg^{d,V})^\#$. Then by Yoneda's Lemma, there exist for each $F \in \bsp_d$ natural isomorphisms
\begin{equation} \label{eq:Yonedalemma}
\Hom_{\bsp_d}(\bsg^{d,V},F) \cong F(V) \quad  \text{and} \quad \Hom_{\bsp_d}(F,\bss^{d,V}) \cong F^\#(V).
\end{equation}
It follows that $\bsg^{d,V}$ is projective in $\bsp_d$ and $\bss^{d,V}$ is injective in $\bsp_d$ \cite[\S3.5]{Axtell:2013}. Taking $V = k$, $\bsg^d$ is projective in $\bsp_d$ and $\bss^d$ is injective in $\bsp_d$. The next theorem follows from \cite[Proposition A.1]{Axtell:2013} and the proof of \cite[Theorem 4.2]{Axtell:2013}.

\begin{theorem}[Axtell] \label{thm:enoughprojectives}
Let $m,n,d \in \N$ such that $m,n \geq d$, and set $V = k^{m|n} = k^m \oplus \Pi(k^n)$. Then the functor $\bsg^{d,V} \oplus (\Pi \circ \bsg^{d,V})$ is a projective generator in $(\bsp_d)_\ev$, and evaluation on $V$ induces an equivalence of categories between $\bsp_d$ and the category of finite-dimensional left supermodules for the Schur superalgebra $S(m|n,d) := \End_{k\fS_d}((k^{m|n})^{\otimes d}) \cong \End_{\bsg^d \bsv}(V)$.
\end{theorem}

Let $d,e \in \N$. The notions of projectivity and injectivity for objects in $\bsp_d^e$ are defined analogously as for $\bsp$. Let $m,n,r,s \in \N$, and suppose $m,n \geq d$ and $r,s \geq e$. Set $V = k^{m|n}$ and set $W = k^{r|s}$. Using the fact that $V$ and $W$ satisfy the hypotheses of \cite[Proposition A.1]{Axtell:2013} for the categories $\bsg^d \bsv$ and $\bsg^e \bsv$, respectively, one can then apply \cite[Proposition A.1]{Axtell:2013} to prove:

\begin{proposition}
Let $d,e,m,n,r,s \in \N$ such that $m,n \geq d$ and $r,s \geq e$. Set $V = k^{m|n}$, and set $W = k^{r|s}$. Then the functor $(\bsg^{d,V} \boxtimes \bsg^{e,W}) \oplus [\Pi \circ (\bsg^{d,V} \boxtimes \bsg^{e,W})]$ is a projective generator for the category $(\bsp_d^e)_\ev$, and evaluation on $(V,W)$ induces an equivalence of categories between $\bsp_d^e$ and the category of finite-dimensional left-supermodules for the superalgebra
\[
\End_{\bsg_d^e(\bsv \times \bsv)}((V,W)) \cong S(m|n,d) \otimes S(r|s,e).
\]
\end{proposition}

\subsection{Cohomology groups and Yoneda products} \label{subsec:cohomology}

Theorem \ref{thm:enoughprojectives} implies that $\bsp_\ev$ is an abelian category with enough projectives and enough injectives, so we can apply the usual machinery of homological algebra to study extension groups in $\bsp_\ev$.

By construction, extension groups in $\bsp_\ev$ are purely even vector spaces. Our ultimate goal is to use strict polynomial superfunctors to understand the cohomology of finite supergroup schemes, or equivalently, of finite-dimensional cocommutative Hopf super\-algebras. Even the simplest examples in that context indicate that odd cohomology classes are important. For example, if $V$ is a finite-dimensional purely odd superspace, then the ordinary exterior algebra $\Lambda(V)$ is naturally a finite-dimensional cocommutative Hopf superalgebra, and it is a classical result that $\Hbul(\Lambda(V),k) \cong S^\bullet(V^*)$ \cite[2.2(2)]{Priddy:1970}. In particular, $\Hbul(\Lambda(V),k)$ is generated as an algebra by the odd subspace $\opH^1(\Lambda(V),k)_{\ol{1}} \cong V^*$. We thus take the following approach to studying extensions between strict polynomial superfunctors:

\begin{definition} \label{def:ExtP}
Given $F \in \bsp$, define $\Ext_{\bsp}^n(F,-)$ to be the $n$-th right derived functor of
\[
\Hom_{\bsp}(F,-): \bsp_{\ev} \rightarrow \fsvec_{\ev}.
\]
\end{definition}

Taking Definition \ref{def:ExtP} as our starting point, we get from \eqref{eq:oddtoevenhom} odd isomorphisms
\begin{equation} \label{eq:oddtoevenext}
\begin{aligned}
\Ext_{\bsp}^n(F,G)_{\ol{1}} &\simeq \Ext_{\bsp}^n(F,\bsPi \circ G)_{\ol{0}} = \Ext_{\bsp_\ev}^n(F,\bsPi \circ G), \text{ and} \\
\Ext_{\bsp}^n(F,G)_{\ol{1}} &\simeq \Ext_{\bsp}^n(\bsPi \circ F,G)_{\ol{0}} = \Ext_{\bsp_\ev}^n(\bsPi \circ F,G).
\end{aligned}
\end{equation}
Through judicious use of post-composition by the parity change functor $\bsPi$, we can appeal to $\bsp_\ev$ to define or deduce results about extension groups ``in $\bsp$.'' While we could probably complete all of the calculations of this paper while working purely with extension groups in $\bsp_{\ev}$, doing so would be an inconvenience since we'd have to keep track of isomorphisms like those in \eqref{eq:oddtoevenext}. (It turns out that all of the extension classes appearing in our main theorem are purely even, but it was not a priori obvious that this would be the outcome.)

The following discussion provides a framework for constructing the hypercohomology spectral sequences discussed in Section \ref{subsec:hypercohomology}. Let $(C,d^C)$ and $(D,d^D)$ be chain complexes in the category $\bsp_\ev$, and let $n \in \Z$. Define an \emph{even} (resp.\ \emph{odd}) \emph{chain map} $\varphi: C \rightarrow D[n]$ to consist for each $i \in \Z$ of an even (resp.\ odd) homomorphism $\varphi_i: C_{i+n} \rightarrow D_i$ such that $d_i^D \circ \varphi_i = \varphi_{i-1} \circ d_{i+n}^C$.\footnote{Given a graded space $X = \bigoplus_{i \in \Z} X_i$ and given $n \in \Z$, define $X[n]$ to be the graded space with $X[n]_i = X_{n+i}$.} Then an even chain map $\varphi: C \rightarrow D[n]$ is precisely a chain map of degree $-n$ in the category $\bsp_\ev$, while an odd chain map $\varphi: C \rightarrow D[n]$ is equivalent by \eqref{eq:oddtoevenhom} to an even chain map $\varphi': C \rightarrow (\bsPi \circ D)[n]$. Say that two even (resp.\ odd) chain maps $\varphi,\psi: C \rightarrow D[n]$ are \emph{even} (resp.\ \emph{odd}) \emph{homotopic}, and write $\varphi \simeq \psi$, if there exists for each $i \in \Z$ an even (resp.\ odd) homomorphism $\Sigma_i: C_{i+n} \rightarrow D_{i+1}$ such that $\varphi_i - \psi_i = d_{i+1}^D \circ \Sigma_i + \Sigma_{i-1} \circ d_{i+n}^C$. As usual, the property of being even (resp.\ odd) homotopic is an equivalence relation on the set of even (resp.\ odd) chain maps, the composition of two homogeneous chain maps is again a homogeneous chain map, and composition of homogeneous chain maps is compatible with homotopy equivalence.


Now let $F,G,H \in \bsp$, and let $C$ and $D$ be projective resolutions in $\bsp_{\ev}$ of $F$ and $G$, respectively. Then the even (resp.\ odd) subspace of $\Ext_{\bsp}^n(F,G)$ identifies with the vector space of homotopy classes of even (resp.\ odd) chain maps $\varphi: C \rightarrow D[n]$. Identifying homogeneous elements of $\Ext_{\bsp}^m(G,H)$ and $\Ext_{\bsp}^n(F,G)$ with homotopy classes of homogeneous chain maps, it follows that the composition of chain maps induces an even bilinear map
\begin{equation} \label{eq:Yonedaproduct}
\Ext_{\bsp}^m(G,H) \otimes \Ext_{\bsp}^n(F,G) \rightarrow \Ext_{\bsp}^{m+n}(F,H),
\end{equation}
which we call the \emph{Yoneda product of extensions in $\bsp$}. Associativity of \eqref{eq:Yonedaproduct} follows from the fact that composition of chain maps is associative. In particular, if $F \in \bsp$, then $\Ext_{\bsp}^\bullet(F,F)$ has the structure of a graded superalgebra, which we call the \emph{Yoneda algebra of $F$}.

\begin{remark} \label{rem:Yonedareduction}
Using \eqref{eq:oddtoevenext}, the Yoneda product \eqref{eq:Yonedaproduct} can be expressed in terms of the usual Yoneda product in $\bsp_\ev$. For example, the product $\Ext_{\bsp}^m(G,H)_{\ol{0}} \otimes \Ext_{\bsp}^n(F,G)_{\ol{1}} \rightarrow \Ext_{\bsp}^{m+n}(F,H)_{\ol{1}}$ can be computed via the composite map
\begin{align*}
\Ext_{\bsp}^m(G,H)_{\ol{0}} \otimes \Ext_{\bsp}^n(F,G)_{\ol{1}}
&\simrightarrow \Ext_{\bsp_\ev}^m(G,H) \otimes \Ext_{\bsp_\ev}^n(\bsPi \circ F,G) \\
&\rightarrow \Ext_{\bsp_\ev}^{m+n}(\bsPi \circ F, H) \simrightarrow \Ext_{\bsp}^{m+n}(F,H)_{\ol{1}},
\end{align*}
while $\Ext_{\bsp}^m(G,H)_{\ol{1}} \otimes \Ext_{\bsp}^n(F,G)_{\ol{1}} \rightarrow \Ext_{\bsp}^{m+n}(F,H)_{\ol{0}}$ can be computed
via the composite map
\begin{align*}
\Ext_{\bsp}^m(G,H)_{\ol{1}} \otimes \Ext_{\bsp}^n(F,G)_{\ol{1}}
&\simrightarrow \Ext_{\bsp_\ev}^m(\bsPi \circ G, H) \otimes \Ext_{\bsp_\ev}^n(F,\bsPi \circ G) \\
&\rightarrow \Ext_{\bsp_\ev}^{m+n}(F,H) = \Ext_{\bsp}^{m+n}(F,H)_{\ol{0}}.
\end{align*}
(Technically, the second composite computes the desired product multiplied by $-1$, due to the fact that the first arrow is induced by a tensor product of odd maps and due to the super sign conventions for tensor products of maps.)
\end{remark}

\begin{remark} \label{rem:nohomnoext}
If $F \in \bsp_m$, $G \in \bsp_n$, and $m \neq n$, then $\Hom_{\bsp}(F,G) = 0$. From this it immediately follows that $\Ext_{\bsp}^\bullet(F,G) = 0$ whenever $F$ and $G$ are homogeneous of different degrees. In the future we will often apply this observation without further comment.
\end{remark}

We leave it to the reader to formulate bisuperfunctor analogues of the definitions in this section.

\subsection{Operations on cohomology groups} \label{subsec:operations}

\subsubsection{Duality} \label{subsubsec:dualityiso}

Let $Q$ and $R$ be injective resolutions in the category $\bsp_\ev$ of $F$ and $G$, respectively. By duality, the even (resp.\ odd) subspace of $\Ext_{\bsp}^n(F,G)$ identifies with the set of homotopy classes of even (resp.\ odd) chain maps $\psi: Q[-n] \rightarrow R$. Since the duality functor $F \mapsto F^\#$ sends projective objects to injective objects and vice versa, it follows that the operation of sending a homogeneous chain map $\varphi: C \rightarrow D[n]$ to the dual map $\varphi^\#: D[n]^\# \rightarrow C^\#$ induces an isomorphism $\Ext_{\bsp}^n(F,G) \simrightarrow \Ext_{\bsp}^n(G^\#,F^\#)$, which we denote by $z \mapsto z^\#$. Moreover, if $z$ and $w$ are homogeneous elements such that the Yoneda product $z \cdot w$ makes sense, it follows that $(z \cdot w)^\# = (-1)^{\ol{z} \cdot \ol{w}} w^\# \cdot z^\#$, since this holds when composing homogeneous morphisms in $\bsp$. Finally, since $F \cong F^{\#\#}$, it follows that $\Ext_{\bsp}^\bullet(F^{\#\#},G^{\#\#}) \cong \Ext_{\bsp}^\bullet(F,G)$. Thus we can consider the duality functor $F \mapsto F^\#$ as inducing an anti-involution on extension groups in $\bsp$.

\subsubsection{Precomposition} \label{subsubsec:precomposition}

Let $H \in \bsp$. Precomposition with $H$, $F \mapsto F \circ H$, defines an exact even linear endofunctor on the category $\bsp_\ev$. Then for each $F,G \in \bsp$, there exists an induced even linear map $\Ext_{\bsp_\ev}^\bullet(F,G) \rightarrow \Ext_{\bsp_\ev}^\bullet(F \circ H, G \circ H)$ that is compatible with the Yoneda product in $\bsp_\ev$. Since precomposition by $H$ commutes with the isomorphisms in \eqref{eq:paritychangebsphom}, we can use \eqref{eq:oddtoevenext} together with Remark \ref{rem:Yonedareduction}, to deduce that precomposition by $H$ lifts to an even linear map $\Ext_{\bsp}^\bullet(F,G) \rightarrow \Ext_{\bsp}^\bullet(F \circ H, G \circ H)$ that is compatible with the Yoneda product \eqref{eq:Yonedaproduct}.

\subsubsection{Conjugation by \texorpdfstring{$\Pi$}{Pi}} \label{subsubsec:conjugationbyPi}

The operation $F \mapsto \Pi \circ F$ of postcomposition with $\Pi$ preserves projective resolutions in $\bsp_\ev$, sends homogeneous chain maps to homogeneous chain maps of the same parity, and is compatible with the composition of homomorphisms. Then postcomposition with $\Pi$ induces for each $F,G \in \bsp$ an even isomorphism $\Ext_{\bsp}^\bullet(F,G) \simrightarrow \Ext_{\bsp}^\bullet(\Pi \circ F, \Pi \circ G)$ that is compatible with the Yoneda product \eqref{eq:Yonedaproduct}. Now given $F \in \bsp$, set $F^\Pi = \Pi \circ F \circ \Pi$. We refer to the operation $F \mapsto F^\Pi$ as \emph{conjugation by $\Pi$}. Combining the comments in this section with those of Section \ref{subsubsec:precomposition}, it follows for each $F,G \in \bsp$ that there exists an even isomorphism $\Ext_{\bsp}^\bullet(F,G) \simrightarrow \Ext_{\bsp}^\bullet(F^\Pi,G^\Pi)$ that is compatible with \eqref{eq:Yonedaproduct}. We denote this map by $z \mapsto z^\Pi$, and refer to it as the conjugation action of $\Pi$ on extension groups in $\bsp$. Since $\Pi \circ \Pi = \bsi$, then $(z^\Pi)^\Pi = z$.

\subsection{Cup products and coproducts} \label{subsec:externalproducts}

Let $V,W \in \bsv$, and let $d,e \in \N$. The exponential property for $\bsg$ implies that $\bsg^{d,V} \otimes \bsg^{e,W}$ is isomorphic to a direct summand of $\bsg^{d+e,V \oplus W} \in \bsp_{d+e}$. Then it follows from Theorem \ref{thm:enoughprojectives} and the K\"{u}nneth formula that if $P$ and $Q$ are projective resolutions in $\bsp_\ev$ of $A$ and $C$, respectively, then the tensor product of complexes $P \otimes Q$ is a projective resolution in $\bsp_\ev$ of $A \otimes C$. Similarly, the external tensor product $P \boxtimes Q$ is a projective resolution in $(\bibsp)_\ev$ of $A \boxtimes C$. Now given $B,D \in \bsp$, there exist well-defined even linear maps
\begin{align}
\Ext_{\bsp}^m(A,B) \otimes \Ext_{\bsp}^n(C,D) &\rightarrow \Ext_{\bsp}^{m+n}(A \otimes C, B \otimes D) \quad \text{and} \label{eq:externalbspproduct} \\
\Ext_{\bsp}^m(A,B) \otimes \Ext_{\bsp}^n(C,D) &\rightarrow \Ext_{\bibsp}^{m+n}(A \boxtimes C, B \boxtimes D) \label{eq:externalbibspproduct}
\end{align}
induced by sending cocycles $\varphi: P \rightarrow B$ and $\psi: Q \rightarrow D$ to the tensor product of maps $\varphi \otimes \psi: P \otimes Q \rightarrow B \otimes D$ or to the external tensor product of maps $\varphi \boxtimes \psi: P \boxtimes Q \rightarrow B \boxtimes D$, respectively. The same argument as for \cite[Proposition 3.6]{Suslin:1997} shows that \eqref{eq:externalbibspproduct} induces an isomorphism
\begin{equation} \label{eq:externalproductextiso}
\kappa: \Ext_{\bsp}^\bullet(A,B) \otimes \Ext_{\bsp}^\bullet(C,D) \simrightarrow \Ext_{\bibsp}^\bullet(A \boxtimes C, B \boxtimes D).
\end{equation}

Suppose $A$ is a $\bsp$-coalgebra and $B$ is a $\bsp$-algebra. Then there exist even bilinear maps
\begin{align}
\Ext_{\bsp}^m(A,C) \otimes \Ext_{\bsp}^n(A,D) &\rightarrow \Ext_{\bsp}^{m+n}(A,C \otimes D), \label{eq:Acupproduct} \\
\Ext_{\bsp}^m(C,B) \otimes \Ext_{\bsp}^n(D,B) &\rightarrow \Ext_{\bsp}^{m+n}(C \otimes D,B), \quad \text{and} \label{eq:Bcupproduct} \\
\Ext_{\bsp}^m(A,B) \otimes \Ext_{\bsp}^n(A,B) &\rightarrow \Ext_{\bsp}^{m+n}(A,B)  \label{eq:ABcupproduct}
\end{align}
that arise in the usual fashion from \eqref{eq:externalbspproduct} by composing with the maps in cohomology induced by the coproduct $\Delta_A: A \rightarrow A \otimes A$, the product $m_B: B \otimes B \rightarrow B$, or both. We refer to these maps as the \emph{cup products} of the corresponding $\Ext$-groups.

Given a graded $\bsp$-algebra $B = \bigoplus_{n \in \Z} B^n$, consider the diagram
\[
\xymatrix{
B^i \otimes B^j \ar@{->}[r]^-{m_B} \ar@{->}[d]^-{T} & B^{i+j} \ar@{=}[d] \\
B^j \otimes B^i \ar@{->}[r]^-{m_B} & B^{i+j}
}
\]
in which the left-hand vertical arrow is induced by the supertwist map. Set $\ve(B) = 0$ if for each $i,j \in \Z$ the above diagram commutes, and set $\ve(B)=1$ if for each $i,j \in \Z$ the above diagram commutes up to the sign $(-1)^{ij}$, i.e., if the grading makes $B$ into a graded-commutative $\bsp$-algebra. If either of these conditions is satisfied, say that $B$ is $\ve(B)$-commutative. Similarly, one defines the notion of $\ve(A)$-cocommutativity for a graded $\bsp$-coalgebra $A$. Now the next lemma follows from essentially the same ``straightforward (but tiresome)'' reasoning as \cite[Lemma 1.11]{Franjou:1999}.

\begin{lemma} \label{lem:cupcommute}
Let $A = \bigoplus_{n \in \Z} A^n$ be a graded $\bsp$-coalgebra, and let $B = \bigoplus_{n \in \Z} B^n$ be a graded $\bsp$-algebra. Consider the diagram
\begin{equation} \label{eq:cupcommutativity}
\vcenter{\xymatrix{
\Ext_{\bsp}^s(A^i,B^i) \otimes \Ext_{\bsp}^t(A^j,B^j) \ar@{->}[r] \ar@{->}[d]^-{T} & \Ext_{\bsp}^{s+t}(A^{i+j},B^{i+j}) \ar@{=}[d] \\
\Ext_{\bsp}^t(A^j,B^j) \otimes \Ext_{\bsp}^s(A^i,B^i) \ar@{->}[r] & \Ext_{\bsp}^{s+t}(A^{i+j},B^{i+j})
}}
\end{equation}
in which the horizontal arrows are the corresponding cup products, and the left-hand vertical arrow is the supertwist map. If $A$ is $\ve(A)$-cocommutative and if $B$ is $\ve(B)$-commutative, then the diagram commutes up to the sign $(-1)^{st + \ve(A) \cdot ij + \ve(B) \cdot ij}$.
\end{lemma}

Recall that the diagonal functor $\bsD: \bsv \rightarrow \bsv \times \bsv$ and the direct sum functor $\bssigma: \bsv \times \bsv \rightarrow \bsv$ are adjoint functors (on both sides). Using the definition \eqref{eq:pairhom} of homomorphisms in the category $\bsg^n(\bsv \times \bsv)$, it follows that $\bsD$ and $\bssigma$ extend to a pair of adjoint functors $\bsg^n(\bsv) \rightarrow \bsg^n(\bsv \times \bsv)$ and $\bsg^n(\bsv \times \bsv) \rightarrow \bsg^n(\bsv)$, which we also denote $\bsD$ and $\bssigma$. Then precomposition by $\bsD$ and $\bssigma$ defines a pair of adjoint functors $\bsp_n \rightarrow \bsp(n)$, $F \mapsto F \circ \bssigma$, and $\bsp(n) \rightarrow \bsp_n$, $F \mapsto F \circ \bsD$. Extending componentwise, we get a pair of exact adjoint functors $\bsp \rightarrow \bibsp$, $F \mapsto F \circ \bssigma$, and $\bibsp \rightarrow \bsp$, $G \mapsto G \circ \bsD$. Then for each $F \in \bsp$ and $G \in \bibsp$, one gets an isomorphism
\begin{equation} \label{eq:adjunctioniso}
\alpha: \Ext_{\bibsp}^\bullet(F \circ \bssigma,G) \simrightarrow \Ext_{\bsp}^\bullet(F,G \circ \bsD)
\end{equation}
that is natural with respect to even homomorphisms in $F$ or $G$. Now the next theorem follows by the same reasoning as its classical analogue; cf.\ \cite[\S1.7]{Franjou:1999} and \cite[\S\S5.3--5.4]{Touze:2010b}, and also \cite[\S4.4]{Kuhn:1995}.

\begin{theorem} \label{thm:exponentialext}
Let $A \in \bsp$ be an exponential superfunctor, let $C \in \bsp_m$, and let $D \in \bsp_n$. Write $\bigoplus_{n \in \N} A^n$ for the polynomial decomposition of $A$. Then for each $m,n \in \N$, the cup products \eqref{eq:Acupproduct} and \eqref{eq:Bcupproduct} induce isomorphisms
\begin{align}
\Ext_{\bsp}^\bullet(A^m,C) \otimes \Ext_{\bsp}^\bullet(A^n,D) &\cong \Ext_{\bsp}^\bullet(A^{m+n},C \otimes D), \quad \text{and} \label{eq:firstcupiso} \\
\Ext_{\bsp}^\bullet(C,A^m) \otimes \Ext_{\bsp}^\bullet(D,A^n) &\cong \Ext_{\bsp}^\bullet(C \otimes D,A^{m+n}). \label{eq:secondcupiso}
\end{align}
\end{theorem}

Let $F \in \bsp$. We say that $F$ is \emph{additive} if the external direct sum $F \boxplus F \in \bibsp$ is isomorphic as a strict polynomial bisuperfunctor to $F \circ \bssigma$, i.e., if for each $V,W \in \bsv$ there exists a bifunctorial isomorphism $F(V) \oplus F(W) \cong F(V \oplus W)$. For example, if $A$ is an exponential superfunctor with polynomial decomposition $\bigoplus_{n \in \N} A^n$, $A^0 = k$, and $n$ is the least positive integer such that $A^n \neq 0$, then $A^n$ is additive. The next theorem, a variant of a vanishing theorem originally due to Pirashvili, is thus closely related to Theorem \ref{thm:exponentialext}. Its proof follows from a repetition of the proof of \cite[Theorem 2.13]{Friedlander:1997}, after first applying \eqref{eq:oddtoevenext} to reduce to extension groups in the abelian category $\bsp_{\ev}$. (For historical context, see Kuhn \cite[Remark 6.4]{Kuhn:1998}.)

\begin{theorem} \label{thm:vanishing}
Let $T$ and $T'$ be homogeneous strict polynomial superfunctors of positive degrees, and let $F \in \bsp$ be an additive functor. Then $\Ext_{\bsp}^\bullet(F,T \otimes T')=0$.
\end{theorem}

Finally, let $A$ be a $\bsp$-algebra and let $B$ be a $\bsp$-coalgebra. Then the coproduct
\begin{equation} \label{eq:ABcoproduct}
\Ext_{\bsp}^\bullet(A,B) \rightarrow \Ext_{\bsp}^\bullet(A,B) \otimes \Ext_{\bsp}^\bullet(A,B)
\end{equation}
is defined in terms of the coproduct $\Delta_B$, the isomorphism $\alpha$ of \eqref{eq:adjunctioniso}, the homomorphism $\mu_A$ of \eqref{eq:mu}, and the isomorphism $\kappa$ of \eqref{eq:externalproductextiso} as the composite linear map
\begin{equation} \label{eq:extcoproduct}
\begin{gathered}
\Ext_{\bsp}^\bullet(A,B)
\stackrel{\Delta_{B*}}{\longrightarrow} \Ext_{\bsp}^\bullet(A, B \otimes B)
\stackrel{\alpha^{-1}}{\longrightarrow} \Ext_{\bibsp}^\bullet(A \circ \bssigma, B \boxtimes B) \\
\stackrel{\mu_A^*}{\longrightarrow} \Ext_{\bibsp}^\bullet(A \boxtimes A, B \boxtimes B)
\stackrel{\kappa^{-1}}{\longrightarrow} \Ext_{\bsp}^\bullet(A,B) \otimes \Ext_{\bsp}^\bullet(A,B).
\end{gathered}
\end{equation}
The reader can formulate the coproduct analogue of Lemma \ref{lem:cupcommute} by replacing the horizontal arrows in \eqref{eq:cupcommutativity} with the corresponding (left-facing) coproducts. If $A$ is an exponential superfunctor, then the composition $\kappa^{-1} \circ \mu_A^* \circ \alpha^{-1}$ in \eqref{eq:extcoproduct} is the inverse of the cup product isomorphism \eqref{eq:firstcupiso}; cf.\ the first paragraph of Remark \ref{rem:lambda}.

\subsection{\texorpdfstring{$\Ext^n$ and $n$-extensions}{Extn and n-extensions}} \label{subsec:extnextensions}

Recall that, while $\bsp$ is not an abelian category, it is closed under kernels and cokernels of homogeneous morphisms. Then it makes sense to consider exact sequences in $\bsp$ in which each morphism is homogeneous. Given $F,G \in \bsp$, define a \emph{homogeneous $n$-extension of $F$ by $G$ in $\bsp$} to be an exact sequence
\begin{equation} \label{eq:nextension}
E: 0 \rightarrow G = E_{n+1} \rightarrow E_n \rightarrow \cdots \rightarrow E_1 \rightarrow E_0 = F \rightarrow 0
\end{equation}
in $\bsp$ in which each morphism is homogeneous. Define the parity of $E$ to be the sum of the parities of the morphisms appearing in $E$. Say that two $n$-extensions $E$ and $E'$ of $F$ by $G$ in $\bsp$ satisfy the relation $E \leadsto E'$ if there exists a commutative diagram
\begin{equation}
\vcenter{\xymatrix{
0 \ar@{->}[r] & G \ar@{->}[r] \ar@{=}[d] & E_n \ar@{->}[r] \ar@{->}[d] & \cdots \ar@{->}[r] & E_1 \ar@{->}[r] \ar@{->}[d] & F \ar@{->}[r] \ar@{=}[d] & 0 \\
0 \ar@{->}[r] & G \ar@{->}[r] & E_n' \ar@{->}[r] & \cdots \ar@{->}[r] & E_1' \ar@{->}[r] & F \ar@{->}[r] & 0
}}
\end{equation}
in which each vertical arrow is a homogeneous morphism in $\bsp$. Then the relation $E \leadsto E'$ generates an equivalence relation $E \sim E'$ on the set of homogeneous $n$-extensions of $F$ by $G$.

One can show that if $E \leadsto E'$, and hence if $E \sim E'$, then $E$ and $E'$ must be of the same parity. Define $\Yext_{\bsp}^n(F,G)_{\ol{0}}$ (resp.\ $\Yext_{\bsp}^n(F,G)_{\ol{1}}$) to be the set of equivalence classes of even (resp.\ odd) homogeneous $n$-extensions of $F$ by $G$ in $\bsp$. One can also check that the operation of composing (i.e., splicing) homogeneous extensions induces a well-defined product between equivalence classes. We call this product the \emph{composition of homogeneous extensions in $\bsp$}.

\begin{proposition} \label{prop:extnandnextensions}
For each $F,G \in \bsp$ and each $n \in \N$, there exist bijections
\begin{align*}
\theta_0: &\Ext_{\bsp}^n(F,G)_{\ol{0}} \rightarrow \Yext_{\bsp}^n(F,G)_{\ol{0}}, \quad \text{and} \\
\theta_1: &\Ext_{\bsp}^n(F,G)_{\ol{1}} \rightarrow \Yext_{\bsp}^n(F,G)_{\ol{1}}
\end{align*}
under which the Yoneda product \eqref{eq:Yonedaproduct} of homogeneous elements corresponds to the composition of homogeneous extensions in $\bsp$.
\end{proposition}

\begin{proof}
First we construct $\theta_0$. Let $\Yext_{\bsp_\ev}^n(F,G)$ be the set of equivalence classes of $n$-extensions of $F$ by $G$ in the category $\bsp_\ev$. Since $\bsp_\ev$ is an abelian category, it is well-known that there exists a bijection $\theta': \Ext_{\bsp_\ev}^n(F,G) \rightarrow \Yext_{\bsp_\ev}^n(F,G)$ under which the Yoneda product for $\Ext_{\bsp_\ev}^n(F,G)$ corresponds to the composition of extensions for $\Yext_{\bsp_\ev}^n(F,G)$. Next, the inclusion of categories $\bsp_\ev \hookrightarrow \bsp$ induces a function $\theta'': \Yext_{\bsp_\ev}^n(F,G) \rightarrow \Yext_{\bsp}^n(F,G)_{\ol{0}}$ that is compatible with the composition of extensions. We claim that $\theta''$ is a bijection. Assuming this, the composite
\[
\Ext_{\bsp}^n(F,G)_{\ol{0}} = \Ext_{\bsp_\ev}^n(F,G) \stackrel{\theta'}{\rightarrow} \Yext_{\bsp_\ev}^n(F,G) \stackrel{\theta''}{\rightarrow} \Yext_{\bsp}^n(F,G)_{\ol{0}}
\]
then provides the desired bijection $\theta_0$.

To see that $\theta''$ is a surjection, let $[E] \in \Yext_{\bsp}^n(F,G)_{\ol{0}}$, and write $E$ as in \eqref{eq:nextension}. If each morphism appearing in $E$ is even, then $[E]$ is in the image of $\theta''$, and we are done. Otherwise, let $i$ be the least index such that the morphism $E_{i+1} \rightarrow E_i$ is odd. Then one can show that $E \leadsto E'$, where $E'$ is obtained from $E$ by replacing $E_{i+1}$ by $\bsPi \circ E_{i+1}$, and replacing the morphisms $E_{i+1} \rightarrow E_i$ and $E_{i+2} \rightarrow E_{i+1}$ by the composites $\bsPi \circ E_{i+1} \simeq E_{i+1} \rightarrow E_i$ and $E_{i+2} \rightarrow E_{i+1} \simeq \bsPi \circ E_{i+1}$, respectively. Now arguing by induction on $i$, and using the assumption that $E$ was even, it follows that $E$ is equivalent to an $n$-extension in which each morphism is even, and hence that $[E] \in \im(\theta'')$. To see that $\theta''$ is an injection, let $\theta''': \Yext_{\bsp}^n(F,G)_{\ol{0}} \rightarrow \Ext_{\bsp}^n(F,G)_{\ol{0}} = \Ext_{\bsp_\ev}^n(F,G)$ be defined by precisely the same procedure as in the first half of the proof of \cite[Theorem IV.9.1]{Hilton:1997}. Then the composite $\theta''' \circ \theta'' \circ \theta'$ is the identity map, which proves that $\theta''$ must be one-to-one.

Define $\theta_1$ to be the composite function
\[
\Ext_{\bsp}^n(F,G)_{\ol{1}} \simeq \Ext_{\bsp}^n(F,\bsPi \circ G)_{\ol{0}} \stackrel{\theta_0}{\longrightarrow} \Yext_{\bsp}^n(F,\bsPi \circ G)_{\ol{0}} \simeq \Yext_{\bsp}^n(F,G)_{\ol{1}},
\]
where the last isomorphism is the obvious parity-reversing bijection induced by the odd isomorphism $G \simeq \bsPi \circ G$. Then $\theta_1$ is a bijection, and using Remark \ref{rem:Yonedareduction} one can then check that $\theta_0$ and $\theta_1$ are compatible with the respective products.
\end{proof}

\subsection{Hypercohomology} \label{subsec:hypercohomology}

Since $\Hom_{\bsp}(-,-)$ defines a bifunctor $\bsp_\ev \rightarrow \fsvec_\ev$, contravariant in the first variable and covariant in the second, we can consider for each chain complex $A$ in $\bsp_\ev$ and each cochain complex $C$ in $\bsp_\ev$ the $n$-th hypercohomology group $\bbExt_{\bsp}^n(A,C)$ of $\Hom_{\bsp}(-,-)$, as defined for example in \cite[XVII.2]{Cartan:1999}. In particular, we will make extensive use of the fact that if $A$ is an object in $\bsp$ considered as a chain complex concentrated in degree zero, then the two hypercohomology spectral sequences converging to $\bbExt_{\bsp}^\bullet(A,C)$ are each right modules over the Yoneda algebra $\Ext_{\bsp}^\bullet(A,A)$. For the reader's convenience we briefly describe some of the details behind this construction in this special case.

Let $A \in \bsp$, and let $C$ be a (non-negative) cochain complex in $\bsp_\ev$. Let $P = P_\bullet$ be a projective resolution in $\bsp_\ev$ of $A$, and let $Q = Q^{\bullet,\bullet}$ be an injective Cartan-Eilenberg resolution in $\bsp_\ev$ of $C$. For $i,j \in \N$, set $\Hom_{\bsp}(P,Q)^{i,j} = \bigoplus_{r+s=j} \Hom_{\bsp}(P_r,Q^{i,s})$. This indexing induces on $\Hom_{\bsp}(P,Q)$ the structure of a first quadrant double complex in which the horizontal differential is induced by the horizontal differential from $Q$, and in which the vertical differential is induced by the differential from $P$ and the vertical differential from $Q$. Now $\bbExt_{\bsp}^n(A,C)$ is the $n$-th cohomology group of the double complex $\Hom_{\bsp}(P,Q)$. Interpreting homogeneous elements of $\Ext_{\bsp}^m(A,A)$ as homotopy classes of homogeneous chain maps $P \rightarrow P[m]$, the right action of $\Ext_{\bsp}^m(A,A)$ on $\bbExt_{\bsp}^n(A,C)$,
\[
\bbExt_{\bsp}^n(A,C) \otimes \Ext_{\bsp}^m(A,A) \rightarrow \bbExt_{\bsp}^{m+n}(A,C),
\]
is then induced by the composition of homomorphisms
\[
\Hom_{\bsp}(P,Q) \otimes \Hom_{\bsp}(P,P) \rightarrow \Hom_{\bsp}(P,Q).
\]

The first quadrant double complex $\Hom_{\bsp}(P,Q)$ gives rise to two spectral sequences that each converge to $\bbExt_{\bsp}^\bullet(A,C)$. The \emph{first hypercohomology spectral sequence} arises from computing the cohomology of the double complex first along columns, and takes the form
\begin{equation} \label{eq:firsthypercohomology}
{}^IE_1^{s,t} = \Ext_{\bsp}^t(A,C^s) \Rightarrow \bbExt_{\bsp}^{s+t}(A,C),
\end{equation}
Then the differential $d_1: {}^IE_1^{s,t} \rightarrow {}^IE_1^{s+1,t}$ identifies with the map in cohomology induced by the differential $C^s \rightarrow C^{s+1}$ of the complex $C$. The \emph{second hypercohomology spectral sequence} arises from computing the cohomology of the double complex first along rows, and takes the form
\begin{equation} \label{eq:secondhypercohomology}
{}^{II}E_2^{s,t} = \Ext_{\bsp}^s(A,\opH^t(C)) \Rightarrow \bbExt_{\bsp}^{s+t}(A,C).
\end{equation}
The two spectral sequences are related by the composite map 
\begin{equation} \label{eq:secondfirstcomposite}
{}^{II}E_2^{s,0} \twoheadrightarrow {}^{II}E_{\infty}^{s,0} \hookrightarrow \bbExt_{\bsp}^s(A,C) \twoheadrightarrow {}^I E_\infty^{0,s} \hookrightarrow {}^I E_1^{0,s},
\end{equation}
which identifies with the map in cohomology $\Ext_{\bsp}^s(A,\opH^0(C)) \rightarrow \Ext_{\bsp}^s(A,C^0)$ induced by the inclusion $\opH^0(C) \hookrightarrow C^0$; cf.\ \cite[XVII.3]{Cartan:1999}, though what we index as the first spectral sequence is indexed there as the second and vice versa. The reader can check that the filtrations on $\Hom_{\bsp}(P,Q)$ that give rise to \eqref{eq:firsthypercohomology} and \eqref{eq:secondhypercohomology} are compatible with the right action of $\Hom_{\bsp}(P,P)$, and thus \eqref{eq:firsthypercohomology} and \eqref{eq:secondhypercohomology} become spectral sequences of right $\Ext_{\bsp}^\bullet(A,A)$-modules. In particular, the right action on the $E_1$-page of \eqref{eq:firsthypercohomology} and the right action on the $E_2$-page of \eqref{eq:secondhypercohomology} identify with the corresponding Yoneda products defined in Section \ref{subsec:cohomology}.

\begin{remark} \label{rem:restriction}
By general abstract nonsense, restriction from $\bsp$ to $\cp$ extends for each $F,G \in \bsp$ to a linear map $\Ext_{\bsp}^\bullet(F,G) \rightarrow \Ext_{\cp}^\bullet(F|_{\bsvzero},G|_{\bsvzero})$ that is compatible with Yoneda products. More generally, if $A \in \bsp$ and if $C$ is a cochain complex in $\bsp_{\ev}$, then restriction from $\bsp$ to $\cp$ induces a linear map $\bbExt_{\bsp}^\bullet(A,C) \rightarrow \bbExt_{\cp}^\bullet(A|_{\bsvzero},C|_{\bsvzero})$ on hypercohomology groups.
\end{remark}

\section{The Yoneda algebra \texorpdfstring{$\Ext_{\bsp}^\bullet(\bsir,\bsir)$}{Ext(Ir,Ir)}} \label{sec:yonedaalgebra}

Our goal in this section is to describe the structure of the Yoneda algebra $\Ext_{\bsp}^\bullet(\bsir,\bsir)$. Since $\bsir = \bsirzero \oplus \bsirone$, it follows that $\Ext_{\bsp}^\bullet(\bsir,\bsir)$ is isomorphic as an algebra to the matrix ring \eqref{eq:matrixring}. Thus, it suffices to describe each of the components of the matrix ring and to describe the possible products between them. Our strategy is based on the inductive approach in \cite{Friedlander:1997} (in turn based on that in \cite{Franjou:1994}) using hypercohomology spectral sequences.

\subsection{The super de Rham complex}

Set $\bso = \bss \otimes \bsa$, and recall from Section \ref{subsec:Palgebras} that $\bso$ inherits from $\bss$ and $\bsa$ the structure of a $\bsp$-algebra. Given $i,n \in \N$ with $i \leq n$, define $\bso_n^i$ be the subfunctor $\bss^{n-i} \otimes \bsa^i$ of $\bso$. Following \cite[\S4]{Friedlander:1997}, we call $i$ the cohomological degree and $n$ the total degree of $\bso_n^i$. Now set $\bso_n = \bigoplus_{i=0}^n \bso_n^i$ and $\bso^i = \bigoplus_{n=0}^i \bso_n^i$. Then the cohomological grading $\bso = \bigoplus_{i \in \N} \bso^i$ makes $\bso$ into a graded-commutative graded $\bsp$-algebra.

The product and coproduct maps\footnote{Recall the discussion immediately preceding Remark \ref{rem:lambda}.} on $\bss$ and $\bsa$ induce natural transformations
\begin{align*}
\bsd &: \bso_n^i = \bss^{n-i} \otimes \bsa^i \rightarrow \bss^{n-i-1} \otimes \bsi \otimes \bsa^i \rightarrow \bss^{n-i-1} \otimes \bsa^{i+1} = \bso_n^{i+1}, \quad \text{and} \\
\bsk &: \bso_n^i = \bss^{n-i} \otimes \bsa^i \rightarrow \bss^{n-i} \otimes \bsi \otimes \bsa^{i-1} \rightarrow \bss^{n-i+1} \otimes \bsa^{i-1} = \bso_n^{i-1}.
\end{align*}
One can check for each $V \in \bsv$ that $\bsd(V)$ and $\bsk(V)$ each make $\bso(V)$ into a differential graded superalgebra, and hence that $\bsd$ and $\bsk$ are differentials. We refer to the resulting complexes $(\bso,\bsd)$ and $\bs{Kz}:=(\bso,\bsk)$ as the \emph{super de Rham complex} and the \emph{super Koszul complex}. On $\bsvzero$, $(\bso,\bsd)$ and $(\bso,\bsk)$ restrict to the ordinary de Rham complex $(\Omega,d)$ and the ordinary Koszul complex $Kz := (\Omega,\kappa)$, respectively; cf.\ \cite{Franjou:1994,Friedlander:1997}. On $\bsvone$, $(\bso,\bsd)$ and $(\bso,\bsk)$ restrict to the dual de Rham complex $(\Omega^\#,d^\#)$ and the dual Koszul complex $Kz^\# = (\Omega^\#,\kappa^\#)$. Specifically, $\bso_n^i|_{\bsvone} \cong (\Omega_n^{n-i})^\#$. If $V$ is an abelian Lie superalgebra, then $(\bso(V),\bsk(V))$ identifies with the Koszul resolution of $V$ as studied in \cite[\S3.1]{Drupieski:2013b}. In particular, $(\bso,\bsk)$ is acyclic.

\begin{lemma} \label{lem:homotopy}
On $\bso_n$, the transformation $\bsd \circ \bsk + \bsk \circ \bsd$ acts as multiplication by $n$.
\end{lemma}

\begin{proof}
Let $V \in \bsv$, and set $f = \bsd(V) \circ \bsk(V) + \bsk(V) \circ \bsd(V)$. Since $\bsd(V)$ and $\bsk(V)$ each make $\bso(V)$ into a differential graded algebra, it follows that $f$ acts as an ordinary algebra derivation on $\bso(V)$, i.e., $f(ab) = f(a) \cdot b + a \cdot f(b)$ for all $a,b \in \bso(V)$. Then it suffices to consider the action of $f$ on a set of algebra generators for $\bso(V) = \bss(V) \otimes \bsa(V)$, where the conclusion is easily verified. 
\end{proof}

Since the de Rham differential $\bsd$ is a derivation, it follows that the cohomology $\Hbul(\bso)$ of $\bso$ with respect to $\bsd$ inherits the structure of a graded-commutative graded $\bsp$-algebra. Since $\bsd$ respects the total degree, one gets $\Hbul(\bso) = \bigoplus_{n \in \N} \Hbul(\bso_n)$, and Lemma \ref{lem:homotopy} implies that $\Hbul(\bso_n) \neq 0$ only if $p \mid n$. The next theorem is a super analogue of the Cartier isomorphism; see \cite{Cartier:1957} or \cite[Theorem 4.1]{Friedlander:1997}. Unlike its ordinary analgoue, the super Cartier isomorphism does not preserve the cohomological degree.

\begin{theorem} \label{thm:Cartieriso}
There exists a $\bsp$-algebra isomorphism $\theta: \bso^{(1)} \simrightarrow \Hbul(\bso)$, which we call the super Cartier isomorphism, that restricts for each $n \in \N$ to an isomorphism ${\bso_n}^{(1)} \cong \Hbul(\bso_{pn})$.
\end{theorem}

\begin{proof}
First we show for each $V \in \bsv$ that $\bso(V^{(1)}) \cong \Hbul(\bso(V))$ as superalgebras. Then we show that this family of isomorphisms lifts to an isomorphism of strict polynomial superfunctors. Recall from Section \ref{subsec:Palgebras} that $\bso$ is an exponential superfunctor. For $V,W \in \bsv$, the exponential isomorphism $\bso(V \oplus W) \cong \bso(V) \otimes \bso(W)$ defines an isomorphism of complexes between $\bso(V \oplus W)$ and the tensor product of complexes $\bso(V) \otimes \bso(W)$. Then to prove that $\bso(V^{(1)}) \cong \Hbul(\bso(V))$ as superspaces, it suffices by the K\"{u}nneth theorem to assume that $V$ is one-dimensional.

First suppose $V = k$, and let $v \in V$ be nonzero. Then $\bso_n^0(V) = \bss^n(V) \otimes \bsa^0(V)$ is spanned by $v^n \otimes 1$, $\bso_n^1(V) = \bss^{n-1}(V) \otimes \bsa^1(V)$ is spanned by $v^{n-1} \otimes v$, and $\bso_n^i(V) = 0$ otherwise. Now $\bsd(v^n \otimes 1) = n \cdot (v^{n-1} \otimes v)$, so $\opH^i(\bso_n(V))$ is one-dimensional and spanned by the class of $v^n \otimes 1$ if $i = 0$ and $p \mid n$, is one-dimensional and spanned by the class of $v^{n-1} \otimes v$ if $i=1$ and $p \mid n$, and is zero otherwise. Then $\Hbul(\bso_{pn}(V)) \cong \bso_n(V^{(1)})$ as superspaces.

Next suppose $V = \Pi(k)$, and let $v \in V$ be nonzero. Then $\bso_n^n(V) = \bss^0(V) \otimes \bsa^n(V)$ is spanned by $1 \otimes \gamma_n(v)$, $\bso_n^{n-1}(V) = \bss^1(V) \otimes \bsa^{n-1}(V)$ is spanned by $v \otimes \gamma_{n-1}(v)$, and $\bso_n^i(V) = 0$ otherwise. Now $\bsd(v \otimes \gamma_{n-1}(v)) = n \cdot (1 \otimes \gamma_n(v))$, so $\opH^i(\bso(V))$ is one-dimensional and spanned by the class of $v \otimes \gamma_{n-1}(v)$ if $i=n-1$ and $p \mid n$, is one-dimensional and spanned by the class of $1 \otimes \gamma_n(v)$ if $i=n$ and $p \mid n$, and is zero otherwise. Then $\Hbul(\bso_{pn}(V)) \cong \bso_n(V^{(1)})$ as superspaces.

Now let $V \in \bsv$ be arbitrary. Let $\{ v_1,\ldots,v_\ell \}$ be a homogeneous basis for $V$, and let $\{ v_1',\ldots,v_\ell' \}$ be the same set but considered as a basis for $V^{(1)}$. From the previous two paragraphs, it follows not only that $\bso(V^{(1)}) \cong \Hbul(\bso(V))$ as superspaces, but that $\Hbul(\bso(V))$ is generated as a superalgebra by the cohomology classes of $v_i^p \otimes 1$ and $v_i^{p-1} \otimes v_i$ for $\ol{v}_i = \ol{0}$, and the classes of $v_i \otimes \gamma_{p-1}(v_i)$ and $1 \otimes \gamma_{p^e}(v_i)$ for $\ol{v}_i = \ol{1}$ and $e \geq 1$. These elements generate a subalgebra of $\bso(V)$ isomorphic to $\bso(V^{(1)})$. Explicitly, there exists an injective algebra homomorphism $\theta(V): \bso(V^{(1)}) \hookrightarrow \bso(V)$ satisfying
\begin{equation} \label{eq:Cartieriso}
\left. \begin{aligned} v_i' \otimes 1 &\mapsto v_i^p \otimes 1 \\ 1 \otimes v_i' &\mapsto v_i^{p-1} \otimes v_i \end{aligned} \right\rbrace \text{for $\ol{v}_i = \ol{0}$,} \quad
\left. \begin{aligned} v_i' \otimes 1 &\mapsto v_i \otimes \gamma_{p-1}(v_i) \\ 1 \otimes \gamma_{p^e}(v_i') &\mapsto 1 \otimes \gamma_{p^{e+1}}(v_i) \end{aligned} \right\rbrace \text{for $\ol{v}_i = \ol{1}$ and $e \geq 0$.}
\end{equation}
Then $\theta(V)$ induces an isomorphism of superalgebras $\bso(V^{(1)}) \cong \Hbul(\bso(V))$, which by abuse of notation we also denote by $\theta(V)$. Using the relations in the divided power algebra and the fact that $\lambda^p = \lambda$ for all $\lambda \in \F_p \subseteq k$, one can check that $\theta(V)(1 \otimes \gamma_n(v_i')) = 1 \otimes \gamma_{pn}(v_i)$ for all $n \in \N$.

Set $F = \Hbul(\bso) \in \bsp$. Now we show that the isomorphism $\theta(V): \bso(V^{(1)}) \rightarrow F(V)$ lifts to an isomorphism of strict polynomial superfunctors. To do this, we must show for each $V,W \in \bsv$, each $\phi \in \bsg^{pn} \Hom_k(V,W)$, and each $z \in \bso_n(V^{(1)})$ that
\begin{equation} \label{eq:Cartiercompatible}
[\theta(W) \circ \bso^{(1)}(\phi)](z) = [F(\phi) \circ \theta(V)](z).
\end{equation}
Since the multiplication morphisms for $\bso^{(1)}$ and $\Hbul(\bso)$ are natural transformations, it suffices to verify \eqref{eq:Cartiercompatible} as $z$ ranges over a set of generators of the algebra $\bso(V^{(1)})$. By linearity, it also suffices to verify \eqref{eq:Cartiercompatible} as $\phi$ ranges over a basis for $\bsg^{pn} \Hom_k(V,W)$. We will also find it convenient not to consider $F(\phi)$ directly, but to consider the function $\bso(\phi)$ that induces $F(\phi)$.

Let $\{ w_1,\ldots,w_m \}$ be a homogeneous basis for $W$, and let $\{ w_1',\ldots,w_m' \}$ be the same set but considered as a basis for $W^{(1)}$. For each $1 \leq i \leq \ell$ and $1 \leq j \leq m$, let $e_{i,j} \in \Hom_k(V,W)$ be the ``matrix unit'' satisfying $e_{i,j}(v_a) = \delta_{j,a} w_i$, and let $e_{i,j}'$ be the corresponding element of $\Hom_k(V^{(1)},W^{(1)})$. Then the $e_{i,j}$ form a homogeneous basis for $\Hom_k(V,W)$. Consequently, $\bsg (\Hom_k(V,W))$ admits a basis consisting of all monomials $\prod_{i,j} \gamma_{a_{i,j}}(e_{i,j})$ (the products taken, say, in the lexicographic order) with $a_{i,j} \in \N$ and $a_{i,j} \leq 1$ if $\ol{e_{i,j}} = \ol{1}$. Of course, a similar statement holds for $\bsg (\Hom_k(V^{(1)},W^{(1)}))$. Now to verify \eqref{eq:Cartiercompatible}, we can assume that $\phi = \prod_{i,j} \gamma_{a_{i,j}}(e_{i,j})$ for some $a_{i,j} \in \N$ with $a_{i,j} \leq 1$ if $\ol{e_{i,j}} = \ol{1}$, and that $z$ is one of the generators for $\bso(V^{(1)})$ appearing in \eqref{eq:Cartieriso}. 

First we consider the expression $[\theta(W) \circ \bso^{(1)}(\phi)](z)$. Since $(\bso^{(1)})(\phi) = \bso(\varphi^\#(\phi))$, where $\varphi^\#$ is the dual Frobenius morphism described in \eqref{eq:dualFrobenius}, it follows that $\varphi^\#(\phi) = 0$ if $p \nmid a_{i,j}$ for some $a_{i,j}$, and $\varphi^\#(\phi) = \prod_{i,j} \gamma_{a_{i,j}/p}(e_{i,j}')$ otherwise. So suppose $p \mid a_{i,j}$ for each $a_{i,j}$; say $a_{i,j} = pb_{i,j}$. If $\ol{v}_n' = \ol{0}$ and $z = v_n' \otimes 1 \in \bso_1(V^{(1)})$, then by assumption $\sum_{i,j} a_{i,j} = p$, so $\phi = \gamma_{p}(e_{i,j})$ for some $i,j$. Then
\begin{align*}
[\theta(W) \circ \bso^{(1)}(\gamma_p(e_{i,j}))](v_n' \otimes 1) &= \delta_{j,n} \cdot \theta(W)(w_i' \otimes 1) = \delta_{j,n} (w_i^p \otimes 1), & \text{and similarly,} \\
[\theta(W) \circ \bso^{(1)}(\gamma_p(e_{i,j}))](1 \otimes v_n') &= \delta_{j,n} (w_i^{p-1} \otimes w_i) & \text{if $\ol{v}_n' = \ol{0}$, and} \\
[\theta(W) \circ \bso^{(1)}(\gamma_p(e_{i,j}))](v_n' \otimes 1) &= \delta_{j,n} (w_i \otimes \gamma_{p-1}(w_i)) & \text{if $\ol{v}_n' = \ol{1}$.}
\end{align*}
Now suppose $\ol{v}_n' = \ol{1}$ and $z = 1 \otimes \gamma_{p^e}(v_n') \in \bso_{p^e}(V^{(1)})$. Then $\sum_{i,j} a_{i,j} = p^{e+1}$, and
\begin{align*}
[\theta(W) \circ \bso^{(1)}(\phi)](z) &= \textstyle [\theta(W) \circ \bso(\prod_{i,j} \gamma_{b_{i,j}}(e_{i,j}'))](1 \otimes \gamma_{p^e}(v_n')) \\
&= \textstyle \theta(W) \left( \prod_{b_{i,j} \neq 0} [\delta_{j,n} \cdot (1 \otimes \gamma_{b_{i,j}}(w_i'))] \right) \\
&= \textstyle \prod_{a_{i,j} \neq 0} [\delta_{j,n} \cdot (1 \otimes \gamma_{a_{i,j}}(w_i))].
\end{align*}

Next we consider the expression $[\bso(\phi) \circ \theta(V)](z)$. Set $c = \sum_{i,j} a_{i,j}$, let $H$ be the Young subgroup $\prod_{i,j} \fS_{a_{i,j}}$ of $\fS_c$, and let $J \subset \fS_c$ be a set of right coset representatives for $H$. Then as in \cite[IV.5.3]{Bourbaki:2003}, we can write $\phi = \prod_{i,j} \gamma_{a_{i,j}}(e_{i,j})$ in the form
\begin{equation} \label{eq:phiorbitsum} \textstyle \textstyle
\phi = \sum_{\sigma \in J} \left( \bigotimes_{i,j} [(e_{i,j})^{\otimes a_{i,j}}] \right).\sigma,
\end{equation}
where the factors in the tensor product are taken in the lexicographic order. Now suppose $\ol{v}_n' = \ol{0}$ and $z = v_n' \otimes 1$. Then $\sum_{i,j} a_{i,j} = p$, and $[\bso(\phi) \circ \theta(V)](v_n' \otimes 1) = \bss^p(\phi)(v_n^p) \otimes 1$. From \eqref{eq:phiorbitsum} it follows that $\bss^p(\phi)(v_n^p) = 0$ unless $a_{i,j} =0$ for all $j \neq n$. So suppose $a_{i,j} = 0$ for all $j \neq n$. Then
\[ \textstyle
\bss^p(\phi)(v_n^p) = \binom{p}{a_{1,n},a_{2,n},\ldots,a_{\ell,n}} \left(\prod_{i=1}^\ell w_i^{a_{i,n}} \right)\in \bss^p(W).
\]
Since $k$ is a field of characteristic $p$, the multinomial coefficient in this expression is equal to zero if $a_{i,n} < p$ for some $i$. Then $[\bso(\phi) \circ \theta(V)](v_n' \otimes 1) = 0$ unless $\phi = \gamma_p(e_{i,j})$ for some $i$ and $j$, and in this case one has $[\bso(\gamma_p(e_{i,j})) \circ \theta(V)](v_n' \otimes 1) = \delta_{j,n}(w_i^p \otimes 1)$. A similar analysis shows that $[\bso(\gamma_p(e_{i,j})) \circ \theta(V)](1 \otimes v_n') = \delta_{j,n}(w_i^{p-1} \otimes w_i)$, but that if $\phi = \prod_{i,j} \gamma_{a_{i,j}}(e_{i,j})$ and $a_{i,j} \neq 0$ for $j \neq n$, or if $a_{i,j} \neq 0$ for more than one factor in the product, then $[\bso(\phi) \circ \theta(V)](1 \otimes v_n')$ is either equal to zero or (assuming $a_{i,j} = 0$ for $j \neq n$) is equal to the coboundary of $(\prod_{1 \leq i \leq m} w_i^{a_{i,n}}) \otimes 1$.

Now suppose $\ol{v}_n' = \ol{1}$, $z = 1 \otimes \gamma_{p^e}(v_n')$, and $\phi = \prod_{i,j} \gamma_{a_{i,j}}(e_{i,j})$. Then
\[ \textstyle
[\bso(\phi) \circ \theta(V)](1 \otimes \gamma_{p^e}(v_n')) = \bso(\phi)(1 \otimes \gamma_{p^{e+1}}(v_n)) = \prod_{a_{i,j} \neq 0} [\delta_{j,n} \cdot (1 \otimes \gamma_{a_{i,j}}(w_i))].
\]
The analysis of $\bso(V)$ from the one-dimensional case shows that the last expression in this equation is a coboundary unless $p \mid a_{i,j}$ for all $a_{i,j}$. Finally, suppose $\ol{v}_n' = \ol{1}$ and $z = v_n' \otimes 1$. Then
\begin{align*}
[\bso(\phi) \circ \theta(V)](v_n' \otimes 1) &= \bso(\phi)(v_n \otimes \gamma_{p-1}(v_n)) \\
&= [\bso(\phi) \circ \bsk(V)](1 \otimes \gamma_p(v_n)) \\
&= [\bsk(W) \circ \bso(\phi)](1 \otimes \gamma_p(v_n)).
\end{align*}
Now $\bso(\phi)(1 \otimes \gamma_p(v_n)) = \prod_{a_{i,j} \neq 0} [\delta_{j,n} (1 \otimes \gamma_{a_{i,j}}(w_i))]$, and this expression is a coboundary in $\bso_p(W)$ unless $p \mid a_{i,j}$ for all $a_{i,j}$. But $\bsk(W) \circ \bsd(W) = - \bsd(W) \circ \bsk(W)$ on $\bso_p(W)$ by Lemma \ref{lem:homotopy}, so it follows that $[\bso(\phi) \circ \theta(V)](v_n' \otimes 1)$ is a coboundary unless $p \mid a_{i,j}$ for all $a_{i,j}$, i.e., unless $\phi = \gamma_p(e_{i,j})$ for some $i$ and $j$. On the other hand, if $\phi = \gamma_p(e_{i,j})$, then $[\bso(\phi) \circ \theta(V)](v_n' \otimes 1) = \delta_{j,n}(w_i \otimes \gamma_{p-1}(w_i))$.

Combining the observations of the previous four paragraphs, \eqref{eq:Cartiercompatible} then follows.
\end{proof}

\begin{remark} \label{remark:Cartierrestriction}
As a $\bsp$-algebra, $\bso^{(1)} \cong \bss^{(1)} \otimes \bsa^{(1)}$. Identifying $\Hbul(\bso_{pn})$ with $\bso_n^{(1)}$ via the super Cartier isomorphism, one gets, in the notation of \eqref{eq:twistcomponents},
\[
\opH^t(\bso_{pn}) = \bigoplus_{\substack{a+b+c+d=n \\ b(p-1)+c+pd = t}} (S_0^{a(1)} \otimes \Lambda_1^{b(1)}) \otimes (\Lambda_0^{c(1)} \gotimes \Gamma_1^{d(1)}).
\]
\end{remark}

\subsection{The super Koszul kernel subcomplex} \label{subsec:koszulkernel}

Given $i,n \in \N$ with $i \leq n$, set
\[
\bsK_n^i = \ker \{ \bsk: \bso_n^i \rightarrow \bso_n^{i-1} \}.
\]
It follows from Lemma \ref{lem:homotopy} that $(\bsK_{pn},\bsd)$ is a subcomplex of $(\bso_{pn},\bsd)$. We call $(\bsK_{pn},\bsd)$ the \emph{Koszul kernel subcomplex} of $\bso_{pn}$. On $\bsvzero$, $\bsK_n^i$ restricts to the Koszul kernel subfunctor $K_n^i$ defined in \cite{Franjou:1994,Friedlander:1997}. In particular, $(\bsK_{pn},\bsd)$ restricts to the ordinary Koszul kernel subcomplex $(K_{pn},d)$. On $\bsvone$, $\bsK_n^i$ restricts to $(K_n^{n-i-1})^\#$, and $(\bsK_{pn},\bsd)$ restricts to $(K_{pn}^\#,d^\#)$.

Since $(\bso,\bsk)$ is acyclic, there exist for each $0 \leq i \leq n$ and $r \geq 1$  short exact sequences in $\bsp_\ev$
\begin{equation} \label{eq:superKoszulses}
\begin{gathered} 
0 \rightarrow \bsK_n^i \rightarrow \bso_n^i \stackrel{\bsk}{\rightarrow} \bsK_n^{i-1} \rightarrow 0, \\
0 \rightarrow \bsK_n^i \circ \bsi_0^{(r)} \rightarrow \bso_n^i \circ \bsi_0^{(r)} \stackrel{\bsk^{(r)}}{\rightarrow} \bsK_n^{i-1} \circ \bsi_0^{(r)} \rightarrow 0, \\
0 \rightarrow \bsK_n^i \circ \bsi_1^{(r)} \rightarrow \bso_n^i \circ \bsi_1^{(r)} \stackrel{\bsk^{(r)}}{\rightarrow} \bsK_n^{i-1} \circ \bsi_1^{(r)} \rightarrow 0.
\end{gathered}
\end{equation}
Since $\bsK_n^0 = \bso_n^0 \cong \bss^n$, $\bsK_n^n = 0$, and $\bso_n^n \cong \bsa^n$, the associated long exact sequences in cohomology together with Theorem \ref{thm:vanishing} then imply:

\begin{lemma} \label{lem:degreeshift}
Suppose $F \in \bsp$ is additive. Then for all $i,n,r \in \N$ with $i < n$ and $r \geq 1$,
\begin{gather}
\Ext_{\bsp}^s(F,\bss^n) \cong \Ext_{\bsp}^{s+i}(F,\bsK_n^i) \cong \Ext_{\bsp}^{s+n-1}(F,\bsa^n),  \label{eq:degreeshift} \\
\Ext_{\bsp}^s(F,S_0^{n(r)}) \cong \Ext_{\bsp}^{s+i}(F,\bsK_n^i \circ \bsi_0^{(r)}) \cong \Ext_{\bsp}^{s+n-1}(F,\Lambda_0^{n(r)}), \label{eq:evendegreeshift} \\
\Ext_{\bsp}^s(F,\Lambda_1^{n(r)}) \cong \Ext_{\bsp}^{s+i}(F,\bsK_n^i \circ \bsi_1^{(r)}) \cong \Ext_{\bsp}^{s+n-1}(F,\Gamma_1^{n(r)}). \label{eq:odddegreeshift}
\end{gather}
\end{lemma}

\begin{corollary} \label{cor:degreeshift}
Suppose $F \in \bsp$ is additive. Then for all $n,r \in \N$ with $r \geq 1$,
\begin{align*}
\Ext_{\bsp}^s(F,S_1^{n(r)}) &\cong \Ext_{\bsp}^{s+n-1}(F,\Lambda_1^{n(r)}) \quad \text{and} \\
\Ext_{\bsp}^s(F,\Lambda_0^{n(r)}) &\cong \Ext_{\bsp}^{s+n-1}(F,\Gamma_0^{n(r)}).
\end{align*}
\end{corollary}

\begin{proof}
The reader can check that $F \in \bsp$ is additive if and only if $F \circ \Pi$ and $\Pi \circ F$ are additive. Then the isomorphisms follow from \eqref{eq:evendegreeshift}, \eqref{eq:odddegreeshift}, and \eqref{eq:conjugateiso} by considering the conjugation action of $\Pi$ on extension groups in $\bsp$.
\end{proof}

\begin{remark} \label{rem:Yonedacompatible}
The connecting homomorphisms that induce the isomorphisms in Lemma \ref{lem:degreeshift} can be realized as left Yoneda products by the extension classes of the short exact sequences in \eqref{eq:superKoszulses}; cf.\ \cite[IV.9]{Hilton:1997}. Thus, the isomorphism $\Ext_{\bsp}^s(F,\bss^n) \cong \Ext_{\bsp}^{s+n-1}(F,\bsa^n)$ identifies with left multiplication by the extension class of the super Koszul complex $\bs{Kz}_n$. Using \eqref{eq:evenlift} to consider the $r$-th Frobenius twists of the ordinary Koszul complex $Kz_n$ and its dual ${Kz_n}^\#$ as exact sequences in $\bsp$, the composite isomorphism in \eqref{eq:evendegreeshift} identifies with left multiplication by the extension class in $\Ext_{\bsp}^{n-1}({S_0^n}^{(r)},{\Lambda_0^n}^{(r)})$ of ${Kz_n}^{(r)}$, and the composite isomorphism in \eqref{eq:odddegreeshift} identifies with left multiplication by the extension class in  $\Ext_{\bsp}^{n-1}({\Lambda_1^n}^{(r)},{\Gamma_1^n}^{(r)})$ of
\[
\begin{cases}
Kz_n^{\#(r)} \circ \Pi & \text{if $n$ is even,} \\
\Pi \circ Kz_n^{\#(r)} \circ \Pi & \text{if $n$ is odd.}
\end{cases}
\]
Conjugating by $\Pi$ (if $n$ is odd), or pre-composing with $\Pi$ (if $n$ is even), we see that similar remarks also apply to the isomorphisms in Corollary \ref{cor:degreeshift}. Since the Yoneda product is assoc\-iative, it follows that the isomorphisms in Lemma \ref{lem:degreeshift} and Corollary \ref{cor:degreeshift} commute with right multiplication by elements of the Yoneda algebra $\Ext_{\bsp}^\bullet(F,F)$.
\end{remark}

\begin{remark} \label{rem:restrictcommute}
Since the restriction functor $\bsp \rightarrow \cp$ maps $\bs{Kz}_n$ to the ordinary Koszul complex $Kz_n$, and since it forgets the fact that ${Kz_n}^{(r)}$ had been lifted to $\bsp$, it follows that the isomorphisms in \eqref{eq:degreeshift} and \eqref{eq:evendegreeshift} fit into commutative diagrams
\[
\xymatrix{
\Ext_{\bsp}^\bullet(F,\bss^n) \ar@{->}[r]^-{\sim} \ar@{->}[d] & \Ext_{\bsp}^{\bullet+n-1}(F,\bsa^n) \ar@{->}[d] \\
\Ext_{\cp}^\bullet(F|_{\bsvzero},S^n) \ar@{->}[r]^-{\sim} & \Ext_{\cp}^{\bullet+n-1}(F|_{\bsvzero},\Lambda^n)
} \qquad
\xymatrix{
\Ext_{\bsp}^\bullet(F,S_0^{n(r)}) \ar@{->}[r]^-{\sim} \ar@{->}[d] & \Ext_{\bsp}^{\bullet+p^r-1}(F,\Lambda_0^{n(r)}) \ar@{->}[d] \\
\Ext_{\cp}^\bullet(F|_{\bsvzero},S^{n(r)}) \ar@{->}[r]^-{\sim} & \Ext_{\cp}^{\bullet+p^r-1}(F|_{\bsvzero},\Lambda^{n(r)})
}
\]
in which the vertical arrows are the corresponding restriction homomorphisms. The isomorphisms in the bottom rows of each diagram are the isomorphisms in  \cite[Proposition 4.4]{Friedlander:1997}, which also admit descriptions as left multiplication by the extension classes of $Kz_n$ and ${Kz_n}^{(r)}$, respectively.
\end{remark}

It is straightforward to check that the  Koszul differential is compatible with the super Cartier isomorphism. For example, if $\ol{v}_i' = \ol{1}$ and $e \geq 0$, then
\begin{align*}
[\theta(V) \circ \bsk(V^{(1)})](1 \otimes \gamma_{p^e}(v_i')) &= \theta(V)(v_i' \otimes \gamma_{p^e-1}(v_i')) \\
&= \theta(V)(v_i' \otimes 1) \cdot \theta(V)(1 \otimes \gamma_{p^e-1}(v_i')) \\
&= (v_i \otimes \gamma_{p-1}(v_i)) \cdot (1 \otimes \gamma_{p^{e+1}-p}(v_i)) \\
&= v_i \otimes \gamma_{p^{e+1}-1}(v_i) \\
&= [\bsk(V) \circ \theta(V)](1 \otimes \gamma_{p^e}(v_i')).
\end{align*}
Then one has the following analogue of \cite[Proposition 3.5]{Franjou:1994}:

\begin{proposition} \label{prop:Cartierkernel}
The super Cartier isomorphism restricts for each $n \in \N$ to an isomorphism
\[
\bsK_n^{(1)} \cong \Hbul(\bsK_{pn}).
\]
In particular, the inclusion of complexes $\bsK_{pn} \hookrightarrow \bso_{pn}$ induces an inclusion $\Hbul(\bsK_{pn}) \hookrightarrow \Hbul(\bso_{pn})$.
\end{proposition}

\begin{proof}
As in \cite{Franjou:1994}, we consider $\bso_{pn}$ as a module over the self-injective algebra $D := k[x]/(x^2)$, with $x$ acting via the Koszul differential $\bsk$. More accurately, for each $V \in \bsv$ the superspace $\bso_{pn}(V)$ is naturally a $D$-module, with $x$ acting via $\bsk(V)$, but for the sake of legibility we will omit explicit references to $V$ in most of the rest of the proof. Then $\bso_{pn}^D = \bsK_{pn}$ and $[\bso_n^{(1)}]^D = \bsK_n^{(1)}$. Moreover, since the action of $\bsk$ on $\bso$ is acyclic, $\bso$ is free, hence injective, as a $D$-module. The super Cartier isomorphism $\theta: \bso_n^{(1)} \simrightarrow \Hbul(\bso_{pn})$ restricts to an isomorphism $\bsK_n^{(1)} \simrightarrow \Hbul(\bso_{pn})^D$, which we denote by $\theta^D$. Evidently, $\theta^D$ factors through the composite
\[
\bsK_n^{(1)} \stackrel{\theta_K}{\longrightarrow} \Hbul(\bsK_{pn}) \stackrel{\Hbul(\iota)}{\longrightarrow} \Hbul(\bso_{pn})^D,
\]
where $\theta_K$ is the natural map induced by $\theta$, and $\Hbul(\iota)$ is the map in cohomology induced by the inclusion of complexes $\iota: \bsK_{pn} = \bso_{pn}^D \hookrightarrow \bso_{pn}$. Then to show that $\theta_K$ is an isomorphism, it suffices to show that $\Hbul(\iota)$ is an isomorphism.

For $j \in \Z$, set $Q^j = \bso_{pn}$, and let $d_Q: Q^j \rightarrow Q^{j+1}$ be $(-1)^j \bsd$, where $\bsd: \bso_{pn} \rightarrow \bso_{pn}$ is the de~Rham differential on $\bso_{pn}$. Then by Lemma \ref{lem:homotopy}, $Q^\bullet$ is a cochain complex of $D$-modules, and $\opH^j(Q) = \Hbul(\bso_{pn})$ for each $j \in \Z$. Also, $\opH^j(Q^D) = \Hbul(\bso_{pn}^D) = \Hbul(\bsK_{pn})$. Next, let $(P,d_P)$ be a projective resolution of the trivial $D$-module $k$. Now consider the right half-plane double complex $C^{i,j} = \Hom_D(P_i,Q^j)$. The row-wise and column-wise filtrations of the total complex $\Tot(C)$ are both exhaustive and weakly convergent in the sense of \cite[\S3.1]{Mccleary:2001}, so by \cite[Theorem 3.2]{Mccleary:2001} they give rise, as in Section \ref{subsec:hypercohomology}, to a pair of spectral sequences ${}^IE$ and ${}^{II}E$ that each converge to $\Hbul(\Tot(C))$. First computing the cohomology of $C$ along columns, and using the fact that $\opH^j(Q) = \Hbul(\bso_{pn}) \cong \bso_n^{(1)}$ is injective for the action of $D$, one gets
\[
{}^IE_1^{i,j} = \Hom_D(P_i,\opH^j(Q)),
\quad \text{hence} \quad
{}^IE_2^{i,j} = \begin{cases}
\Hom_D(k,\opH^j(Q)) = \Hbul(\bso_{pn})^D & \text{if $i =0$,} \\
0 & \text{if $i \neq 0$.}
\end{cases}
\]
On the other hand, first computing the cohomology of $C$ along rows, and using the injectivity of $Q^j = \bso_{pn}$ as a $D$-module, one gets
\[
{}^{II}E_1^{i,j} = \begin{cases}
\Hom_D(k,Q^j) & \text{if $i =0$,} \\
0 & \text{if $i \neq 0$,}
\end{cases} \quad \text{hence} \quad
{}^{II}E_2^{i,j} = \begin{cases}
\opH^j(Q^D) = \Hbul(\bsK_{pn}) & \text{if $i =0$,} \\
0 & \text{if $i \neq 0$.}
\end{cases}
\]
In particular, both spectral sequences collapse at the $E_2$-page to the column $i=0$, so the map
\[
\Hbul(\bsK_{pn}) = {}^{II}E_2^{0,j} = {}^{II}E_\infty^{0,j} \simrightarrow \opH^j(\Tot(C)) \simrightarrow {}^IE_\infty^{0,j} = {}^IE_2^{0,j} = \Hbul(\bso_{pn})^D
\]
is an isomorphism. In other words, for each $V \in \bsv$, the spaces $\Hbul(\bsK_{pn}(V))$ and $\Hbul(\bso_{pn}(V))^D$ are each of the same finite dimension. Then for each $V \in \bsv$, the map $\Hbul(\iota)$ must be an isomorphism, and hence so must $\theta_K$.
\end{proof}

Replacing $\bsK_n$ by $\Hbul(\bsK_{pn})$, there is the following analogue of Lemma \ref{lem:degreeshift}:

\begin{lemma} \label{lem:Kdegreeshift}
Let $n \in \N$ and let $F \in \bsp_{pn}$ be an additive functor. Then for $0 \leq t \leq n-1$,
\begin{align}
\Ext_{\bsp}^s(F,S_0^{n(1)}) &\cong \Ext_{\bsp}^s(F,\opH^0(\bsK_{pn})) \cong \Ext_{\bsp}^{s+t}(F,\opH^t(\bsK_{pn})), \label{eq:lowdegreeshift} \\
\Ext_{\bsp}^s(F,\Lambda_1^{n(1)}) &\cong \Ext_{\bsp}^s(F,\opH^{n(p-1)}(\bsK_{pn})) \cong \Ext_{\bsp}^{s+t}(F,\opH^{t+n(p-1)}(\bsK_{pn})), \label{eq:highdegreeshift}
\end{align}
and $\Ext_{\bsp}^s(F,\opH^t(\bsK_{pn})) = 0$ if $n \leq t < n(p-1)$ or if $t \geq pn$.
\end{lemma}

\begin{proof}
It follows from Lemma \ref{lem:homotopy} that there exists a short exact sequence of cochain complexes
\[
0 \rightarrow \bsK_{pn}^\bullet \rightarrow \bso_{pn}^\bullet \stackrel{\bsk}{\rightarrow} \bsK_{pn}^{\bullet-1} \rightarrow 0,
\]
provided the differentials $\bsd: \bso_{pn}^i \rightarrow \bso_{pn}^{i+1}$ and $\bsd: \bsK_{pn}^i \rightarrow \bsK_{pn}^{i+1}$ are each replaced by $(-1)^i \bsd$. Then the corresponding long exact sequence in cohomology takes the form
\[
\cdots \rightarrow \opH^{t-1}(\bsK_{pn}^{\bullet-1}) \rightarrow \opH^t(\bsK_{pn}^\bullet) \rightarrow \opH^t(\bso_{pn}^\bullet) \rightarrow \opH^t(\bsK_{pn}^{\bullet-1}) \rightarrow \opH^{t+1}(\bsK_{pn}^\bullet) \rightarrow \cdots.
\]
Since $\opH^t(\bsK_{pn}^{\bullet-1}) = \opH^{t-1}(\bsK_{pn}^\bullet)$, and since by Proposition \ref{prop:Cartierkernel} the inclusion of complexes $\bsK_{pn}^\bullet \hookrightarrow \bso_{pn}^\bullet$ induces an inclusion at the level of cohomology groups, the long exact sequence decomposes into a family of short exact sequences of the form
\begin{equation} \label{eq:sescohomology}
0 \rightarrow \opH^t(\bsK_{pn}) \rightarrow \opH^t(\bso_{pn}) \rightarrow \opH^{t-1}(\bsK_{pn}) \rightarrow 0.
\end{equation}
Identifying $\opH^t(\bso_{pn})$ with a subfunctor of ${\bso_n}^{(1)}$ as in Remark \ref{remark:Cartierrestriction}, and identifying $\opH^{t-1}(\bsK_{pn})$ with a subfunctor of ${\bsK_n}^{(1)}$, the map $\opH^t(\bso_{pn}) \rightarrow \opH^{t-1}(\bsK_{pn})$ in \eqref{eq:sescohomology} identifies with $\bsk^{(1)}$.

Now consider the long exact sequence in cohomology obtained by applying $\Hom_{\bsp}(F,-)$ to \eqref{eq:sescohomology}. It follows from Theorem \ref{thm:vanishing} and Remark \ref{remark:Cartierrestriction} that $\Ext_{\bsp}^s(F,\opH^t(\bso_{pn})) = 0$, and hence that $\Ext_{\bsp}^s(F,\opH^{t-1}(\bsK_{pn})) \cong \Ext_{\bsp}^{s+1}(F,\opH^t(\bsK_{pn}))$, except perhaps if $t$ equals $0$, $n$, $n(p-1)$, or $pn$. Since $\opH^0(\bsK_{pn}) = \opH^0(\bso_{pn}) \cong {S_0^n}^{(1)}$ and $\opH^{pn}(\bsK_{pn}) = 0$, this establishes \eqref{eq:lowdegreeshift}, the second isomorphism in \eqref{eq:highdegreeshift}, and the equality $\Ext_{\bsp}^s(F,\opH^t(\bsK_{pn})) = 0$ for $t \geq pn$.

Next recall that the Koszul differential defines an isomorphism $\bso_n^n \cong \bsK_n^{n-1}$. Then the family of maps $\bsk^{(1)}: \opH^t(\bso_{pn}) \rightarrow \opH^{t-1}(\bsK_{pn})$ in \eqref{eq:sescohomology} restricts to an isomorphism from the subfunctor ${\bso_n^n}^{(1)}$ of $\Hbul(\bso_{pn})$ to the subfunctor $\bsK_n^{n-1(1)}$ of $\Hbul(\bsK_{pn})$. By Remark \ref{remark:Cartierrestriction} and Theorem \ref{thm:vanishing},
\begin{align*}
\Ext_{\bsp}^s(F,\bso_n^{n(1)}) &\cong \Ext_{\bsp}^s(F,\opH^n(\bso_{pn})) \oplus \Ext_{\bsp}^s(F,\opH^{pn}(\bso_{pn})). 
\end{align*}
Then taking $t = n$ in \eqref{eq:sescohomology}, it follows in the long exact sequence in cohomology that the map
\[
\Ext_{\bsp}^s(F,\bsk^{(1)}): \Ext_{\bsp}^s(F,\opH^n(\bso_{pn})) \rightarrow \Ext_{\bsp}^s(F,\opH^{n-1}(\bsK_{pn}))
\]
is injective. On the other hand,
\[
\Ext_{\bsp}^{s-(n-1)}(F,S_0^{n(1)}) \cong \Ext_{\bsp}^s(F,\Lambda_0^{n(1)}) \cong \Ext_{\bsp}^s(F,\opH^n(\bso_{pn})) 
\]
by \eqref{eq:evendegreeshift}, and $\Ext_{\bsp}^{s-(n-1)}(F,S_0^{n(1)}) \cong \Ext_{\bsp}^s(F,\opH^{n-1}(\bsK_{pn}))$ by \eqref{eq:lowdegreeshift}. Moreover, extension groups between homogeneous functors in $\bsp$ are always finite-dimensional, since by Theorem \ref{thm:enoughprojectives} they can be identified with extension groups between finite-dimensional modules for a finite-dimensional algebra. Then $\Ext_{\bsp}^s(F,\bsk^{(1)})$ must be an isomorphism, and it follows from the long exact sequence in cohomology that $\Ext_{\bsp}^s(F,\opH^n(\bsK_{pn})) = 0$. Now with the observations of the previous paragraph, this shows that $\Ext_{\bsp}^\bullet(F,\opH^t(\bsK_{pn})) = 0$ for $n \leq t < n(p-1)$.

Finally, taking $t = n(p-1)$ in \eqref{eq:sescohomology}, it now follows from the long exact sequence in cohomology, Remark \ref{remark:Cartierrestriction}, and Theorem \ref{thm:vanishing}, and the previous observations that the inclusion $\bsK_{pn} \hookrightarrow \bso_{pn}$ induces isomorphisms
\begin{equation} \label{eq:rowthreeiso}
\Ext_{\bsp}^\bullet(F,\opH^{n(p-1)}(\bsK_{pn})) \cong \Ext_{\bsp}^\bullet(F,\opH^{n(p-1)}(\bso_{pn})) \cong \Ext_{\bsp}^\bullet(F,\Lambda_1^{n(1)}).
\qedhere
\end{equation}
\end{proof}

\subsection{Vector space structure of \texorpdfstring{$\Ext_{\bsp}^\bullet(\bsir,{S_0}^{p^{r-1}(1)})$}{ExtP(Ir,S0pr-1(1))}} \label{subsec:vectorspacej=1}

Our goal in this section is to describe for $r \geq 1$ the $E_2$-pages of the spectral sequence
\begin{align}
\Ebar_2^{s,t} &= \Ext_{\bsp}^s(\bsir,\opH^t(\bso_{p^r})) \Rightarrow \bbExt_{\bsp}^{s+t}(\bsir,\bso_{p^r}), \quad \text{and} \label{eq:Oprspecseq} \\
E_2^{s,t} &= \Ext_{\bsp}^s(\bsir,\opH^t(\bsK_{p^r})) \Rightarrow \bbExt_{\bsp}^{s+t}(\bsir,\bsK_{p^r}) \label{eq:Kprspecseq}
\end{align}
obtained by taking $A = \bsir$ and $C = \bso_{p^r}$ or $C = \bsK_{p^r}$ in \eqref{eq:secondhypercohomology}. Since $\bsir = \bsirzero \oplus \bsirone$, it follows that \eqref{eq:Oprspecseq} and \eqref{eq:Kprspecseq} each decompose into a direct sum of hypercohomology spectral sequences. Specifically, given $\ell \in \set{0,1}$, let
\begin{align}
\Ebar_{2,\ell}^{s,t} &= \Ext_{\bsp}^s(\bsi_\ell^{(r)},\opH^t(\bso_{p^r})) \Rightarrow \bbExt_{\bsp}^{s+t}(\bsi_\ell^{(r)},\bso_{p^r}), \quad \text{and} \label{eq:ellOprspecseq} \\
E_{2,\ell}^{s,t} &= \Ext_{\bsp}^s(\bsi_\ell^{(r)},\opH^t(\bsK_{p^r})) \Rightarrow \bbExt_{\bsp}^{s+t}(\bsi_\ell^{(r)},\bsK_{p^r}) \label{eq:ellKprspecseq}
\end{align}
be the hypercohomology spectral sequences obtained by taking $A = \bsilr$ and $C = \bso_{p^r}$ or $C = \bsK_{p^r}$ in \eqref{eq:secondhypercohomology}, respectively. Then \eqref{eq:Oprspecseq} is the direct sum of the spectral sequences obtained by taking $\ell = 0$ and $\ell = 1$ in \eqref{eq:ellOprspecseq}, and similarly for \eqref{eq:Kprspecseq} and \eqref{eq:ellKprspecseq}. We will exploit these direct sum decompositions to make explicit calculations. As a byproduct of our work, we will obtain the vector space structure of the extension group $\Ext_{\bsp}^\bullet(\bsir,{S_0}^{p^{r-1}(1)})$. First we compute the abutments of the spectral sequences \eqref{eq:ellOprspecseq} and \eqref{eq:ellKprspecseq}.

\begin{lemma} \label{lem:abutment}
Let $r \geq 1$, and let $\ell \in \set{0,1}$. Then:
\begin{align*}
\bbExt_{\bsp}^s(\bsi_\ell^{(r)},\bso_{p^r}) &= \begin{cases} k & \text{if $\ell = 0$ and either $s = 0$ or $s = 2p^r-1$,} \\ 0 & \text{otherwise.} \end{cases} \\
\bbExt_{\bsp}^s(\bsi_\ell^{(r)},\bsK_{p^r}) &= \begin{cases} k & \text{if $\ell = 0$ and $s$ is even with $0 \leq s < 2p^r$,} \\ 0 & \text{otherwise.} \end{cases}
\end{align*}
\end{lemma}

\begin{proof}
First take $A = {\bsi_\ell}^{(r)}$ and $C = \bso_{p^r}$ in \eqref{eq:firsthypercohomology}. Then it follows from Theorem \ref{thm:vanishing}, \eqref{eq:degreeshift}, the injectivity of $\bss^n$, and \eqref{eq:Yonedalemma} that ${}^IE_1^{s,t} = k$ if $\ell = 0$ and $s = t = 0$, or if $\ell = 0$, $s = p^r$, and $t = p^r-1$, but that ${}^IE_1^{s,t} = 0$ otherwise. Since the nonzero terms in the spectral sequence are of non-adjacent total degrees, it follows that the spectral sequence collapses at the $E_1$-page and that $\bbExt_{\bsp}^s({\bsi_\ell}^{(r)},\bso_{p^r})$ is described as claimed. Now take $A = {\bsi_\ell}^{(r)}$ and $C = \bsK_{p^r}$ in \eqref{eq:firsthypercohomology}. Then as above, it follows from \eqref{eq:degreeshift} that ${}^IE_1^{s,t} = k$ if $\ell = 0$ and $s=t$ with $0 \leq s < p^r$, but that ${}^IE_1^{s,t} = 0$ otherwise. Again, the nonzero terms in the spectral are of non-adjacent total degrees, so it follows that $\bbExt_{\bsp}^s({\bsi_\ell}^{(r)},\bsK_{p^r})$ is described as in the statement of the lemma.
\end{proof}

Our focus for the rest of this section is on analyzing the spectral sequences \eqref{eq:ellOprspecseq} and \eqref{eq:ellKprspecseq}. \textbf{For the rest of this section}, set $q = p^{r-1}$. Theorem \ref{thm:vanishing} and Remark \ref{remark:Cartierrestriction} imply that
\begin{equation} \label{eq:Ebaridentification}
\begin{aligned}
\Ebar_{2,\ell}^{s,0} &\cong \Ext_{\bsp}^s(\bsi_\ell^{(r)},S_0^{q(1)}), &
\Ebar_{2,\ell}^{s,q} &\cong \Ext_{\bsp}^s(\bsi_\ell^{(r)},\Lambda_0^{q(1)}), \\
\Ebar_{2,\ell}^{s,(p-1)q} &\cong \Ext_{\bsp}^s(\bsi_\ell^{(r)},\Lambda_1^{q(1)}), &
\Ebar_{2,\ell}^{s,pq} &\cong \Ext_{\bsp}^s(\bsi_\ell^{(r)},\Gamma_1^{q(1)}),
\end{aligned}
\end{equation}
and $\Ebar_{2,\ell}^{s,t} = 0$ otherwise. Then Lemma \ref{lem:degreeshift}, \eqref{eq:conjugatebsir}, and \eqref{eq:conjugateiso} imply
\begin{equation} \label{eq:rowisomorphisms}
\begin{split}
\Ebar_{2,0}^{s,0} &\cong \Ebar_{2,0}^{s+q-1,q} \cong \Ebar_{2,1}^{s+q-1,(p-1)q} \cong \Ebar_{2,1}^{s+2(q-1),pq}, \quad \text{and} \\
\Ebar_{2,1}^{s,0} &\cong \Ebar_{2,1}^{s+q-1,q} \cong \Ebar_{2,0}^{s+q-1,(p-1)q} \cong \Ebar_{2,0}^{s+2(q-1),pq}.
\end{split}
\end{equation}
Similarly, Lemma \ref{lem:Kdegreeshift} implies that
\begin{equation} \label{eq:Eisovanishing}
\begin{aligned}
E_{2,\ell}^{s,0} &\cong E_{2,\ell}^{s+t,t} \text{ and } E_{2,\ell}^{s,(p-1)q} \cong E_{2,\ell}^{s+t,(p-1)q+t} & \text{for $0 \leq t \leq q-1$, but} \\
E_{2,\ell}^{s,t} &=0 & \text{if $q \leq t < (p-1)q$ or if $t \geq pq$.}
\end{aligned}
\end{equation}
The proof of Lemma \ref{lem:Kdegreeshift} shows that the inclusion $\bsK_{p^r} \hookrightarrow \bso_{p^r}$ induces isomorphisms
\begin{equation} \label{eq:EEbariso}
E_{2,\ell}^{s,0} \cong \Ebar_{2,\ell}^{s,0} \quad \text{and} \quad E_{2,\ell}^{s,(p-1)q} \cong \Ebar_{2,\ell}^{s,(p-1)q}.
\end{equation}
Finally, Lemma \ref{lem:abutment} implies that
\begin{equation} \label{eq:abutments}
\begin{split}
\textstyle \bigoplus_{i+j=s} \Ebar_{\infty,\ell}^{i,j} &= \begin{cases} k & \text{if $\ell = 0$ and $s = 0$ or $s = 2p^r-1$,} \\ 0 & \text{otherwise}, \end{cases}  \\
\textstyle \bigoplus_{i+j=s} E_{\infty,\ell}^{i,j} &= \begin{cases} k & \text{if $\ell = 0$ and $s$ is even with $0 \leq s < 2p^r$,} \\ 0 & \text{otherwise.} \end{cases}
\end{split}
\end{equation}

\begin{theorem} \label{thm:quadruplecalculation}
Set $q = p^{r-1}$. In the spectral sequences $\Ebar_{2,0}$ and $\Ebar_{2,1}$, one has
\begin{align}
\Ebar_{2,0}^{s,0} \cong \Ebar_{2,0}^{s+q-1,q} \cong \Ebar_{2,1}^{s+q-1,(p-1)q} \cong \Ebar_{2,1}^{s+2(q-1),pq} &\cong \begin{cases} k & \text{if $s \equiv 0 \mod 2q$ and $s \geq 0$,} \\ 0 & \text{otherwise,} \end{cases} \label{eq:quadrupleeven} \\
\Ebar_{2,1}^{s,0} \cong \Ebar_{2,1}^{s+q-1,q} \cong \Ebar_{2,0}^{s+q-1,(p-1)q} \cong \Ebar_{2,0}^{s+2(q-1),pq} &\cong \begin{cases} k  & \text{if $s \equiv p^r \mod 2q$ and $s \geq p^r$,} \\ 0 & \text{otherwise.} \end{cases} \label{eq:quadrupleodd}
\end{align}
\end{theorem}

\begin{proof}
The proof is by induction on $s$. To help avoid getting lost in a sea of spectral sequence notation, we break the proof into four steps that we illustrate in Figures \ref{fig:step1}--\ref{fig:step5}. The figures are drawn for the case $p=5$ and $r=2$, but are representative (via an appropriate rescaling) of the general situation. The figures illustrate the information about each spectral sequence that has been explicitly determined, or can be deduced using the isomorphisms and equalities preceding the theorem, at the completion of that step of the proof. In each figure, a solid horizontal line indicates a row in which terms may be nonzero. On those lines, an open dot at the point $(s,t)$ means that the corresponding term is zero, a closed dot at $(s,t)$ means that the corresponding term is isomorphic to $k$, and a lack of any dot indicates that no information about that term has yet been determined. Arrows indicate that the corresponding differential has been determined to be an isomorphism. In the diagrams for $\Ebar_{2,0}$, a dashed line passes through terms of total degree $2p^r-1$.

\emph{Step 1.} Since \eqref{eq:rowisomorphisms} is valid for all $s \in \Z$, and since \eqref{eq:ellOprspecseq} and \eqref{eq:ellKprspecseq} are first quadrant spectral sequences, \eqref{eq:quadrupleeven} and \eqref{eq:quadrupleodd} are true for $s < 0$. Next, it follows from \eqref{eq:abutments} that $\Ebar_{2,0}^{0,0} = k$ and $\Ebar_{2,1}^{0,0} = 0$, and hence that \eqref{eq:quadrupleeven} and \eqref{eq:quadrupleodd} are true if $s = 0$. See Figure \ref{fig:step1}.

\emph{Step 2.} Since $\Ebar_{\infty,1} = 0$, it follows that $\Ebar_{2,1}^{s,0} \neq 0$ only if there exist integers $i,j$ with $i+j=s-1$ and $i \leq s-2$ such that $\Ebar_{2,1}^{i,j} \neq 0$. Now an induction argument using the information from Step 1 and the isomorphism $\Ebar_{2,1}^{s,0} \cong \Ebar_{2,1}^{s+q-1,q}$ shows that $\Ebar_{2,1}^{s,0} = 0$ for $s \leq p^r-1$. Then \eqref{eq:quadrupleodd} is true for $s \leq p^r-1$. See Figure \ref{fig:step2}.

\emph{Step 3.} Since $\Ebar_{\infty,0}^{s,t} = 0$ if $s+t \notin \{0,2p^r-1 \}$, and since $\Ebar_{2,0}^{s,(p-1)q} = \Ebar_{2,0}^{s,pq} = 0$ for $s \leq p^r+q-2$ by Step 2, it follows for $0 < s < 2p^r-1$ that the differential on the $(q+1)$-th page of \eqref{eq:ellOprspecseq} induces an isomorphism $\dbar_{q+1}: \Ebar_{2,0}^{s-q-1,q} \simrightarrow \Ebar_{2,0}^{s,0}$. Then an induction argument like that in the proof of \cite[(4.5.5)]{Friedlander:1997} shows that \eqref{eq:quadrupleeven} is true for $s < 2p^r-1$. See Figure \ref{fig:step3}.

\emph{Step 4.} Now combining the validity of \eqref{eq:quadrupleeven} for $s < 2p^r-1$ with \eqref{eq:Eisovanishing}, \eqref{eq:EEbariso}, and the validity of \eqref{eq:quadrupleodd} for $s \leq p^r-1$, it follows that all terms of total degree $2p^r-1$ in $E_{2,0}$ are zero except perhaps $E_{2,0}^{2p^r-1,0}$ and $E_{2,0}^{p^r+q-1,(p-1)q}$, and that $E_{2,0}^{2p^r-q-1,q-1} \cong k$ is the only nonzero term of total degree $2p^r-2$ in $E_{2,0}$. But $E_{\infty,0}$ is one-dimensional in total degree $2p^r-2$ by \eqref{eq:abutments}, so $E_{2,0}^{2p^r-q-1,q-1}$ must survive to the abutment, and hence the differential $d_{q+1}: E_{2,0}^{2p^r-q-1,q-1} \rightarrow E_{2,0}^{2p^r-1,0}$ must be trivial. Now since $E_{\infty,0} = 0$ in all odd total degrees, it follows that $E_{2,0}^{2p^r-1,0} = 0$. Then by \eqref{eq:EEbariso}, the statement \eqref{eq:quadrupleeven} is true for $s = 2p^r-1$ as well. See Figure \ref{fig:step4}.

\emph{Step 5.} We have established so far that \eqref{eq:quadrupleeven} is true for $s < 2p^r$, and that \eqref{eq:quadrupleodd} is true for $s < p^r$. Given $m \in \N$, set $u = p^r+2qm$ and set $v = u+p^r = 2p^r+2qm$. We now proceed by induction on $m$ to show that if \eqref{eq:quadrupleeven} is true for $s < v$ and if \eqref{eq:quadrupleodd} is true for $s < u$, then \eqref{eq:quadrupleeven} is true for $s < v+2q$ and \eqref{eq:quadrupleodd} is true for $s < u+2q$.

First, it follows from the induction hypothesis and \eqref{eq:Eisovanishing} and \eqref{eq:EEbariso} that the nonzero terms $E_{2,1}^{s,t}$ (resp.\ $E_{2,0}^{s,t}$) with $(p-1)q \leq t < pq$ that have been explicitly determined so far are in distinct even (resp.\ odd) total degrees. Similarly, the nonzero terms $E_{2,1}^{s,t}$ (resp.\ $E_{2,0}^{s,t}$) with $0 \leq t < q$ that have been explicitly determined so far are in distinct odd (resp.\ even) total degrees. Then since $E_{\infty,1} = 0$, there must be a nontrivial differential originating at the term $E_{2,1}^{u-(p-1)q-1,(p-1)q} \cong k$. By the induction hypothesis, $E_{2,1}^{s,t} = 0$ for all $s,t$ with $0 < t < q$ and $s+t=u$. Then the differential
\begin{equation} \label{eq:Kpr1differential}
d_{(p-1)q+1}: E_{2,1}^{u-(p-1)q-1,(p-1)q} \rightarrow E_{2,1}^{u,0}
\end{equation}
must be nontrivial. By the induction hypothesis there are no other nontrivial terms of total degree $u-1$ in $E_{2,1}$, so since $E_{\infty,1} = 0$, this differential must be an isomorphism. Then $k \cong E_{2,1}^{u,0} \cong \Ebar_{2,1}^{u,0}$, and by a similar argument as for Step 2, we deduce from the induction hypothesis that $\Ebar_{2,1}^{s,0} = 0$ for $u < s < u+2q$. Then \eqref{eq:quadrupleodd} is true for $s < u+2q$.

Applying \eqref{eq:rowisomorphisms}, \eqref{eq:Eisovanishing}, and  \eqref{eq:EEbariso}, we now get that $E_{2,0}^{u+q-1,(p-1)q} \cong k$. Reasoning as in the previous paragraph, we also see that this term cannot be the image of any nontrivial differential. But $E_{\infty,0}$ is zero in all total degrees $\geq 2p^r-1$, so we deduce that there must be a nontrivial differential originating at $E_{2,0}^{u+q-1,(p-1)q}$. By the induction hypothesis, $E_{2,0}^{s,t} = 0$ for all $s,t$ with $0 < t < q$ and $s+t=u+pq-1$. Then as in the previous paragraph, we deduce that the differential
\begin{equation} \label{eq:Kpr0differential}
d_{(p-1)q+1}: E_{2,0}^{u+q-1,(p-1)q} \rightarrow E_{2,0}^{u+pq,0} = E_{2,0}^{v,0}
\end{equation}
is an isomorphism. So $k \cong E_{2,0}^{v,0} \cong \Ebar_{2,0}^{v,0}$. Finally, again arguing as in Step 2, it follows from the induction hypothesis that $\Ebar_{2,0}^{s,0} = 0$ for $v < s < v+2q$ and hence that \eqref{eq:quadrupleeven} is true for $s < v+2q$; see Figure \ref{fig:step5}. This completes the induction argument laid out at the beginning of Step 5, and hence completes the inductive proof that \eqref{eq:quadrupleeven} and \eqref{eq:quadrupleodd} are true for all $s$; see Figure \ref{fig:endsteps}.
\end{proof}

\subsection{Module structure of \texorpdfstring{$\Ext_{\bsp}^\bullet(\bsir,{S_0}^{p^{r-1}(1)})$}{ExtP(Ir,S0pr-1(1))}} \label{subsec:modulej=1}

In this section we continue our investigation of the hypercohomology spectral sequences \eqref{eq:ellOprspecseq} and \eqref{eq:ellKprspecseq}, with the goal of describing the structure of $\Ext_{\bsp}^\bullet(\bsir,{S_0}^{p^{r-1}(1)})$ as a right module over the Yoneda algebra $\Ext_{\bsp}^\bullet(\bsir,\bsir)$.

\textbf{In this section}, set $q = p^{r-1}$.

\begin{lemma} \label{lem:differential}
In the spectral sequence $\Ebar_{2,0}^{s,t} \Rightarrow \bbExt_{\bsp}^{s+t}(\bsirzero,\bso_{p^r})$, the differential
\[
\dbar_{q+1}: \Ebar_{2,0}^{s-q-1,q} \rightarrow \Ebar_{2,0}^{s,0}
\]
is an isomorphism if $s \equiv 0 \mod 2p^{r-1}$ and $s > 0$, and hence is an isomorphism for all $s > 0$.
\end{lemma}

\begin{proof}
We defer the proof until Section \ref{subsec:fgpgeq3}, though we observe that the last claim follows from the first since by Theorem \ref{thm:quadruplecalculation} the only time $s > 0$ and the domain and codomain of the differential are not both $0$ is when $s \equiv 0 \mod 2p^{r-1}$. 
\end{proof}

\begin{remark} \label{rem:restrictionj=1}
Let $\Ebar$ be the spectral sequence considered in Lemma \ref{lem:differential}, and for the purposes of this remark only, let
\[
E_2^{s,t} = \Ext_{\cp}^s(I^{(r)},\Omega_q^{t(1)}) \Rightarrow \bbExt_{\cp}^{s+t}(I^{(r)},\Omega_{p^r})
\]
be the second hypercohomology spectral sequence in the category $\cp$ for the ordinary de Rham complex $\Omega_{p^r}$; this is one of the spectral sequences that was the focus of attention in the proof of the case $j=1$ of \cite[Theorem 4.5]{Friedlander:1997}. By general abstract nonsense, the restriction functor $\bsp \rightarrow \cp$ induces a homomorphism of spectral sequences $\Ebar \rightarrow E$ such that the induced map $\Ebar_2 \rightarrow E_2$ identifies with the restriction homomorphism described in Remark \ref{rem:restriction}. Then it follows from Remark \ref{rem:restrictcommute} that this homomorphism gives rise for each $s$ to a commutative diagram
\[
\xymatrix{
\Ebar_2^{s,0} \ar@{->}[r]^-{\sim} \ar@{->}[d] & \Ebar_2^{s+q-1,q} \ar@{->}[d] \ar@{->}[r]^-{\dbar_{q+1}} & \Ebar_2^{s+2q,0} \ar@{->}[d] \\
E_2^{s,0} \ar@{->}[r]^-{\sim} & E_2^{s+q-1,q} \ar@{->}[r]^-{d_{q+1}} & E_2^{s+2q,0}
}
\]
in which the vertical arrows are the corresponding restriction homomorphisms. The $p^{r-1}$-power map $\varphi_{r-1}: \bsirzero \rightarrow {S_0}^{p^{r-1}(1)}$ spans $\Hom_{\bsp}(\bsirzero,{S_0}^{p^{r-1}(1)})$, and its restriction spans $\Hom_{\cp}(I^{(r)},S^{p^{r-1}(1)})$, so the restriction map $\Ebar_2^{0,0} \rightarrow E_2^{0,0}$ is an isomorphism. Then applying Lemma \ref{lem:differential}, and using the fact from the $j=0$ case of \cite[(4.5.6)]{Friedlander:1997} that $d_{q+1}: E_2^{s-q-1,q} \rightarrow E_2^{s,0}$ is an isomorphism for all $0 < s < 2p^r$, it follows via induction that the restriction map
\[
\Ext_{\bsp}^s(\bsi_0^{(r)},S_0^{p^{r-1}(1)}) \rightarrow \Ext_{\cp}^s(I^{(r)},S^{p^{r-1}(1)})
\]
is an isomorphism for all $0 \leq s < 2p^r$.
\end{remark}

\begin{lemma} \label{lem:Kprdegreeshift}
Let $r,s \in \N$ with $r \geq 1$. There exist isomorphisms of right $\Ext_{\bsp}^\bullet(\bsir,\bsir)$-modules
\begin{align}
\Ext_{\bsp}^s(\bsir,S_1^{q(1)}) &\cong \Ext_{\bsp}^{s+q-1}(\bsir,\Lambda_1^{q(1)}) \cong \Ext_{\bsp}^{s+p^r}(\bsir,S_0^{q(1)}), \text{ and} \label{eq:prshift110} \\
\Ext_{\bsp}^s(\bsir,S_0^{q(1)}) &\cong \Ext_{\bsp}^{s+q-1}(\bsir,\Lambda_0^{q(1)}) \cong \Ext_{\bsp}^{s+p^r}(\bsir,S_1^{q(1)}). \label{eq:prshift001}
\end{align}
In particular, there exist isomorphisms of right $\Ext_{\bsp}^\bullet(\bsir,\bsir)$-modules
\begin{align}
\Ext_{\bsp}^s(\bsir,S_1^{q(1)}) &\cong \Ext_{\bsp}^{s+2p^r}(\bsir,S_1^{q(1)}), \text{ and} \label{eq:2prshiftodd} \\
\Ext_{\bsp}^s(\bsir,S_0^{q(1)}) &\cong \Ext_{\bsp}^{s+2p^r}(\bsir,S_0^{q(1)}). \label{eq:2prshifteven}
\end{align}
\end{lemma}

\begin{proof}
It suffices to prove \eqref{eq:prshift110}, since then \eqref{eq:prshift001} is defined to be the isomorphism obtained from \eqref{eq:prshift110} by conjugating by the parity change functor $\Pi$, and then \eqref{eq:2prshiftodd} and \eqref{eq:2prshifteven} follow from composing \eqref{eq:prshift110} and \eqref{eq:prshift001}. The first isomorphism in \eqref{eq:prshift110} holds by Corollary \ref{cor:degreeshift} and Remark \ref{rem:Yonedacompatible}, so it suffices to establish the second isomorphism in \eqref{eq:prshift110}.

Theorem \ref{thm:quadruplecalculation} implies that the nonzero terms in the $E_2$-page of \eqref{eq:ellKprspecseq} are in distinct total degrees, and the total degree of a nonzero term satisfies a parity condition depending on the term's row; cf.\ Step 5 of the proof of Theorem \ref{thm:quadruplecalculation}. Then since $E_{\infty,0}^{s,t} = 0$ if $s+t \geq 2p^r$, and since $E_{\infty,1} = 0$, it follows that each differential in \eqref{eq:ellKprspecseq} that can possibly be nontrivial must be nontrivial. Specifically, if $0 \leq t < q $ and $s \equiv 0 \mod 2q$ with $s \geq 2p^r$, then
\[
d_{(p-1)q+1}: E_{2,0}^{s-(p-1)q-1+t,(p-1)q+t} \rightarrow E_{2,0}^{s+t,t}
\]
must be an isomorphism, and similarly, if $s \equiv p^r \mod 2q$ and $s \geq p^r$, then
\[
d_{(p-1)q+1}: E_{2,1}^{s-(p-1)q-1+t,(p-1)q+t} \rightarrow E_{2,1}^{s+t,t}
\]
must be an isomorphism. Combining the $t=0$ cases of these two observations, it follows for all $s \geq 0$ that the differential
\[
d_{(p-1)q+1}: E_2^{s+q-1,(p-1)q} \rightarrow E_2^{s+p^r,0}
\]
in \eqref{eq:Kprspecseq} defines an isomorphism $\Ext_{\bsp}^{s+q-1}(\bsir,\Lambda_1^{q(1)}) \cong \Ext_{\bsp}^{s+p^r}(\bsir,S_0^{q(1)})$. This isomorphism is moreover compatible with the right action of $\Ext_{\bsp}^\bullet(\bsir,\bsir)$, since \eqref{eq:Kprspecseq} is a spectral sequence of right $\Ext_{\bsp}^\bullet(\bsir,\bsir)$-modules.
\end{proof}

In the rest of this section we apply Lemmas \ref{lem:differential} and \ref{lem:Kprdegreeshift}, and the results of Section \ref{subsec:vectorspacej=1}, to describe $\Ext_{\bsp}^\bullet(\bsir,{S_0}^{q(1)})$ as a right module over the Yoneda algebra $\Ext_{\bsp}^\bullet(\bsir,\bsir)$. We also make use of Corollaries \ref{cor:inductionresult} and \ref{cor:powermap}, which do not rely on the results of this section.

Observe that $\Ext_{\bsp}^\bullet(\bsir,S_0^{q(1)})$ identifies with the bottom row of the spectral sequence \eqref{eq:Oprspecseq}. Applying \eqref{eq:evendegreeshift}, the differential $\dbar_{q+1}: \Ebar_2^{s-q-1,q} \rightarrow \Ebar_2^{s,0}$ in \eqref{eq:Oprspecseq} identifies with an $\Ext_{\bsp}^\bullet(\bsir,\bsir)$-equivariant map
\begin{equation} \label{eq:dq+1}
\dbar_{q+1}: \Ext_{\bsp}^{s-2q}(\bsir,S_0^{q(1)}) \rightarrow \Ext_{\bsp}^s(\bsir,S_0^{q(1)}).
\end{equation}
Theorem \ref{thm:quadruplecalculation} and Corollary \ref{cor:powermap} imply that $\Ext_{\bsp}^\bullet(\bsir,S_0^{q(1)})$ is generated over $\Ext_{\bsp}^\bullet(\bsir,\bsir)$ by the $p^{r-1}$-power map $\varphi_{r-1}: \bsirzero \rightarrow S_0^{q(1)}$, which spans $\Hom_{\bsp}(\bsir,S_0^{q(1)})$. More specifically, we get that each $z' \in \Ext_{\bsp}^s(\bsir,S_0^{q(1)})$ can be uniquely expressed in the form $z' = \varphi_{r-1} \cdot z$ for some $z \in \Ext_{\bsp}^s(\bsirzero,\bsirzero)$ or $z \in \Ext_{\bsp}^s(\bsirone,\bsirzero)$. Now fix $\bse_r \in \Ext_{\bsp}^{2q}(\bsirzero,\bsirzero)$ such that
\begin{equation} \label{eq:erdefinition}
\dbar_{q+1}(\varphi_{r-1}) = \varphi_{r-1} \cdot \bse_r.
\end{equation}
Then \eqref{eq:dq+1} takes the form $\dbar_{q+1}(\varphi_{r-1} \cdot z) = \varphi_{r-1} \cdot \bse_r \cdot z$. In particular, applying Lemma \ref{lem:differential} we get by induction for all $\ell \geq 0$ that $\Ext_{\bsp}^{2q\ell}(\bsirzero,S_0^{q(1)})$ is spanned by $\varphi_{r-1} \cdot (\bse_r)^\ell$, and hence that $\Ext_{\bsp}^{2q\ell}(\bsirzero,\bsirzero)$ is spanned by $(\bse_r)^\ell$. (Even without Lemma \ref{lem:differential}, whose proof we have not yet given, Step 3 in the proof of Theorem \ref{thm:quadruplecalculation} shows that this holds for $0 \leq \ell < p$.)

Now define $\bsc_r \in \Ext_{\bsp}^{p^r}(\bsirone,\bsirzero)$ such that $\varphi_{r-1} \cdot \bsc_r$ is the image of $\varphi_{r-1}^\Pi = \Pi \circ \varphi_{r-1} \circ \Pi$ under the composite isomorphism \eqref{eq:prshift110}. In other words,
\begin{equation} \label{eq:crdefinition}
d_{(p-1)q+1} \left( (Kz_q^{(1)})^\Pi \cdot \varphi_{r-1}^\Pi \right) = \varphi_{r-1} \cdot \bsc_r.
\end{equation}
Then $\varphi_{r-1}^\Pi \cdot \bsc_r^\Pi = (\varphi_{r-1} \cdot \bsc_r)^\Pi$ is the image of $\varphi_{r-1}$ under \eqref{eq:prshift001}, and $\varphi_{r-1} \cdot (\bsc_r \cdot \bsc_r^\Pi)$ is the image of $\varphi_{r-1}$ under \eqref{eq:2prshifteven}. In particular, $\bsc_r \cdot \bsc_r^\Pi$ spans $\Ext_{\bsp}^{2p^r}(\bsirzero,\bsirzero)$. But by the previous paragraph, $(\bse_r)^p$ also spans $\Ext_{\bsp}^{2p^r}(\bsirzero,\bsirzero)$, so $\bsc_r \cdot \bsc_r^\Pi = \mu_r \cdot (\bse_r)^p$ for some nonzero scalar $\mu_r \in k$.

\begin{proposition} \label{prop:basesj=1}
Let $\bse_r$ and $\bsc_r$ be as defined in \eqref{eq:erdefinition} and \eqref{eq:crdefinition}. Then:
\begin{enumerate}
\item The set $\{ \varphi_{r-1} \cdot (\bse_r)^\ell: \ell \in \N \}$ is a basis for $\Ext_{\bsp}^\bullet(\bsi_0^{(r)},S_0^{p^{r-1}(1)})$.

\item The set $\{ \varphi_{r-1} \cdot (\bse_r)^\ell \cdot \bsc_r : \ell \in \N \}$ is a basis for $\Ext_{\bsp}^\bullet(\bsi_1^{(r)},S_0^{p^{r-1}(1)})$.
\end{enumerate}
\end{proposition}

\begin{proof}
We have already observed that $\varphi_{r-1} \cdot (\bse_r)^\ell$ spans $\Ext_{\bsp}^{2q\ell}(\bsirzero,S_0^{q(1)})$, so (1) is immediate by dimension comparison in each degree. Next, as observed in the paragraph preceding the proposition, the product $\bsc_r \cdot \bsc_r^\Pi$ is a nonzero scalar multiple of $(\bse_r)^p$. Then, up to a nonzero scalar factor, right multiplication by $\bsc_r^\Pi$ maps the set in (2) to a subset of the set in (1). Then the set in (2) must be linearly independent, and hence by dimension comparison be a basis for $\Ext_{\bsp}^\bullet(\bsirone,{S_0}^{p^{r-1}(1)})$.
\end{proof}

\begin{remark} \label{rem:e1extension}
Observe that restriction from $\bsp$ to $\cp$ gives rise to a commutative diagram
\[
\xymatrix{
\Ext_{\bsp}^s(\bsi_0^{(r)},\bsi_0^{(r)}) \ar@{->}[r] \ar@{->}[d] & \Ext_{\bsp}^s(\bsi_0^{(r)},S_0^{p^{r-1}(1)}) \ar@{->}[d] \\
\Ext_{\cp}^s(I^{(r)},I^{(r)}) \ar@{->}[r] & \Ext_{\cp}^s(I^{(r)},S^{p^{r-1}(1)})
}
\]
in which the vertical arrows are the restriction maps and the horizontal arrows are induced by the corresponding $p^{r-1}$-power maps. By Remark \ref{rem:restrictionj=1}, the right-hand vertical arrow is an isomorphism whenever the terms in the right-hand column are both nonzero. More precisely, the argument in Remark \ref{rem:restrictionj=1} implies that the restriction functor maps $\varphi_{r-1} \cdot \bse_r$ to the extension class denoted $\ol{e}_r$ in \cite[p.\ 244]{Friedlander:1997}. Then the commutativity of the diagram implies that the restriction functor sends $\bse_r$ to the distinguished extension class denoted $e_r$ in \cite[p.\ 244]{Friedlander:1997}, i.e., $\bse_r|_{\bsvzero} = e_r$. On the other hand, $\bsc_r|_{\bsvzero} = 0$ because $\bsirone|_{\bsvzero} = 0$.

Now let $\alpha: \bss \rightarrow \bsg$ be the unique $\bsp$-algebra homomorphism extending the identification $\bss^1 = \bsi = \bsg^1$. The restriction of $\alpha$ to $\bss^p$ fits into an exact sequence
\begin{equation} \label{eq:e1extension}
0 \rightarrow \bsi_0^{(1)} \rightarrow \bss^p \stackrel{\alpha}{\rightarrow} \bsg^p \rightarrow \bsi_0^{(1)} \rightarrow 0
\end{equation}
in which ${\bsi_0}^{(1)} \rightarrow \bss^p$ is the $p$-power map and $\bsg^p \rightarrow {\bsi_0}^{(1)}$ is the dual Frobenius map. By \cite[Lemma 4.12]{Friedlander:1997}, the restriction of \eqref{eq:e1extension} to $\bsvzero$ represents the element $e_1 \in \Ext_{\cp}^2(I^{(1)},I^{(1)})$. Since $\bse_1|_{\bsvzero} = e_1$, it follows that \eqref{eq:e1extension} is a representative extension for the cohomology class $\bse_1$ under the bijection $\theta_0$ of Proposition \ref{prop:extnandnextensions}. A representative extension for (a scalar multiple of) the cohomology class $\bsc_1$ is given in the proof of Lemma \ref{lem:c1restriction}.
\end{remark}

\begin{remark} \label{rem:twistedclasses}
Let $j \geq 1$. Recall from Section \ref{subsubsec:precomposition} that precomposition with $\bsi^{(j)}$ extends to an even linear map on extension groups in $\bsp$. We denote this map by $z \mapsto z^{(j)}$. Since for each $F \in \bsp$ one has $(F \circ \bsi^{(j)})|_{\bsvzero} = (F|_{\bsvzero}) \circ I^{(j)}$, it follows that the map $z \mapsto z^{(j)}$ is compatible with the restriction functor $\bsp \rightarrow \cp$. Now suppose $1 \leq i < r$. Then taking $j = r-i$, there exists a commutative diagram
\[
\xymatrix@C+1em{
\Ext_{\bsp}^{2p^{i-1}}(\bsi_0^{(i)},\bsi_0^{(i)}) \ar@{->}[r]^{z \mapsto z^{(r-i)}} \ar@{->}[d] & \Ext_{\bsp}^{2p^{i-1}}(\bsi_0^{(r)},\bsi_0^{(r)}) \ar@{->}[d] \\
\Ext_{\cp}^{2p^{i-1}}(I^{(i)},I^{(i)}) \ar@{->}[r]^{z \mapsto z^{(r-i)}} & \Ext_{\cp}^{2p^{i-1}}(I^{(r)},I^{(r)}).
}
\]
Moreover, the bottom arrow of this diagram is an injection by \cite[Corollary 4.9]{Friedlander:1997}. Since for $i \geq 1$ the extension class $\bse_i \in \Ext_{\bsp}^{2p^{i-1}}({\bsi_0}^{(i)},{\bsi_0}^{(i)})$ restricts to the nonzero class $e_i \in \Ext_{\cp}^{2p^{i-1}}(I^{(i)},I^{(i)})$ by Remark \ref{rem:e1extension}, it follows from the commutativity of the diagram that ${\bse_i}^{(r-i)} \neq 0$. On the other hand, ${\bsc_i}^{(r-i)} = 0$ because $\Ext_{\bsp}^{p^i}(\bsirone,\bsirzero) = 0$ by Theorem \ref{thm:quadruplecalculation} and the assumption $i < r$.
\end{remark}

\subsection{Vector space structure of \texorpdfstring{$\Ext_{\bsp}^\bullet(\bsir,{S_0}^{p^{r-j}(j)})$, $j \geq 1$}{ExtP(Ir,S0pr-j(j)), jgeq1}} \label{subsec:vectorspacejgeq1}

Recall for $n \in \N$ that $\Omega_n$ denotes the component of total degree $n$ of the ordinary de Rham complex functor $\Omega$, and $K_n$ denotes the Koszul kernel subfunctor of $\Omega_n$ defined in \cite[\S4]{Friedlander:1997}. Given $j \geq 1$, we use \eqref{eq:evenlift} to consider ${\Omega_n}^{(j)}$ and ${K_n}^{(j)}$ as strict polynomial superfunctors. Then for $1 \leq i \leq n$ one has
\begin{equation} \label{eq:Onij}
\Omega_n^{i(j)} = (S^{n-i} \otimes \Lambda^{i}) \circ \bsijzero = S_0^{n-i(j)} \otimes \Lambda_0^{i(j)}.
\end{equation}
Now let $j,r \in \N$ with $1 \leq j \leq r$. Our goal in this section is to describe the spectral sequences
\begin{align}
\Ebar_2^{s,t} &= \Ext_{\bsp}^s(\bsir,\opH^t(\Omega_{p^{r-j}}^{(j)})) \Rightarrow \bbExt_{\bsp}^{s+t}(\bsir,\Omega_{p^{r-j}}^{(j)}), \text{ and} \label{eq:Oprjspecseq} \\
E_2^{s,t} &= \Ext_{\bsp}^s(\bsir,\opH^t(K_{p^{r-j}}^{(j)})) \Rightarrow \bbExt_{\bsp}^{s+t}(\bsir,K_{p^{r-j}}^{(j)}) \label{eq:Kprjspecseq}
\end{align}
obtained by taking $A = \bsir$ and $C = \Omega_{p^{r-j}}^{(j)}$ or $C = K_{p^{r-j}}^{(j)}$ in \eqref{eq:secondhypercohomology}. As in Section \ref{subsec:vectorspacej=1}, we will exploit the fact that for $\ell \in \set{0,1}$, there exist spectral sequences 
\begin{align}
\Ebar_{2,\ell}^{s,t} &= \Ext_{\bsp}^s(\bsi_\ell^{(r)},\opH^t(\Omega_{p^{r-j}}^{(j)})) \Rightarrow \bbExt_{\bsp}^{s+t}(\bsi_\ell^{(r)},\Omega_{p^{r-j}}^{(j)}), \text{ and} \label{eq:ellOprjspecseq} \\
E_{2,\ell}^{s,t} &= \Ext_{\bsp}^s(\bsi_\ell^{(r)},\opH^t(K_{p^{r-j}}^{(j)})) \Rightarrow \bbExt_{\bsp}^{s+t}(\bsi_\ell^{(r)},K_{p^{r-j}}^{(j)}) \label{eq:ellKprjspecseq}
\end{align}
that are direct summands of \eqref{eq:Oprjspecseq} and \eqref{eq:Kprjspecseq}, respectively. One of the results of our investigation will be the following analogue of \cite[Theorem 4.5]{Friedlander:1997}, which computes the row $t=0$ of \eqref{eq:Oprjspecseq}:

\begin{theorem} \label{thm:constructionbyinduction}
Let $j,r \in \N$ with $1 \leq j \leq r$. Then
\begin{align*}
\Ext_{\bsp}^s(\bsi_0^{(r)},S_0^{p^{r-j}(j)}) \cong \Ext_{\bsp}^{s+p^{r-j}-1}(\bsi_0^{(r)},\Lambda_0^{p^{r-j}(j)}) &\cong \begin{cases} k & \text{if $s \equiv 0 \mod 2p^{r-j}$ and $s \geq 0$,} \\ 0 & \text{otherwise.} \end{cases} \\
\Ext_{\bsp}^s(\bsi_1^{(r)},S_0^{p^{r-j}(j)}) \cong \Ext_{\bsp}^{s+p^{r-j}-1}(\bsi_1^{(r)},\Lambda_0^{p^{r-j}(j)}) &\cong \begin{cases} k & \text{if $s \equiv p^r \mod 2p^{r-j}$ and $s \geq p^r$,} \\ 0 & \text{otherwise.} \end{cases}
\end{align*}
\end{theorem}

Theorem \ref{thm:constructionbyinduction} is true if $j=1$ by Theorem \ref{thm:quadruplecalculation} and \eqref{eq:Ebaridentification}, so for the rest of this section let us assume by way of induction that $1 \leq j < r$ and that Theorem \ref{thm:constructionbyinduction} is true for the given values of $j$ and $r$. Using this assumption, we can compute the abutments of \eqref{eq:Oprjspecseq} and \eqref{eq:Kprjspecseq}.

\begin{lemma} \label{lem:abutmentsinduction}
Let $j,r \in \N$ with $1 \leq j < r$. Then:
\begin{align*}
\bbExt_{\bsp}^s(\bsi_0^{(r)},K_{p^{r-j}}^{(j)}) &= \begin{cases} k & \text{if $s \geq 0$ is even,} \\ 0 & \text{otherwise.} \end{cases} \\
\bbExt_{\bsp}^s(\bsi_0^{(r)},\Omega_{p^{r-j}}^{(j)}) &= \begin{cases} k & \text{if $s \geq 0$ and $s \equiv 0 \mod 2p^{r-j}$ or $s \equiv -1 \mod 2p^{r-j}$,} \\ 0 & \text{otherwise.} \end{cases} \\
\bbExt_{\bsp}^s(\bsi_1^{(r)},K_{p^{r-j}}^{(j)}) &= \begin{cases} k & \text{if $s \geq p^r$ and $s$ is odd,} \\ 0 & \text{otherwise.} \end{cases} \\
\bbExt_{\bsp}^s(\bsi_1^{(r)},\Omega_{p^{r-j}}^{(j)}) &= \begin{cases} k & \text{if $s \equiv p^r \mod 2p^{r-j}$ and $s \geq p^r$,} \\ k & \text{if $s \equiv p^r-1 \mod 2p^{r-j}$ and $s \geq p^r-1+2p^{r-j}$,} \\ 0 & \text{otherwise.} \end{cases}
\end{align*}
\end{lemma}

\begin{proof}
In this proof set $q = p^{r-j}$ and let $\ell \in \{0,1\}$. First take $A = \bsilr$ and $C = {K_q}^{(j)}$ in \eqref{eq:firsthypercohomology}. If $s \geq q$, then $K_q^s = 0$, and hence $E_1^{s,t} = 0$. On the other hand, if $s < q$, then \eqref{eq:evendegreeshift} implies that
\[
E_1^{s,t} \cong E_1^{0,t-s} \cong \Ext_{\bsp}^{t-s}(\bsi_\ell^{(r)},S_0^{q(j)}).
\]
By assumption, Theorem \ref{thm:constructionbyinduction} is true for the given values of $r$ and $j$, so we can apply it via the above isomorphism to explicitly describe the vector space $E_1^{s,t}$. Specifically, if $0 \leq s < q$, then
\[
E_1^{s,s+t} = \begin{cases} k & \text{if $\ell = 0$, $t \geq 0$, and $t \equiv 0 \mod 2q$,} \\ k & \text{if $\ell = 1$, $t \geq p^r$, and $t \equiv p^r \mod 2q$,} \\ 0 & \text{otherwise.} \end{cases}
\]
In particular, the total degrees of any two nonzero terms in the $E_1$-page of the spectral sequence must be of the same parity. Then it follows that all differentials in the spectral sequence are zero, and hence that $E_1 \cong E_\infty$. But $E_\infty \cong \bbExt_{\bsp}^\bullet(\bsilr,{K_q}^{(j)})$, so the calculation of $\bbExt_{\bsp}^\bullet(\bsilr,{K_q}^{(j)})$ follows from the explicit calculation of the $E_1$-page as a total vector space.

Now take $A = \bsilr$ and $C = {\Omega_q}^{(j)}$ in \eqref{eq:firsthypercohomology}. Then applying Theorem \ref{thm:vanishing} and Theorem \ref{thm:constructionbyinduction} for the given values of $r$ and $j$, the calculation of $\bbExt_{\bsp}^\bullet(\bsirzero,{\Omega_q}^{(j)})$ follows from a repetition of the proof of \cite[(4.5.2)]{Friedlander:1997}. Specifically, Theorem \ref{thm:vanishing} implies that $E_1^{s,t} = 0$ unless $s = 0$ or $s=q$, and the differential between these two columns fits into an exact sequence
\[
0 \rightarrow E_\infty^{0,t} \rightarrow \Ext_{\bsp}^t(\bsi_\ell^{(r)},S_0^{q(j)}) \rightarrow \Ext_{\bsp}^{t+1-q}(\bsi_\ell^{(r)},\Lambda_0^{q(j)}) \rightarrow E_\infty^{q,t+1-q} \rightarrow 0.
\]
From Theorem \ref{thm:constructionbyinduction} it follows that the second and third terms in this exact sequence are never simultaneously nonzero.\footnote{This uses the assumption $r > j$ and its consequence $2q = 2p^{r-j} > 2$.} Then the differential between the columns $s=0$ and $s=q$ vanishes, which implies that $E_1 \cong E_\infty$. Now the calculation of $E_\infty \cong \bbExt_{\bsp}^\bullet(\bsilr,{\Omega_q}^{(j)})$ follows from Theorem \ref{thm:constructionbyinduction} and the above four-term exact sequence. In particular, if $\ell = 0$ and $s \equiv 0 \mod 2q$ with $s \geq 0$, or if $\ell = 1$ and $s \equiv p^r \mod 2q$ with $s \geq p^r$, then the edge maps
\[
\bbExt_{\bsp}^s(\bsi_\ell^{(r)},\Omega_q^{(j)}) \twoheadrightarrow E_\infty^{0,s} \hookrightarrow E_1^{0,s} \cong \Ext_{\bsp}^s(\bsi_\ell^{(r)},S_0^{q(j)})
\]
are isomorphisms of one-dimensional spaces.
\end{proof}

Our focus for the rest of this section is on analyzing the spectral sequences \eqref{eq:ellOprjspecseq} and \eqref{eq:ellKprjspecseq}. Whereas the super Cartier isomorphism does not preserve the cohomological degree, the ordinary Cartier isomorphism \cite[3.3]{Franjou:1994} induces isomorphisms of strict polynomial superfunctors
\[
\Hbul(\Omega_{p^{r-j}}^{(j)}) \cong \Omega_{p^{r-j-1}}^{\bullet(j+1)} \quad \text{and} \quad \Hbul(K_{p^{r-j}}^{(j)}) \cong K_{p^{r-j-1}}^{\bullet(j+1)}
\]
that preserve the cohomological degree. \textbf{For the rest of this section}, set $q = p^{r-j-1}$. Then
\[
\Ebar_{2,\ell}^{s,t} \cong \Ext_{\bsp}^s(\bsi_\ell^{(r)},\Omega_q^{t(j+1)}) \quad \text{and} \quad E_{2,\ell}^{s,t} \cong \Ext_{\bsp}^s(\bsi_\ell^{(r)},K_q^{t(j+1)}).
\]
Now \eqref{eq:Onij} and Theorem \ref{thm:vanishing} imply that $\Ebar_{2,\ell}^{s,t} = 0$ unless $t = 0$ or $t=q$. Next, the inclusion of complexes ${K_{p^{r-j}}}^{(j)} \hookrightarrow {\Omega_{p^{r-j}}}^{(j)}$ induces a map of spectral sequences $E \rightarrow \Ebar$ that on the $E_2$-page identifies with the map in cohomology induced by the inclusion ${K_q}^{(j+1)} \hookrightarrow {\Omega_q}^{(j+1)}$. In particular, since $K_q^0 = \Omega_q^0$, the induced map $E_{2,\ell}^{s,0} \rightarrow \Ebar_{2,\ell}^{s,0}$ is an isomorphism. The following lemma now follows from a word-for-word repetition of the proof of \cite[(4.5.4)]{Friedlander:1997}.

\begin{proposition} \label{prop:FS454}
In \eqref{eq:ellKprjspecseq}, all differentials to terms in the row $t=0$ are zero. Hence,
\[
E_{\infty,\ell}^{s,0} \cong E_{2,\ell}^{s,0} \cong \Ext_{\bsp}^s(\bsi_\ell^{(r)},S_0^{q(j+1)}).
\]
\end{proposition}

\begin{corollary} \label{cor:collapse}
On the $E_2$-page of \eqref{eq:ellKprjspecseq}, the total degrees of any two nonzero terms must be of the same parity. Consequently, all differentials in \eqref{eq:ellKprjspecseq} are zero, and $E_{2,\ell} \cong E_{\infty,\ell}$.
\end{corollary}

\begin{proof}
As in the proof of Lemma \ref{lem:abutmentsinduction}, we have $E_{2,\ell}^{s,t} = 0$ if $t \geq q$, and $E_{2,\ell}^{s,t} \cong E_{2,\ell}^{s-t,0}$ if $0 \leq t < q$. Now $E_{2,\ell}^{s-t,0} \cong E_{\infty,\ell}^{s-t,0}$ by Proposition \ref{prop:FS454}, and the total vector space structure of $E_{\infty,\ell}$ is given by Lemma \ref{lem:abutmentsinduction}. In particular, the total degrees of any two nonzero terms in $E_{\infty,\ell}$ must be of the same parity. Then the same conclusion also follows for $E_{2,\ell}$, whence the conclusion of the corollary. 
\end{proof}

We can now complete the proof of Theorem \ref{thm:constructionbyinduction}.

\begin{proof}[Proof of Theorem \ref{thm:constructionbyinduction}]
By assumption, the theorem is true for the given values of $r$ and $j$. Then by induction, it suffices to show that the theorem is also true for $j+1$. By Corollary \ref{cor:collapse} and its proof, we have $E_{2,\ell} \cong E_{\infty,\ell}$, $E_{2,\ell}^{s,0} \cong E_{2,\ell}^{s+t,t}$ for $0 \leq t < q$, and $E_{2,\ell}^{s,t} = 0$ for $t \geq q$. Also, $E_{2,\ell}^{s,0} = 0$ for $s < 0$. Then by Lemma \ref{lem:abutmentsinduction} and induction on $s$, we must have
\[
E_{2,\ell}^{s,0} \cong \Ext_{\bsp}^s(\bsi_\ell^{(r)},S_0^{q(j+1)}) \cong \begin{cases} k & \text{if $\ell = 0$, $s \geq 0$, and $s \equiv 0 \mod 2q$,} \\ k & \text{if $\ell = 1$, $s \geq p^r$, and $s \equiv p^r \mod 2q$,} \\ 0 & \text{otherwise,} \end{cases}
\]
in order for $\bigoplus_{u+v=s} E_{2,\ell}^{u,v}$ to have the same dimension as $\bigoplus_{u+v=s} E_{\infty,\ell}^{u,v} \cong \bbExt_{\bsp}^s(\bsilr,{K_{p^{r-j}}}^{(j)})$. With Lemma \ref{lem:degreeshift}, this completes the proof.
\end{proof}

\begin{corollary} \label{cor:inductionresult}
Let $r \geq 1$. Then
\begin{align*}
\Ext_{\bsp}^s(\bsi_1^{(r)},\bsi_1^{(r)}) \cong \Ext_{\bsp}^s(\bsi_0^{(r)},\bsi_0^{(r)}) &\cong \begin{cases} k & \text{if $s \geq 0$ is even,} \\ 0 & \text{otherwise.} \end{cases} \\
\Ext_{\bsp}^s(\bsi_0^{(r)},\bsi_1^{(r)}) \cong \Ext_{\bsp}^s(\bsi_1^{(r)},\bsi_0^{(r)}) &\cong \begin{cases} k & \text{if $s \geq p^r$ is odd,} \\ 0 & \text{otherwise.} \end{cases}
\end{align*}
\end{corollary}

\begin{proof}
The first isomorphism in each line follows from conjugating by $\Pi$, while the second in each line is true by the case $j=r$ of Theorem \ref{thm:constructionbyinduction}.
\end{proof}

\subsection{Module structure of \texorpdfstring{$\Ext_{\bsp}^\bullet(\bsir,{S_0}^{p^{r-j}(j)})$, $j \geq 1$}{ExtP(Ir,S0pr-j(j)), jgeq1}} \label{subsec:modulejgeq1}

Again let $j,r \in \N$ with $1 \leq j < r$, and set $q = p^{r-j-1}$. In this section we continue our investigation of the spectral sequence \eqref{eq:ellOprjspecseq}, with the goal of describing $\Ext_{\bsp}^\bullet(\bsir,{S_0}^{p^{r-j}(j)})$ as a right module over the Yoneda algebra $\Ext_{\bsp}^\bullet(\bsir,\bsir)$. The spectral sequence \eqref{eq:ellOprjspecseq} has only two nonzero rows,
\begin{equation} \label{eq:Oqjrows}
\begin{aligned}
\Ebar_{2,\ell}^{s,0} &\cong \Ext_{\bsp}^s(\bsi_\ell^{(r)},S_0^{q(j+1)}) & \text{and} \\
\Ebar_{2,\ell}^{s,q} &\cong \Ext_{\bsp}^s(\bsi_\ell^{(r)},\Lambda_0^{q(j+1)}) \cong \Ext_{\bsp}^{s-q+1}(\bsi_\ell^{(r)},S_0^{q(j+1)}),
\end{aligned}
\end{equation}
and hence only one nontrivial differential, which identifies with a map
\begin{equation} \label{eq:differentialjgeq1}
\dbar_{q+1}: \Ext_{\bsp}^{s-2q}(\bsi_\ell^{(r)},S_0^{q(j+1)}) \rightarrow \Ext_{\bsp}^s(\bsi_\ell^{(r)},S_0^{q(j+1)})
\end{equation}
that fits into a four-term exact sequence
\begin{equation} \label{eq:fourterm}
0 \rightarrow \Ebar_{\infty,\ell}^{s-q-1,q} \rightarrow \Ext_{\bsp}^{s-2q}(\bsi_\ell^{(r)},S_0^{q(j+1)}) \stackrel{\dbar_{q+1}}{\rightarrow} \Ext_{\bsp}^s(\bsi_\ell^{(r)},S_0^{q(j+1)}) \rightarrow \Ebar_{\infty,\ell}^{s,0} \rightarrow 0.
\end{equation}

Suppose $\ell = 0$. By Theorem \ref{thm:constructionbyinduction}, the second and third terms of this exact sequence are both isomorphic to $k$ if $s-2q \geq 0$ and $s \equiv 0 \mod 2q$; the second (and hence also the first) term is zero, but the third (and hence also the fourth) term is isomorphic to $k$ if $s = 0$; and the second and third terms (and hence the end terms as well) are zero for all other values of $s$. So suppose $s-2q \geq 0$ and $s \equiv 0 \mod 2q$. If $s \not\equiv 0 \mod 2p^{r-j}$, then Lemma \ref{lem:abutmentsinduction} implies that the end terms of \eqref{eq:fourterm} are both zero, and hence that the differential is an isomorphism. Now suppose $s \equiv 0 \mod 2p^{r-j}$. Since $\Ebar_{2,0}^{s-1,0} = 0 = \Ebar_{2,0}^{s-q,q}$, it follows from Lemma \ref{lem:abutmentsinduction} that the end terms of \eqref{eq:fourterm} must both be isomorphic to $k$, and hence the differential must be trivial. In particular, if $s \equiv 0 \mod 2p^{r-j}$, then the  maps $\Ebar_{2,0}^{s,0} \twoheadrightarrow \Ebar_{\infty,0}^{s,0} \hookrightarrow \bbExt_{\bsp}^s(\bsirzero,{\Omega_{p^{r-j}}}^{(j)})$ are isomorphisms. Combined with the observations from the end of the proof of Lemma \ref{lem:abutmentsinduction}, this implies for $s \equiv 0 \mod 2p^{r-j}$ that the composite map \eqref{eq:secondfirstcomposite} is an isomorphism. A similar analysis can also be applied if $\ell = 1$. We summarize the results of both analyses in the following theorem (cf.\ \cite[(4.5.6)]{Friedlander:1997}):

\begin{theorem} \label{thm:FS456}
Let $j,r \in \N$ with $1 \leq j < r$, and identify the nonzero rows of \eqref{eq:ellOprjspecseq} as in \eqref{eq:Oqjrows}.
\begin{enumerate}
\item Suppose $\ell = 0$. If $s \not\equiv 0 \mod 2p^{r-j}$, then \eqref{eq:differentialjgeq1} is an isomorphism. If $s \equiv 0 \mod 2p^{r-j}$, then \eqref{eq:differentialjgeq1} is trivial, and the $p$-power map $\varphi: {S_0}^{p^{r-j-1}(j+1)} \hookrightarrow {S_0}^{p^{r-j}(j)}$ induces an isomorphism
\[
\Ext_{\bsp}^s(\bsi_0^{(r)},S_0^{p^{r-j-1}(j+1)}) \simrightarrow \Ext_{\bsp}^s(\bsi_0^{(r)},S_0^{p^{r-j}(j)}).
\]

\item Suppose $\ell = 1$. If $s \not\equiv p^r \mod 2p^{r-j}$, then \eqref{eq:differentialjgeq1} is an isomorphism. If $s \equiv p^r \mod 2p^{r-j}$, then \eqref{eq:differentialjgeq1} is trivial, and $\varphi: {S_0}^{p^{r-j-1}(j+1)} \hookrightarrow {S_0}^{p^{r-j}(j)}$ induces an isomorphism
\[
\Ext_{\bsp}^s(\bsi_1^{(r)},S_0^{p^{r-j-1}(j+1)}) \simrightarrow \Ext_{\bsp}^s(\bsi_1^{(r)},S_0^{p^{r-j}(j)}).
\]

\item $(\ol{d}_{q+1})^p : \Ext_{\bsp}^{s-2p^{r-j}}(\bsi_\ell^{(r)},S_0^{q(j+1)}) \rightarrow \Ext_{\bsp}^s(\bsi_\ell^{(r)},S_0^{q(j+1)})$ is the zero map.
\end{enumerate}
\end{theorem}

Repeated application of Theorem \ref{thm:FS456} yields:

\begin{corollary} \label{cor:powermap}
Let $1 \leq j < r$. The $p^{r-j}$-power map $\varphi_{r-j}: \bsi_0^{(r)} \hookrightarrow S_0^{p^{r-j}(j)}$ induces isomorphisms
\begin{align*}
\Ext_{\bsp}^s(\bsi_0^{(r)},\bsi_0^{(r)}) &\simrightarrow \Ext_{\bsp}^s(\bsi_0^{(r)},S_0^{p^{r-j}(j)}) \quad \text{if $s \equiv 0 \mod 2p^{r-j}$, and} \\
\Ext_{\bsp}^s(\bsi_1^{(r)},\bsi_0^{(r)}) &\simrightarrow \Ext_{\bsp}^s(\bsi_1^{(r)},S_0^{p^{r-j}(j)}) \quad \text{if $s \equiv p^r \mod 2p^{r-j}$ and $s \geq p^r$}.
\end{align*}
\end{corollary}

Theorem \ref{thm:FS456} also implies the following direct analogue of \cite[Corollary 4.6]{Friedlander:1997}:

\begin{corollary} \label{cor:stableimpliesall}
Let $j,r \in \N$ with $1 \leq j < r$. Let $\ell \in \set{0,1}$, and let
\[
V \subseteq \Ext_{\bsp}^\bullet(\bsi_\ell^{(r)},S_0^{p^{r-j-1}(j+1)})
\]
be a graded subspace such that the $p$-power map $\varphi: S_0^{p^{r-j-1}(j+1)} \rightarrow S_0^{p^{r-j}(j)}$ induces a surjection from $V$ onto $\Ext_{\bsp}^\bullet(\bsilr,{S_0}^{p^{r-j}(j)})$, and such that $V$ is stable with respect to the endomorphism
\[
\dbar_{p^{r-j-1}+1}: \Ext_{\bsp}^{\bullet-2p^{r-j-1}}(\bsi_\ell^{(r)},S_0^{p^{r-j-1}(j+1)}) \rightarrow \Ext_{\bsp}^\bullet(\bsi_\ell^{(r)},S_0^{p^{r-j-1}(j+1)}).
\]
Then $V = \Ext_{\bsp}^\bullet(\bsilr,{S_0}^{p^{r-j-1}(j+1)})$.
\end{corollary}

\begin{proof}
We give the proof in the case $\ell=1$, the proof for $\ell = 0$ being entirely analogous. The first assumption on $V$ implies that $\Ext_{\bsp}^s(\bsirone,{S_0}^{p^{r-j-1}(j+1)}) \subseteq V$ for all $s \geq p^r$ with $s \equiv p^r \mod 2p^{r-j}$. Then the second assumption on $V$ implies that $\Ext_{\bsp}^s(\bsirone,{S_0}^{p^{r-j-1}(j+1)}) \subseteq V$ for all $s \geq p^r$ with $s \equiv p^r \mod 2p^{r-j-1}$, and hence that $V = \Ext_{\bsp}^\bullet(\bsirone,{S_0}^{p^{r-j-1}(j+1)})$.
\end{proof}

Recall the cohomology classes $\bse_1^{(r-1)},\bse_2^{(r-2)},\ldots,\bse_{r-1}^{(1)},\bse_r,\bsc_r$ that were defined in Section \ref{subsec:modulej=1}. We can now use these classes to describe a basis for $\Ext_{\bsp}^\bullet(\bsir,{S_0}^{p^{r-j}(j)})$. 

\begin{proposition} \label{prop:basisprj}
Let $j,r \in \N$ with $1 \leq j \leq r$. Then the set of monomials
\begin{equation} \label{eq:evenbasisprj}
\{ \varphi_{r-j} \cdot (\bse_{r-j+1}^{(j-1)})^{\ell_1} \cdots (\bse_{r-1}^{(1)})^{\ell_{j-1}} (\bse_r)^{\ell_j}: 0 \leq \ell_1,\ldots,\ell_{j-1} < p,\, \ell_j \geq 0 \}
\end{equation}
is a basis for $\Ext_{\bsp}^\bullet(\bsirzero,S_0^{p^{r-j}(j)})$, and the set of monomials
\begin{equation} \label{eq:oddbasisprj}
\{ \varphi_{r-j} \cdot (\bse_{r-j+1}^{(j-1)})^{\ell_1} \cdots (\bse_{r-1}^{(1)})^{\ell_{j-1}} (\bse_r)^{\ell_j} \cdot \bsc_r: 0 \leq \ell_1,\ldots,\ell_{j-1} < p,\, \ell_j \geq 0 \}
\end{equation}
is a basis for $\Ext_{\bsp}^\bullet(\bsirone,S_0^{p^{r-j}(j)})$.
\end{proposition}

\begin{proof}
The proof is by induction on $j$ in a direct generalization of the argument used for the proof of \cite[Corollary 4.7]{Friedlander:1997}. The base case of the induction argument is handled using Proposition \ref{prop:basesj=1}. Then the induction step is handled using Theorem \ref{thm:FS456}, Corollary \ref{cor:stableimpliesall}, and the twisting functor $F \mapsto F \circ \bsi^{(j)}$ in the same way that the proof of \cite[Corollary 4.7]{Friedlander:1997} uses \cite[Corollary 4.6]{Friedlander:1997} and the twisting functor $F \mapsto F \circ I^{(j)}$.
\end{proof}

The case $j=r$ of Proposition \ref{prop:basisprj} immediately gives:

\begin{corollary} \label{cor:basis}
Let $r \geq 1$. Then the set of monomials
\begin{equation}
\{ (\bse_1^{(r-1)})^{\ell_1} \cdots (\bse_{r-1}^{(1)})^{\ell_{r-1}} (\bse_r)^{\ell_r}: 0 \leq \ell_1,\ldots,\ell_{r-1} < p,\, \ell_r \geq 0 \}
\end{equation}
is a basis for $\Ext_{\bsp}^\bullet(\bsirzero,\bsirzero)$, and the set of monomials
\begin{equation}
\{ (\bse_1^{(r-1)})^{\ell_1} \cdots (\bse_{r-1}^{(1)})^{\ell_{r-1}} (\bse_r)^{\ell_r} \cdot \bsc_r: 0 \leq \ell_1,\ldots,\ell_{r-1} < p,\, \ell_r \geq 0 \}
\end{equation}
is a basis for $\Ext_{\bsp}^\bullet(\bsirone,\bsirzero)$.
\end{corollary}

\subsection{Algebra structure of \texorpdfstring{$\Ext_{\bsp}^\bullet(\bsir,\bsir)$}{ExtP(Ir,Ir)}} \label{subsec:algebrastructure}

Modulo certain structure constants that are equal to $\pm 1$, and modulo the proof of Lemma \ref{lem:differential} (which will be given in Section \ref{subsec:fgpgeq3}), we can now describe the multiplicative structure of the Yoneda algebra $\Ext_{\bsp}^\bullet(\bsir,\bsir)$. We make use of the operations on cohomology groups discussed in Sections \ref{subsubsec:dualityiso}--\ref{subsubsec:conjugationbyPi}. Define $\bse_0$ and $\bse_0^\Pi$ to be the identity elements of $\Hom_{\bsp}(\bsirzero,\bsirzero)$ and $\Hom_{\bsp}(\bsirone,\bsirone)$, respectively.

\begin{theorem} \label{thm:Yonedaalgebra}
Let $r \geq 1$. Then $\Ext_{\bsp}^\bullet(\bsir,\bsir)$ is generated as a $k$-algebra by extension classes $\bse_0',\bse_1',\ldots,\bse_r',\bsc_r,\bse_0'',\bse_1'',\ldots,\bse_r'',\bsc_r^\Pi$ such that, for each $1 \leq i \leq r$, $\bse_i'$ and $\bse_i''$ are nonzero scalar multiples of ${\bse_i}^{(r-i)}$ and $({\bse_i}^{(r-i)})^\Pi$, respectively, and such that only the following relations hold:
\begin{enumerate}
\item[(0)] $\bse_0'$ and $\bse_0''$ are orthogonal idempotents that sum to the identity.

\item $(\bse_r')^p = \bsc_r \cdot \bsc_r^\Pi$ and $(\bse_r'')^p = \bsc_r^\Pi \cdot \bsc_r$.

\item For each $1 \leq i < r$, $(\bse_i')^p = (\bse_i'')^p = 0$.

\item For each $0 \leq i \leq r$, $\bsc_r \cdot \bse_i' = \bse_i'' \cdot \bsc_r = \bsc_r^\Pi \cdot \bse_i'' = \bse_i' \cdot \bsc_r^\Pi = 0$.

\item For each $0 \leq i,j \leq r$, $\bse_i' \cdot \bse_j'' = \bse_j'' \cdot \bse_i' = 0$.

\item For each $0 \leq i \leq r$, there exists $\lambda_i \in \set{\pm 1}$, with $\lambda_0 = 1$, such that
\[
\bse_i' \cdot \bsc_r = \lambda_i \cdot (\bsc_r \cdot \bse_i'') \quad \text{and} \quad \bse_i'' \cdot \bsc_r^\Pi = \lambda_i \cdot (\bsc_r^\Pi \cdot \bse_i').
\]

\item The subalgebra generated by $\bse_0',\bse_1',\ldots,\bse_r',\bse_0'',\bse_1'',\ldots,\bse_r''$ is commutative.
\end{enumerate}
In particular, restriction from $\bsp$ to $\cp$ induces a surjection $\Ext_{\bsp}^\bullet(\bsirzero,\bsirzero) \rightarrow \Ext_{\cp}^\bullet(I^{(r)},I^{(r)})$.
\end{theorem}

\begin{proof}
Set $\bse_0' = \bse_0$ and $\bse_0'' = \bse_0^\Pi$. Then by the matrix ring decomposition \eqref{eq:matrixring} of $\Ext_{\bsp}^\bullet(\bsir,\bsir)$, $\bse_0'$ and $\bse_0''$ are orthogonal idempotents that sum to the identity. Next, it was observed in the paragraph preceding Proposition \ref{prop:basesj=1} that $\bsc_i \cdot \bsc_i^\Pi = \mu_i \cdot (\bse_i)^p$ for some nonzero scalar factor $\mu_i \in k$. By the standing assumption that $k$ is a perfect field, there exists a unique $p$-th root ${\mu_i}^{1/p}$ of $\mu_i$ in $k$. Now for each $1 \leq i \leq r$, set $\bse_i' = ({\mu_i}^{-1/p} \cdot \bse_i)^{(r-i)}$, and set $\bse_i'' = (\bse_i')^\Pi$. Then $(\bse_r')^p = \bsc_r \cdot \bsc_r^\Pi$ and $(\bse_r'')^p = [(\bse_r')^p]^\Pi = [\bsc_r \cdot \bsc_r^\Pi]^\Pi = \bsc_r^\Pi \cdot \bsc_r$, proving (1). We also get for $1 \leq i < r$ that $(\bse_i')^p = (\bsc_i \cdot \bsc_i^\Pi)^{(r-i)} = (\bsc_i)^{(r-i)} \cdot (\bsc_i^\Pi)^{(r-i)} = 0$ by Remark \ref{rem:twistedclasses}, and similarly that $(\bse_i'')^p = 0$, so (2) is true as well. Next, the relations in (3) and (4), and the fact that $\Ext_{\bsp}^\bullet(\bsir,\bsir)$ is generated as a $k$-algebra by the indicated extension classes, follows from the matrix ring decomposition \eqref{eq:matrixring}, Corollary \ref{cor:basis}, and from the conjugation action of $\Pi$. Now consider the anti-involutions $z \mapsto z^\#$ and $z \mapsto z^{\#\Pi} := (z^\#)^\Pi = (z^\Pi)^\#$ on $\Ext_{\bsp}^\bullet(\bsir,\bsir)$. By the uni-dimensionality of the ambient extension groups, it follows for each $1 \leq i \leq r$ that $(\bse_i')^\# = \pm \bse_i'$, $(\bse_i' \cdot \bsc_r)^{\#\Pi} = \pm \bse_i' \cdot \bsc_r$, and $(\bsc_r)^{\#\Pi} = \pm \bsc_r$. But $(\bse_i' \cdot \bsc_r)^{\#\Pi} = (\bsc_r)^{\#\Pi} \cdot (\bse_i')^{\#\Pi} = \pm \bsc_r \cdot \bse_i''$. Thus, $\bse_i' \cdot \bsc_r = \lambda_i \cdot (\bsc_r \cdot \bse_i'')$ for some $\lambda_i \in \set{\pm 1}$, and conjugating by $\Pi$ we get $\bse_i'' \cdot \bsc_r^\Pi = \lambda_i \cdot (\bsc_r^\Pi \cdot \bse_i')$, proving (5). Finally, restriction from $\bsp$ to $\cp$ induces an algebra homomorphism
\begin{equation} \label{eq:resrictiontoFS}
\Ext_{\bsp}^\bullet(\bsi_0^{(r)},\bsi_0^{(r)}) \rightarrow \Ext_{\cp}^\bullet(I^{(r)},I^{(r)}).
\end{equation}
Since $(\bse_i)^{(r-i)}$ restricts to $(e_i)^{(r-i)}$ by Remarks \ref{rem:e1extension} and \ref{rem:twistedclasses}, and since the classes ${e_1}^{(r-1)},\ldots,e_r$ generate $\Ext_{\cp}^\bullet(I^{(r)},I^{(r)})$ as an algebra, \eqref{eq:resrictiontoFS} is a surjection. Then by dimension comparison, \eqref{eq:resrictiontoFS} is an isomorphism in cohomological degrees $s < 2p^r$. Since $\Ext_{\cp}^\bullet(I^{(r)},I^{(r)})$ is a commutative algebra, this implies for all $1 \leq i,j \leq r$ that $\bse_i' \cdot \bse_j' = \bse_j' \cdot \bse_i'$. Conjugating by $\Pi$, we get $\bse_i'' \cdot \bse_j'' = \bse_j'' \cdot \bse_i''$ as well. This completes the proof of (6).
\end{proof}

\begin{remark}
We expect that the scalar factors $\mu_1,\ldots,\mu_r$ and $\lambda_1,\ldots,\lambda_r$ are probably all equal to $1$, and hence for each $1 \leq i \leq r$ that $\bse_i' = {\bse_i}^{(r-i)}$ and $\bse_i'' = ({\bse_i}^{(r-i)})^\Pi$.
\end{remark}

\section{Applications of the universal extension classes} \label{sec:applications}

In this section we present our main application of the extension classes $\bse_r$ and $\bsc_r$ exhibited in Section \ref{subsec:modulej=1}, namely, that the cohomology ring of a finite $k$-supergroup scheme is a finitely-generated $k$-superalgebra. The proof of this result involves a detailed analysis of how the classes $\bse_r$ and $\bsc_r$ restrict to the Frobenius kernel $GL(m|n)_1$ of $GL(m|n)$. Besides the cohomological finite-generation result, we obtain from this analysis a proof of Lemma \ref{lem:differential}, which is a key step in describing the multiplicative structure of $\Ext_{\bsp}^\bullet(\bsir,\bsir)$. We also obtain the curious result that the rational cohomology group $\opH^2(GL(m|n),k)$ is nonzero. We begin in Section \ref{subsec:recollections} by recalling some of our previous work on the cohomology of finite supergroup schemes.

\subsection{Recollections on the cohomology of finite supergroup schemes} \label{subsec:recollections}

An affine $k$-super\-group scheme $G$ is equivalent to the data of its coordinate superalgebra $k[G]$, a commutative $k$-Hopf superalgebra. The commutativity of $k[G]$ implies for each $r \in \N$ that the $p^r$-power map defines a Hopf superalgebra homomorphism $k[G]^{(r)} \rightarrow k[G]$. Then the $r$-th Frobenius kernel $G_r$ of $G$ is the scheme-theoretic kernel of the corresponding comorphism $F_G^r: G \rightarrow G^{(r)}$. In other words, $G_r$ is the affine $k$-supergroup scheme with coordinate superalgebra $k[G_r] = k[G]/(\sum_{f \in I_\ve} k[G] f^{p^r})$. Here $I_\ve$ denotes the augmentation ideal of $k[G]$. An affine $k$-super\-group scheme is \emph{algebraic} if $k[G]$ is a finitely-generated $k$-algebra, is \emph{finite} if $k[G]$ is finite-dimensional, and is \emph{infinitesimal} if it is finite and if $I_\ve$ is nilpotent. If $G$ is infinitesimal, then the minimal non-negative integer $r$ such that $f^{p^r} = 0$ for all $f \in I_\ve$ is the \emph{height} of $G$. For example, if $G$ is an affine $k$-supergroup scheme, then $G_r$ is infinitesimal of height $r$.

Let $G$ be a finite $k$-supergroup scheme. Then the dual superalgebra $k[G]^*$ is a finite-dimensional cocommutative Hopf superalgebra, and the category of left $G$-supermodules is isomorphic to the category of left $k[G]^*$-supermodules. Using the fact that the supertwist map makes $\bsv_\ev$ into an abelian braided monoidal category, one gets from \cite[Theorem 3.12]{Mastnak:2010} that the cohomology ring $\Hbul(G,k) \cong \Hbul(k[G]^*,k)$ is a graded-commutative superalgebra, i.e., if $w \in \opH^i(G,k)$ and $z \in \opH^j(G,k)$, then $z \cdot w = (-1)^{ij} (-1)^{\ol{w} \cdot \ol{z}} w \cdot z$. More generally, let $M$ be a trivial $G$-supermodule. Then there exist natural even isomorphisms
\begin{equation} \label{eq:classtolinearmap}
\opH^i(G,M) \cong \opH^i(G,k) \otimes M \cong \Hom_k(M^*,\opH^i(G,k)).
\end{equation}
Considering a homogeneous element $z \in \opH^i(G,M)$ as a linear map $z: M^* \rightarrow \opH^i(G,k)$ of the same parity, the graded-commutativity of $\Hbul(G,k)$ then implies that $z$ extends uniquely to an even homomorphism of graded superalgebras
\begin{align}
z &:\bss(M^*(i)) \rightarrow \Hbul(G,k), \quad \text{if $i$ is even, or} \label{eq:eveninducedhomomorphism} \\
z &: \bsl(M^*(i)) \rightarrow \Hbul(G,k), \quad \text{if $i$ is odd.} \label{eq:oddinducedhomomorphism}
\end{align}
Here $\bss(V(i))$ and $\bsl(V(i))$ denote the ``$i$-rarefactions'' of the graded superspaces $\bss(V)$ and $\bsl(V)$, i.e., $\bss^{ij}(V(i)) = \bss^j(V)$ for each $j\in \N$, while $\bss^j(V(i)) = 0$ if $j \notin i\Z$, and similarly for $\bsl(V(i))$.

In \cite[\S\S5.3--5.4]{Drupieski:2013b}, we considered for each finite $k$-supergroup scheme $G$ and each finite-dimensional $G$-supermodule $M$ the following questions:
\begin{enumerate}[label=(\thesubsection.\arabic*)]
\setcounter{enumi}{\value{equation}}
\item Is $\Hbul(G,k)$ a finitely-generated $k$-algebra?
\item Is $\Hbul(G,M)$ finitely-generated under the cup product action of $\Hbul(G,k)$?
\setcounter{equation}{\value{enumi}}
\end{enumerate}
In \cite[Theorem 5.3.3]{Drupieski:2013b}, we showed that if the answers to these questions are yes whenever $G$ is an infinitesimal $k$-supergroup scheme, then the answers are yes for every finite $k$-supergroup scheme. In \cite[Theorem 5.4.2]{Drupieski:2013b}, we showed further that if $G$ is infinitesimal of height $r$, then the answers to the above two questions are yes provided that there exists a closed embedding $G \hookrightarrow GL(m|n)_r$ for some $m,n \in \N$, and if for these values of $m$ and $n$ there exist certain conjectured cohomology classes $e_r^{m,n}$ and $c_r^{m,n}$ for $GL(m|n)$ whose restrictions to $GL(m|n)_1$ admit particular descriptions. To precisely state the required conditions for $e_r^{m,n}|_{G_1}$ and $c_r^{m,n}|_{G_1}$, we first recall the May spectral sequence constructed in \cite[\S5.2]{Drupieski:2013b}.

Let $G$ be an algebraic $k$-supergroup scheme, and let $\g = \Lie(G)$ be the Lie superalgebra of $G$. Then by \cite[Corollary 5.2.3]{Drupieski:2013b}, there exists a spectral sequence of rational $G$-supermodules
\begin{equation} \label{eq:Mayspecseq}
E_0^{i,j} = \bsl^j(\g^*) \otimes S^i(\gzero^*(2))^{(1)} \Rightarrow \opH^{i+j}(G_1,k).
\end{equation}
Here $S(\gzero^*(2))^{(1)}$ denotes the $2$-rarefaction of $S(\gzero^*)$, considered as a rational $G$-supermodule via the Frobenius morphism $F_G: G \rightarrow G^{(1)}$. Identifying $\bsl(\g^*)$ with the graded tensor product of algebras $\Lambda(\gzero^*) \gotimes S(\gone^*)$, we get from \cite[Proposition 3.5.3 and Lemma 5.2.2]{Drupieski:2013b} that the subalgebra $S(\gone^*)^p$ of $\bsl(\g^*)$ consisting of the $p$-th powers in $S(\gone^*)$ is a subalgebra of permanent cycles in $E_0$, and that $S(\gone^*)^p \cong S(\gone^*(p))^{(1)}$ as graded $G$-supermodules.

Now let $m$ and $n$ be positive integers, let $G = GL(m|n)$ be the corresponding general linear supergroup, and let $\g = \glmn = \Hom_k(k^{m|n},k^{m|n})$ be the Lie superalgebra of $G$. The conjectured classes $e_r^{m,n}$ and $c_r^{m,n}$ mentioned above are cohomology classes
\[
e_r^{m,n} \in \opH^{2p^{r-1}}(GL(m|n),\glzero^{(r)}) \quad \text{and} \quad
c_r^{m,n} \in \opH^{p^r}(GL(m|n),\glone^{(r)})
\]
whose restrictions $e_r^{m,n}|_{G_1}$ and $c_r^{m,n}|_{G_1}$ admit the following descriptions:
\begin{enumerate}[label=(\thesubsection.\arabic*)]
\setcounter{enumi}{\value{equation}}
\item The $G$-supermodule homomorphism $e_r^{m,n}|_{G_1} : S(\gzero^*(2p^{r-1}))^{(r)} \rightarrow \Hbul(G_1,k)$ induced by \eqref{eq:eveninducedhomomorphism} is equal to the composition of the $p^{r-1}$-power map $S(\gzero^*(2p^{r-1}))^{(r)} \rightarrow S(\gzero^*(2))^{(1)}$ with the horizontal edge map $S(\gzero^*(2))^{(1)} \rightarrow \Hbul(G_1,k)$ of \eqref{eq:Mayspecseq}. \label{eq:ertoG1}

\item The composition of the $G$-supermodule homomorphism $c_r^{m,n}|_{G_1}: S(\gone^*(p^r))^{(r)} \rightarrow \Hbul(G_1,k)$ induced by \eqref{eq:oddinducedhomomorphism} with the vertical edge map of \eqref{eq:Mayspecseq} has image equal to the subalgebra of $\bsl(\g^*)$ generated by all $p^r$-th powers in $S(\gone^*) \subset \bsl(\g^*)$. \label{eq:crtoG1}
\setcounter{equation}{\value{enumi}}
\end{enumerate}

Our goal is to show that the extension classes $\bse_r$, $\bse_r^\Pi$, $\bsc_r$, and $\bsc_r^\Pi$ exhibited in Section \ref{subsec:modulej=1} provide in a natural way the conjectured classes $e_r^{m,n}$ and $c_r^{m,n}$. Specifically, evaluation on the superspace $k^{m|n}$ defines an exact functor from $\bsp$ to the category of rational $G$-supermodules.\footnote{For each $d \in \N$, evaluation on $k^{m|n}$ defines an exact functor from $\bsp_d$ to the category of finite-dimensional supermodules for the Schur superalgebra $S(m|n,d)$, which is a quotient of the superalgebra $\Dist(G)$ of distributions for the supergroup $G = GL(m|n)$; cf.\ \cite[Theorem 3.2]{Brundan:2003a} and \cite[\S2]{El-Turkey:2013}. Specifically, the image of the evaluation functor lies in the subcategory of integrable $\Dist(G)$-supermodules. Since the category of integrable $\Dist(G)$-supermodules is isomorphic to the category of rational $G$-supermodules \cite[Corollary 3.5]{Brundan:2003a}, we thus get an exact functor from $\bsp_d$ to the category of rational $G$-supermodules.} This functor then induces for each $T,T' \in \bsp$ a natural even linear map $\Ext_{\bsp}^\bullet(T,T') \rightarrow \Ext_G^\bullet(T(k^{m|n}),T'(k^{m|n}))$, denoted $z \mapsto z|_G$. Observe that
\begin{equation} \label{eq:Zgradingglmn}
\begin{aligned}
\g_m &= \Hom_k(k^{m|0},k^{m|0}), & \g_{+1} &= \Hom_k(k^{0|n},k^{m|0}),\\
\g_n &= \Hom_k(k^{0|n},k^{0|n}), & \g_{-1} &= \Hom_k(k^{m|0},k^{0|n})
\end{aligned}
\end{equation}
are each naturally subspaces of $\g = \glmn$, with $\g_m \oplus \g_n = \gzero$, and $\g_{-1} \oplus \g_{+1} = \gone$. Restricting $\bse_r$, $\bse_r^\Pi$, $\bsc_r$, and $\bsc_r^\Pi$ to $G$, we get cohomology classes
\[
\begin{aligned}
\bse_r|_G &\in \Ext_G^{2p^{r-1}}(k^{m|0(r)},k^{m|0(r)}) &&\cong \opH^{2p^{r-1}}(G,\g_m^{(r)}), \\
\bse_r^\Pi|_G &\in \Ext_G^{2p^{r-1}}(k^{0|n(r)},k^{0|n(r)}) &&\cong \opH^{2p^{r-1}}(G,\g_n^{(r)}), \\
\bsc_r|_G &\in \Ext_G^{p^{r+1}}(k^{0|n(r)},k^{m|0(r)}) &&\cong \opH^{p^{r+1}}(G,\g_{+1}^{(r)}), \text{ and} \\
\bsc_r^\Pi|_G &\in \Ext_G^{p^{r+1}}(k^{m|0(r)},k^{0|n(r)}) &&\cong \opH^{p^{r+1}}(G,\g_{-1}^{(r)}).
\end{aligned}
\]
Our goal is to show that, up to rescaling, $(\bse_r + \bse_r^\Pi)|_G \in \opH^{2p^{r-1}}(G,\gzero)$ provides the conjectured class $e_r^{m,n}$, and $(\bsc_r + \bsc_r^\Pi)|_G \in \opH^{p^r}(G,\gone)$ provides the conjectured class $c_r^{m,n}$. We begin in Section \ref{subsec:Xg} with some additional preliminary details relevant to height-$1$ infinitesimal supergroup schemes. In Sections \ref{subsec:restrictionc1} and \ref{subsec:restrictioncr} we show for each $r \geq 1$ that $\bsc_r$ and $\bsc_r^\Pi$ restrict nontrivially to $G_1$. Then in Section \ref{subsec:fgpgeq3} we verify the conditions \ref{eq:ertoG1} and \ref{eq:crtoG1}.

\subsection{Some recollections regarding \texorpdfstring{$X(\g)$}{X(g)}} \label{subsec:Xg}

Let $G$ be a height-$1$ infinitesimal $k$-super\-group scheme. Then $\Hbul(G,k)$ identifies with the cohomology ring $\Hbul(V(\g),k)$ for the restricted enveloping superalgebra $V(\g)$ of $\g = \Lie(G)$ (cf.\ \cite[Remark 4.4.3]{Drupieski:2013b}), and $\Hbul(V(\g),k)$ can be computed using the projective resolution $X(\g)$ of $k$ discussed in \cite[\S3.3]{Drupieski:2013b}; see also \cite{Iwai:1965} and \cite[\S6]{May:1966}. We will need a few specific details concerning this resolution. First, write $U(\g)$ for the universal enveloping superalgebra of $\g$. Set $Y(\g) = U(\g) \otimes \bsa(\g)$, and set $W(\g) = V(\g) \otimes \bsa(\g)$. We consider $Y(\g)$ and $W(\g)$ as homologically graded superspaces with $U(\g)$ and $V(\g)$ concentrated in homological degree $0$ and $\bsa^i(\g)$ in homological degree $i$. Then $Y(\g)$ identifies with the Koszul resolution for $\g$ as discussed in \cite[\S3.1]{Drupieski:2013b}. The differential on $Y(\g)$ induces a differential on $W(\g)$, which we denote by $\partial$. Then $\partial$ makes $W(\g)$ into a differential graded superalgebra.

As a graded superspace, $X(\g) = W(\g) \otimes \Gamma(\gzero(2))$. The differential $d$ on $X(\g)$ is constructed in terms of a fixed homogeneous basis for $\g$ and a linear map $t: \Gamma(\gzero(2))^{(1)} \rightarrow W(\gzero)$ of homological degree $-1$. The map $t$ is constructed to satisfy the following properties:
	\begin{itemize}
	\item If $\ve: W(\g) \rightarrow k$ denotes the natural augmentation map on $W(\g)$, then $\ve \circ t = 0$.

	\item If $x$ is one of the fixed basis vectors for $\gzero$, then $t(\gamma_1(x)) = x^{p-1} \subgrp{x} - \subgrp{x^{[p]}}$. Here $x^{[p]}$ is the image of $x$ under the $p$-map making $\gzero$ a restricted Lie algebra, $x^{p-1}$ is the obvious monomial in $V(\g)$, and $\subgrp{x}$ and $\subgrp{x^{[p]}}$ are the obvious monomials in $\Lambda^1(\gzero) \subset \bsa^1(\g) \subset W^1(\g)$.
	
	\item If $\gzero$ is abelian, then $t$ can be constructed so that $t(\Gamma^i(\gzero(2))) \equiv 0$ for $i > 1$; cf.\ \cite[Remark 3.3.2]{Drupieski:2013b}. (In the notation of that remark, if $\gzero$ is abelian, then $r_2 = 0$ regardless of whether or not the restriction mapping is trivial.)
	\end{itemize}
Now the action of $d$ on $X(\g)$ is described in terms of the differential $\partial: W(\g) \rightarrow W(\g)$, the algebra structure of $W(\g)$, and the map $t$ as follows: Let $w \in W(\g)$ and $\gamma \in \Gamma(\gzero(2))$ be homogeneous elements. Denote the homological degree of $w$ by $\deg(w)$. Recall that $\Gamma(\gzero(2))$ is naturally a coalgebra, and write $\Delta_\Gamma(\gamma) = \sum \gamma' \otimes \gamma''$ for the coproduct of $\gamma$. Then
\begin{equation} \label{eq:dtdifferential} \textstyle
d(w \otimes \gamma) = \partial(w) \otimes \gamma + (-1)^{\deg(w)} \sum [w \cdot t(\gamma')] \otimes \gamma''.
\end{equation}

Though not discussed in \cite{Drupieski:2013b}, there exists a diagonal approximation $\Delta_s: X(\g) \rightarrow X(\g) \otimes X(\g)$ defined in terms of the natural coproducts $\Delta_W$ and $\Delta_\Gamma$ on $W(\g)$ and $\Gamma(\gzero(2))$, the algebra structure of $W(\g)$, and a linear map $s: \Gamma(\gzero(2)) \rightarrow W(\gzero) \otimes W(\gzero)$ of homological degree $0$; the map $s$ is called a \emph{twisting diagonal cochain} in \cite{Iwai:1965}, and is called a \emph{$t$-twisting coproduct} in \cite{May:1966}. We will not go into the details of the particular properties that $s$ must satisfy, but given the map $s$, and given $w \in W(\g)$ and $\gamma \in \Gamma(\gzero(2))$ as before, $\Delta_s$ is defined by
\begin{equation} \label{eq:Xgcoproduct} \textstyle
\Delta_s(w \otimes \gamma) = \sum [\Delta_W(w) \cdot s(\gamma')] \cdot \Delta_\Gamma(\gamma'').
\end{equation}

Now $\Hbul(V(\g),k)$ can be computed as the cohomology of the cochain complex $\Hom_{V(\g)}(X(\g),k)$. Applying the isomorphisms of Section \ref{subsec:duality}, there exists an isomorphism of graded superspaces
\begin{equation} \label{eq:HomXg}
\Hom_{V(\g)}(X(\g),k) \cong \Hom_k(\bsa(\g) \otimes \Gamma(\gzero(2)),k) \cong \bsl(\g^*) \otimes S(\gzero^*(2)).
\end{equation}
The diagonal approximation $\Delta_s$ induces a (typically) nonassociative product on the cochain complex $\Hom_{V(\g)}(X(\g),k)$. In particular, the induced product on cochains does not agree in general with the natural superalgebra structure of the tensor product $\bsl(\g^*) \otimes S(\gzero^*(2))$. However, using the fact that the twisting diagonal cochain $s$ has image in the subalgebra $W(\gzero) \otimes W(\gzero)$ of $W(\g) \otimes W(\g)$, one can show that, when restricted to the subspace $S(\gone^*) \otimes S(\gzero^*(2))$ of $\bsl(\g^*) \otimes S(\gzero^*(2))$, the induced product on cochains does agree with the natural superalgebra structure on $S(\gone^*) \otimes S(\gzero^*(2))$.

\begin{example} \label{ex:twodim}
Let $\g$ be the two-dimensional restricted Lie superalgebra generated by an even element $x$ and an odd element $y$ such that $[x,y] = 0$, $[y,y] = 2x$, and $x^{[p]} = x$. Then a typical monomial in $X(\g)$ has the form $x^a y^b \subgrp{x^c} \gamma_d(y) \gamma_e(x)$ for some $a,b,c,d,e \in \N$. (For legibility we have omitted the tensor product symbol between the factors $W(\g)$ and $\Gamma(\gzero(2))$.) Here $x^a y^b$ represents a monomial in $V(\g)$, $\subgrp{x^c}$ represents a monomial in $\Lambda^c(\gzero)$, $\gamma_d(y)$ represents a monomial in $\Gamma^d(\gone)$, and $\gamma_e(x)$ represents a monomial in $\Gamma^{2e}(\gzero(2))$, i.e., $\gamma_e(x) \in \Gamma^e(\gzero)$. Now since $\gzero$ is abelian and $x^{[p]} = x$, the differential $d: X(\g) \rightarrow X(\g)$ takes the form
\begin{align*}
d \left( x^a y^b \subgrp{x^c} \gamma_d(y) \gamma_e(x) \right) &= \partial \left(x^a y^b \subgrp{x^c} \gamma_d(y) \right) \gamma_e(x) \\
&\phantom{=} + (-1)^{c+d}\left[(x^a y^b \subgrp{x^c} \gamma_d(y)) \cdot (x^{p-1}\subgrp{x} - \subgrp{x}) \right] \gamma_{e-1}(x) \\
&= x^{a+1} y^b \subgrp{x^{c-1}} \gamma_d(y) \gamma_e(x) \\
&\phantom{=} + (-1)^c x^ay^{b+1} \subgrp{x^c}\gamma_{d-1}(y)\gamma_e(x) \\
&\phantom{=} - x^ay^b\subgrp{x^{c+1}}\gamma_{d-2}(y)\gamma_e(x) \\
&\phantom{=} + (-1)^c x^{a+p-1}y^b \subgrp{x^{c+1}}\gamma_d(y)\gamma_{e-1}(x) \\
&\phantom{=} -(-1)^c x^a y^b \subgrp{x^{c+1}}\gamma_d(y) \gamma_{e-1}(x).
\end{align*}

Let $x^*,y^*$ be the dual basis for $\g$. Then as above, a typical monomial in $\bsl(\g^*) \otimes S(\gzero^*(2))$ can be written in the form $\subgrp{(x^*)^c}(y^*)^d(x^*)^e$ for some $c,d,e \in \N$, where $\subgrp{(x^*)^c} \in \Lambda^c(\gzero^*)$, $(y^*)^d \in S^d(\gone^*)$, and $(x^*)^e \in S^{2e}(\gzero^*(2))$, i.e., $(x^*)^e \in S^e(\gzero^*)$. Making the identification \eqref{eq:HomXg}, one can check that the differential $d^*$ on $\Hom_{V(\g)}(X(\g),k)$ takes the following form:
\begin{align*}
(y^*)^d(x^*)^e &\stackrel{d^*}{\mapsto} 0 \\
\subgrp{(x^*)}(y^*)^d(x^*)^e &\stackrel{d^*}{\mapsto} (-1)^d (y^*)^d(x^*)^{e+1} -(-1)^d(y^*)^{d+2}(x^*)^e.
\end{align*}
This implies that $\Hbul(V(\g),k) \cong k[y^*,x^*]/\subgrp{x^*-(y^*)^2}$. In particular, no element of the subspace $S(\gzero^*(2))$ of $\Hom_{V(\g)}(X(\g),k)$ is a coboundary.
\end{example}

\subsection{Restriction of \texorpdfstring{$\bsc_1$}{c1}} \label{subsec:restrictionc1}

In this section, fix positive integers $m$ and $n$, let $G = GL(m|n)$ be the corresponding general linear supergroup, and let $G_1$ be its first Frobenius kernel.

\begin{lemma} \label{lem:c1restriction}
$\bsc_1|_{G_1} \neq 0$ and $\bsc_1^\Pi|_{G_1} \neq 0$.
\end{lemma}

\begin{proof}
Given $1 \leq i,j \leq m+n$, write $e_{i,j}$ for the corresponding $(m+n) \times (m+n)$ matrix unit whose $(i,j)$-entry is equal to $1$ and whose other entries are $0$. Let $U$ be the one-dimensional odd subsupergroup scheme of $G$ such that for each commutative $k$-superalgebra $A$,
\begin{equation} \label{eq:subgroupU}
U(A) = \set{I_{m+n} + a \cdot e_{m,m+1}: a \in A_{\ol{1}}},
\end{equation}
where $I_{m+n}$ denotes the $(m+n) \times (m+n)$ identity matrix. Then $U$ is a subsupergroup scheme of $G_1$. We will show that $\bsc_1|_{G_1} \neq 0$ by showing that $\bsc_1|_U \neq 0$. First, consider the super Koszul kernel complex $\bsK_p$. Proposition \ref{prop:Cartierkernel} implies that $\opH^0(\bsK_p) \cong {\bsi_0}^{(1)}$, $\opH^{p-1}(\bsK_p) \cong {\bsi_1}^{(1)}$, and $\opH^i(\bsK_p) = 0$ for $i \notin \set{0,p-1}$. Then by Proposition \ref{prop:extnandnextensions}, the augmented complex
\[
\bsK': 0 \rightarrow \bsi_0^{(1)} \rightarrow \bsK_p^0 \rightarrow \bsK_p^1 \rightarrow \cdots \rightarrow \bsK_p^{p-2} \rightarrow \bsK_p^{p-1} \rightarrow \bsi_1^{(1)} \rightarrow 0
\]
represents an element $\wt{\bsc}_1$ of the one-dimensional space $\Ext_{\bsp}^p({\bsi_1}^{(1)},{\bsi_0}^{(1)})$. In particular, $\wt{\bsc}_1$ is a scalar multiple of $\bsc_1$. We will show that $\wt{\bsc}_1|_U \neq 0$, and hence that $\bsc_1|_U \neq 0$.

Let $x_1,\ldots,x_m$ and $y_1,\ldots,y_n$ be the standard bases for $k^m$ and $k^n$, respectively. Then we consider $x_1,\ldots,x_m,y_1,\ldots,y_n$ as a homogeneous basis for $k^{m|n}$, with $\ol{x}_i = \ol{0}$ and $\ol{y}_j = \ol{1}$ for each $i$ and $j$. Let $\Lambda(z)$ be the ordinary exterior algebra over a one-dimensional odd superspace spanned by the vector $z$. The category of left $U$-supermodules is isomorphic to the category of left supermodules for $k[U]^* \cong \Lambda(z)$. Under this isomorphism, the action of $\Lambda(z)$ on $k^{m|n}$ is defined by $z.x_i = 0$ and $z.y_j = \delta_{j,1} \cdot x_m$. This equivalence also induces an isomorphism
\begin{equation} \label{eq:equivalenceiso}
\Ext_U^\bullet({\bsi_1}^{(1)}(k^{m|n}),{\bsi_0}^{(1)}(k^{m|n})) \cong \Ext_{\Lambda(z)}^\bullet(k^{0|n(1)},k^{m|0(1)}),
\end{equation}
where $k^{0|n(1)}$ and $k^{m|0(1)}$ are considered as trivial $\Lambda(z)$-supermodules. A projective resolution of the trivial $\Lambda(z)$-module $k$ is given by the complex $P_\bullet$, where $P_i = \Lambda(z)$ for each $i \geq 0$, and for $i \geq 1$ the differential $P_i \rightarrow P_{i-1}$ is multiplication by $z$. Then $P'_\bullet := P_\bullet \otimes k^{0|n(1)}$ is a projective resolution of $k^{0|n(1)}$, and the image of $\wt{\bsc}_1|_U$ under the isomorphism \eqref{eq:equivalenceiso} is computed by lifting the identity map $k^{0|n(1)} \rightarrow k^{0|n(1)}$ to a chain map $\varphi: P' \rightarrow \bsK'(k^{m|n})$, and then taking the cohomology class in $\Hom_{\Lambda(z)}(P_\bullet',k^{m|0(1)})$ of the cocycle $\varphi_p: P_p' \rightarrow k^{m|0(1)}$.

For each $i \geq 0$, $P_i'$ is a free $\Lambda(z)$-supermodule of rank $n$ with a homogeneous $\Lambda(z)$-basis given by the vectors $1 \otimes {y_j}^{(1)}$ for $1 \leq j \leq n$. By definition, $\bsK_p^i(k^{m|n})$ is a subspace of $\bss^{p-i}(k^{m|n}) \otimes \bsa^i(k^{m|n})$. We express elements of $\bss(k^{m|n})$ and $\bsa(k^{m|n})$ in terms of the bases \eqref{eq:bssbasis} and \eqref{eq:bsabasis}. Then the action of $z$ on $\bss(k^{m|n}) \otimes \bsa(k^{m|n})$ is given by
\[
\begin{aligned}
z. &\left[(x_1^{a_1} \cdots x_m^{a_m} y_1^{b_1} \cdots y_n^{b_n}) \otimes (x_1^{c_1} \cdots x_m^{c_m} \gamma_{d_1}(y_1) \cdots \gamma_{d_n}(y_n)) \right] \\
&= b_1 \cdot \left[(x_1^{a_1} \cdots x_m^{a_m+1} y_1^{b_1-1} \cdots y_n^{b_n}) \otimes (x_1^{c_1} \cdots x_m^{c_m} \gamma_{d_1}(y_1) \cdots \gamma_{d_n}(y_n)) \right] \\
&\phantom{=}+(-1)^{b_1+\cdots+b_n} \left[(x_1^{a_1} \cdots x_m^{a_m} y_1^{b_1} \cdots y_n^{b_n}) \otimes (x_1^{c_1} \cdots x_m^{c_m+1} \gamma_{d_1-1}(y_1) \cdots \gamma_{d_n}(y_n)) \right].
\end{aligned}
\]
We now define a family of $\Lambda(z)$-supermodule homomorphisms $\varphi_i: P_i' \rightarrow \bsK_p^{p-i-1}(k^{m|n})$ as follows: For $0 \leq i \leq p-1$, define $\varphi_i$ such that $\varphi_i(1 \otimes {y_j}^{(1)}) = \delta_{j,1} \cdot \frac{1}{i!} \cdot [x_m^i y_1 \otimes \gamma_{p-i-1}(y_1)]$. Define $\varphi_p$ such that $\varphi_p(1 \otimes {y_j}^{(1)}) = \delta_{j,1} \cdot \frac{1}{(p-1)!} \cdot {x_m}^{(1)}  = -\delta_{j,1} \cdot {x_m}^{(1)} \in k^{m|0(1)}$, and set $\varphi_i = 0$ for $i > p$. Then one can check that this family of homomorphisms defines a chain map $\varphi: P' \rightarrow \bsK'(k^{m|n})$ lifting the identity map $k^{0|n(1)} \rightarrow k^{0|n(1)}$. Since $\varphi_p$ is nonzero, and since the differential on the complex $\Hom_{\Lambda(z)}(P_\bullet',k^{m|0(1)})$ is trivial, it follows that $\wt{\bsc}_1|_U \neq 0$, and hence $\bsc_1|_{G_1} \neq 0$. The proof that $\bsc_1^\Pi|_{G_1} \neq 0$ is entirely analogous, replacing $U$ by the transpose subsupergroup $U'$ of $G_1$ defined by $U'(A) = \set{I_{m+n}+a \cdot e_{m+1,m} : a \in A_{\ol{1}}}$.
\end{proof}

\subsection{Restriction of \texorpdfstring{$\bsc_r$}{cr}} \label{subsec:restrictioncr}

Again, fix positive integers $m$ and $n$, and set $G = GL(m|n)$. Let $U$ be be as defined in \eqref{eq:subgroupU}, and let $\Hbul(U,k)$ denote the cohomology ring $\Ext_U^\bullet(k,k) \cong \Ext_{\Lambda(z)}^\bullet(k,k)$, which is known to be a polynomial algebra generated by an odd generator $\beta$ of cohomological degree $1$. For each $r \geq 1$ and for each pair of trivial $U$-supermodules $V^{(r)}$ and $W^{(r)}$, there exists a natural identification
\begin{equation} \label{eq:matrixidentification}
\Ext_U^\bullet(V^{(r)},W^{(r)}) \cong \Hbul(U,k) \otimes \Hom_k(V,W)^{(r)}.
\end{equation}
Taking $r=1$, the proof of Lemma \ref{lem:c1restriction} shows that, up to a nonzero scalar factor, $\bsc_1|_U$ identifies with $\beta^p \cdot \alpha^{(1)}$, where $\alpha: k^n \rightarrow k^m$ is the linear map defined with respect to the standard bases for $k^n$ and $k^m$ by the $m \times n$ unit matrix $e_{m,1}$. Similarly, replacing $U$ by its transpose $U'$, $\bsc_1^\Pi|_{U'}$ identifies with a nonzero scalar multiple of $\beta^p \cdot \alpha^{t(1)}$, where $\alpha^t: k^m \rightarrow k^n$ is the transpose of $\alpha$. Our goal in this section is to extend these identifications to all $r \geq 1$. Specifically, we will show by induction on $r$ that, under the identification \eqref{eq:matrixidentification}, $\bsc_r|_U$ identifies with a nonzero scalar multiple of $\beta^{p^r} \cdot \alpha^{(r)}$, and by symmetry that $\bsc_r^\Pi|_{U'}$ identifies with a nonzero scalar multiple of $\beta^{p^r} \cdot \alpha^{t(r)}$, and hence that $\bsc_r|_{G_1} \neq 0$ and $\bsc_r^\Pi|_{G_1} \neq 0$. Our approach in this section is strongly influenced by the methods of Franjou, Friedlander, Scorichenko, and Suslin \cite{Franjou:1999}.

Before delving into details, let us first indicate the general strategy of the induction argument. First, consider $\bsirone$ and $\bsirzero$ as subfunctors of ${\Gamma_1}^{(r)} := \bsa \circ \bsirone$ and ${S_0}^{(r)} := \bss \circ \bsirzero$, respectively. Then for each $i \geq 1$, we get the $i$-fold cup product $\bsc_r^{\cup i} \in \Ext_{\bsp}^{p^r i}(\Gamma_1^{i(r)},S_0^{i(r)})$. Assuming by way of induction that $\bsc_r|_U$ identifies with a nonzero scalar multiple of $\beta^{p^r} \cdot \alpha^{(r)}$, it follows that $(\bsc_r^{\cup i})|_U$ identifies with a nonzero scalar multiple of $\beta^{p^r i} \cdot {\alpha_i}^{(r)}$, where $\alpha_i$ denotes the composite map
\[
\Gamma^i(k^n) = ((k^n)^{\otimes i})^{\fS_i} \hookrightarrow (k^n)^{\otimes i} \stackrel{\alpha^{\otimes i}}{\longrightarrow} (k^m)^{\otimes i} \twoheadrightarrow S^i(k^m).
\]
In particular, $(\bsc_r^{\cup i})|_U \neq 0$, hence $\bsc_r^{\cup i} \neq 0$. One immediately checks that $\alpha_p$ equals the composite
\[
\Gamma^p(k^n) \stackrel{\varphi^\#}{\longrightarrow} (k^n)^{(1)} \stackrel{\alpha^{(1)}}{\longrightarrow} (k^m)^{(1)} \stackrel{\varphi}{\longrightarrow} S^p(k^m),
\]
where $\varphi^\#$ and $\varphi$ denote the dual Frobenius map and the $p$-power map, respectively. Then $\beta^{p^{r+1}} \cdot {\alpha_p}^{(r)}$ is the image of $\beta^{p^{r+1}} \cdot \alpha^{(r+1)}$ under the map in cohomology
\begin{equation} \label{eq:Uphistarphi}
(\varphi^\#,\varphi): \Ext_U^{p^{r+1}}(k^{n(r+1)},k^{m(r+1)}) \rightarrow \Ext_U^{p^{r+1}}(\Gamma^p(k^{n(r)}),S^p(k^{m(r)}))
\end{equation}
induced in the obvious way by $\varphi^\#$ and $\varphi$. Our plan is to show that the corresponding homomorphism between extension groups in $\bsp$,
\begin{equation} \label{eq:bspphistarphi}
(\varphi^\#,\varphi): \Ext_{\bsp}^{p^{r+1}}(\bsi_1^{(r+1)},\bsi_0^{(r+1)}) \rightarrow \Ext_{\bsp}^{p^{r+1}}(\Gamma_1^{p(r)},S_0^{p(r)}),
\end{equation}
maps $\bsc_{r+1}$ to a nonzero scalar multiple of $\bsc_r^{\cup p}$; cf. \cite[Corollary 5.9]{Franjou:1999}. Then, using the fact that \eqref{eq:Uphistarphi} is an injection, which can be seen from \eqref{eq:matrixidentification} by the fact that $\varphi^\#: \Gamma^p(k^n) \rightarrow k^{n(1)}$ is a surjection and $\varphi: k^{m(1)} \rightarrow S^p(k^m)$ is an injection, and using the compatibility of $(\varphi^\#,\varphi)$ with the restriction map $z \mapsto z|_U$, it follows that $\bsc_{r+1}|_U$ identifies with a nonzero multiple of $\beta^{p^{r+1}} \cdot {e_{m,1}}^{(r+1)}$.

The homomorphism \eqref{eq:bspphistarphi} factors as a composite of homomorphisms
\begin{equation} \label{eq:phistarphicomposite}
\Ext_{\bsp}^{p^{r+1}}(\bsi_1^{(r+1)},\bsi_0^{(r+1)}) \stackrel{(\varphi^\#)^*}{\longrightarrow} \Ext_{\bsp}^{p^{r+1}}(\Gamma_1^{p(r)},\bsi_0^{(r+1)}) \stackrel{\varphi_*}{\longrightarrow} \Ext_{\bsp}^{p^{r+1}}(\Gamma_1^{p(r)},S_0^{p(r)})
\end{equation}
induced first by $\varphi^\#$ and then by $\varphi$. Corollary \ref{cor:powermap} implies via conjugation by $\Pi$ and via duality that $(\varphi^\#)^*$ is an isomorphism in cohomological degree $p^{r+1}$. Then to complete the induction argument, it suffices to show that the image of $\varphi_*$ is the subspace spanned by $\bsc_r^{\cup p}$. First, recall the map \eqref{eq:secondfirstcomposite} that relates the first and second hypercohomology spectral sequences; taking $A = \Gamma_1^{p(r)}$ and $C = {\Omega_p}^{(r)}$, this map identifies with $\varphi_*$. We will prove that $\bsc_r^{\cup p} \in \im(\varphi_*)$ by analyzing \eqref{eq:secondfirstcomposite}.

\begin{proposition} \label{prop:imphistar}
In \eqref{eq:phistarphicomposite}, the image of $\varphi_*$ is the subspace spanned by $\bsc_r^{\cup p}$.
\end{proposition}

\begin{proof}
Set $V = \Ext_{\bsp}^\bullet(\bsirone,\bsirzero)$. Taking $A = {\Gamma_1}^{(r)} := \bsa \circ \bsirone$, and taking $B = {S_0}^{(r)} := \bss \circ \bsirzero$ or $B = {\Lambda_0}^{(r)} := \bsa \circ \bsirzero$, respectively, the cup product \eqref{eq:ABcupproduct} induces linear maps
\[
V^{\otimes d} \rightarrow \Ext_{\bsp}^\bullet(\Gamma_1^{d(r)},S_0^{d(r)}), \quad \text{and} \quad V^{\otimes d} \rightarrow \Ext_{\bsp}^\bullet(\Gamma_1^{d(r)},\Lambda_0^{d(r)}),
\]
By Lemma \ref{lem:cupcommute}, these maps factor, respectively, through linear maps
\[
\Theta_S^d : S^d(V) \rightarrow \Ext_{\bsp}^\bullet(\Gamma_1^{d(r)},S_0^{d(r)}), \quad \text{and} \quad \Theta_\Lambda^d : \Lambda^d(V) \rightarrow \Ext_{\bsp}^\bullet(\Gamma_1^{d(r)},\Lambda_0^{d(r)}).
\]
As an inductive tool for analyzing the terms appearing in \eqref{eq:secondfirstcomposite}, we will first prove by induction on $d$ that $\Theta_S^d$ and $\Theta_\Lambda^d$ are isomorphisms for $1 \leq d < p$. The case $d=1$ is a tautology, so suppose that $1 < d < p$. Then the composite homomorphism
\[
S_0^{d(r)} \rightarrow (S_0^{1(r)})^{\otimes d} = (\bsi_0^{(r)})^{\otimes d} \rightarrow S_0^{d(r)}
\]
induced by the coproduct and product morphisms for $S_0^{(r)}$ is equal to multiplication by the nonzero scalar $d!$. This implies that the map in cohomology
\[
\Ext_{\bsp}^\bullet(\Gamma_1^{d(r)},(\bsi_0^{(r)})^{\otimes d}) \rightarrow \Ext_{\bsp}^\bullet(\Gamma_1^{d(r)},S_0^{d(r)})
\]
induced by the product in ${S_0}^{(r)}$ is a surjection, and hence by \eqref{eq:firstcupiso} that $\Theta_S^d$ is a surjection as well. Replacing ${S_0}^{(r)}$ with ${\Lambda_0}^{(r)}$, we get by the same reasoning that $\Theta_\Lambda^d$ is a surjection. Then to prove that $\Theta_S^d$ and $\Theta_\Lambda^d$ are isomorphisms, it suffices to show that they are injections.

Given a cochain complex $C$ in $\bsp_\ev$, write $E(d,C)$ for the corresponding first hypercohomology spectral sequence obtained by taking $A = \Gamma_1^{d(r)}$ in \eqref{eq:firsthypercohomology}. As in \cite[Proposition 4.1]{Franjou:1999}, if $C'$ and $C''$ are cochain complexes in $\bsp_\ev$, then the cup product isomorphism \eqref{eq:firstcupiso} induces an isomorphism of spectral sequences $E(d',C') \otimes E(d'',C'') \cong E(d'+d'',C' \otimes C'')$. In particular, using \eqref{eq:evenlift} to consider the $r$-th Frobenius twist of the ordinary de Rham complex $\Omega$ as a complex of strict polynomial superfunctors, the cup product induces an isomorphism $E(1,{\Omega_1}^{(r)})^{\otimes d} \cong E(d,({\Omega_1}^{(r)})^{\otimes d})$. We then get a homomorphism of spectral sequences $E(1,{\Omega_1}^{(r)})^{\otimes d} \rightarrow E(d,{\Omega_d}^{(r)})$ by composing with the map of chain complexes $({\Omega_1}^{(r)})^{\otimes d} \rightarrow {\Omega_d}^{(r)}$ induced by the multiplication morphism for $\Omega$.

Set $\Ebar = E(1,{\Omega_1}^{(r)})$, and set $E = E(d,{\Omega_d}^{(r)})$. The only nonzero columns in $\Ebar_1$ are $\Ebar_1^{0,\bullet} \cong V$ and $\Ebar_1^{1,\bullet} \cong V$. With these identifications, the differential $\dbar_1: \Ebar_1^{0,\bullet} \rightarrow \Ebar_1^{1,\bullet}$ is the identity. Next, using \eqref{eq:firstcupiso} we can write
\[
E_1^{s,\bullet} \cong \Ext_{\bsp}^\bullet(\Gamma_1^{d(r)},S_0^{d-s(r)} \otimes \Lambda_0^{s(r)}) \cong \Ext_{\bsp}^\bullet(\Gamma_1^{d-s(r)},S_0^{d-s(r)}) \otimes \Ext_{\bsp}^\bullet(\Gamma_1^{s(r)},\Lambda_0^{s(r)}).
\]
Thus, if $0 < s < d$, then $E_1^{s,\bullet} \cong S^{d-s}(V) \otimes \Lambda^s(V)$ by induction on $d$. It now follows that the homomorphism of spectral sequences $(\Ebar)^{\otimes d} \rightarrow E$ factors on the $E_1$-page as the composition of two maps, $\rho: (\Ebar_1)^{\otimes d} \rightarrow \Omega_d(V)$ and $\sigma: \Omega_d(V) \rightarrow E_1$, such that:
\begin{itemize}
\item $\rho$ is induced by the identifications $\Ebar_1^{0,\bullet} = S^1(V)$ and $\Ebar_1^{1,\bullet} = \Lambda^1(V)$, and by multiplication in $\Omega(V)$; and
\item $\sigma$ restricts for each $s$ to a map $\sigma_s: \Omega_d^s(V) \rightarrow E_1^{s,\bullet}$ such that $\sigma_0 = \Theta_S^d$, $\sigma_d = \Theta_\Lambda^d$, and for $0 < s < d$, $\sigma_s$ is the isomorphism $S^{d-s}(V) \otimes \Lambda^s(V) \cong E_1^{s,\bullet}$.
\end{itemize}
Using the above factorization, it follows that $\sigma$ fits into a commutative diagram
\[
\vcenter{\xymatrix{
S^d(V) \ar@{->}[r] \ar@{->}[d]_{\Theta_S^d} & S^{d-1}(V) \otimes \Lambda^1(V) \ar@{->}[r] \ar@{->}[d]^{\cong}_{\sigma_1} & \cdots \ar@{->}[r] & S^1(V) \otimes \Lambda^{d-1}(V) \ar@{->}[r] \ar@{->}[d]^{\cong}_{\sigma_{d-1}} & \Lambda^d(V) \ar@{->}[d]_{\Theta_\Lambda^d} \\
E_1^{0,\bullet} \ar@{->}[r]^{d_1} & E_1^{1,\bullet} \ar@{->}[r]^{d_1} & \cdots \ar@{->}[r]^{d_1} & E_1^{d-1,\bullet} \ar@{->}[r]^{d_1} & E_1^{d,\bullet}
}}
\]
in which the top row is the de Rham complex $\Omega_d(V)$. In other words, $\sigma$ defines a map of cochain complexes. Since $\Omega_d(V)$ is an exact complex unless $d \equiv 0 \mod p$, this implies, by considering the left-most square in the diagram, that $\Theta_S^d$ is an injection, hence an isomorphism. The proof that $\Theta_\Lambda^d$ is an isomorphism is completely parallel to the argument given for $\Theta_S^d$: replace the de Rham complex $\Omega$ with the Koszul complex $Kz$, and interchange the roles of $S$ and $\Lambda$.

Now suppose $d=p$. Consider momentarily the second hypercohomology spectral sequence
\[
{}^{II}E_2^{s,t} = \Ext_{\bsp}^s(\Gamma_1^{p(r)},\opH^t(\Omega_p^{(r)})) \Rightarrow \bbExt_{\bsp}^{s+t}(\Gamma_1^{p(r)},\Omega_p^{(r)}).
\]
Applying the ordinary Cartier isomorphism, Theorem \ref{thm:constructionbyinduction}, conjugation by $\Pi$, and duality,
\[
{}^{II}E_2^{s,t} \cong \Ext_{\bsp}^s(\Gamma_1^{p(r)},\bsi_0^{(r+1)}) \cong \begin{cases} k & \text{if $s \equiv p^{r+1} \mod 2p$, $s \geq p^{r+1}$, and $t \in \set{0,1}$,} \\ 0 & \text{otherwise.} \end{cases}
\]
This implies that all differentials in the spectral sequence are trivial, and that the composite
\[
{}^{II}E_2^{p^{r+1},0} \twoheadrightarrow {}^{II}E_\infty^{p^{r+1},0} \hookrightarrow \bbExt_{\bsp}^{p^{r+1}}(\Gamma_1^{p(r)},\Omega_p^{(r)})
\]
is an isomorphism of one-dimensional spaces. Next consider the first hypercohomology spectral sequences $E := E(p,{\Omega_p}^{(r)})$ and $\bsE := E(p,{Kz_p}^{(r)})$. Then $\bsc_r^{\cup p} \in E_1^{0,p^{r+1}}$. By playing the spectral sequences $E$ and $\bsE$ off each other, we will show that $\bsc_r^{\cup p}$ spans the space of permanent cycles in $E_1^{0,p^{r+1}}$, and hence that the composite
\[
\bbExt_{\bsp}^{p^{r+1}}(\Gamma_1^{p(r)},\Omega_p^{(r)}) \twoheadrightarrow E_\infty^{0,p^{r+1}} \hookrightarrow E_1^{0,p^{r+1}}
\]
is an isomorphism of one-dimensional spaces, finishing the proof of the theorem.

As in the case $1 \leq d < p$, we get by induction on $d$ the existence of chain maps $\sigma: \Omega_p(V) \rightarrow E_1$ and $\tau: Kz_p(V) \rightarrow \bsE_1$ such that
\begin{itemize}
\item $\sigma$ satisfies the same properties as in the case $1 \leq d < p$,
\item the restriction $\tau_0: \Lambda^p(V) = Kz_p^0(V) \rightarrow \bsE_1^{0,\bullet}$ is equal to $\Theta_\Lambda^p$,
\item the restriction $\tau_p: S^p(V) = Kz_p^p(V) \rightarrow \bsE_1^{p,\bullet}$ is equal to $\Theta_S^p$, and
\item for $0 < s < p$, the restriction $\tau_s: Kz_p^s(V) \rightarrow \bsE_1^{s,\bullet}$ is an isomorphism.
\end{itemize}
The exactness of $Kz_p(V)$ then implies that $\bsE_2^{s,\bullet} = 0$ for $1 \leq s \leq p-2$, and that $\Theta_\Lambda^p$ is an injection. We also get $\bsE_2^{p-1,\bullet} \cong \ker(\Theta_S^p)$ and $\bsE_2^{p,\bullet} \cong \coker(\Theta_S^p)$. Since $\opH^i(\Omega_p(V)) = 0$ for $i > 1$ by the Cartier isomorphism, it follows that $E_2^{p,\bullet} \cong \coker(\Theta_\Lambda^p)$ and $E_2^{p-1,\bullet} \cong \ker(\Theta_\Lambda^p)$, and hence $E_2^{s,\bullet} = 0$ for $2 \leq s \leq p-1$. Then $\bsE_2^{0,\bullet} \cong \coker(\Theta_\Lambda^p) \cong E_2^{p,\bullet}$ by the commutativity of the diagram
\[
\xymatrix{
\Lambda^d(V) \ar@{^(->}[r]^{\kappa_p^1(V)} \ar@{->}[d]_{\Theta_\Lambda^p} & Kz_p^1(V) \ar@{->}[d]^{\cong}_{\tau_1} \\
\bsE_1^{0,\bullet} \ar@{->}[r]^{\bsd_1} & \bsE_1^{1,\bullet}
}
\]

Now the only nonzero columns in the two spectral sequences are $\bsE_2^{0,\bullet} \cong E_2^{p,\bullet}$, $\bsE_2^{p-1,\bullet}$, $\bsE_2^{p,\bullet}$, $E_2^{0,\bullet}$, and $E_2^{1,\bullet}$. Since $V$ is concentrated in cohomological degrees $m \equiv p^r \mod 2$, it follows that
\begin{equation} \label{eq:Evanishing}
\bsE_2^{p-1,m} = 0 = E_2^{1,m} \quad \text{if $m \not \equiv p^{r+1} \mod 2p$,}
\end{equation}
and hence that
\begin{equation} \label{eq:E2pmE20m}
\bsE_2^{p,m} \cong \Ext_{\bsp}^m(\Gamma_1^{p(r)},S_0^{p(r)}) \cong E_2^{0,m} \quad \text{if $m \not \equiv p^{r+1} \mod 2p$.}
\end{equation}
Given a $\Z$-graded vector space $W$, write $[W]^m$ for the homogeneous component of degree $m$ in $W$. Then combining \eqref{eq:Evanishing} with the fact that $E_2^{s,\bullet} = 0$ for $2 \leq s \leq p-1$, and the fact that $[E_\infty]^m = 0$ for $m < p^{r+1}$ by the calculations in the previous paragraph, it follows that
\begin{equation} \label{eq:E2isos}
E_2^{0,i} \cong E_p^{0,i} \cong E_p^{p,i-p+1} \cong E_2^{p,i-p+1} \quad \text{for all $i \leq p^{r+1}-2$,}
\end{equation}
with the middle isomorphism induced by the differential $d_p: E_p^{0,i} \rightarrow E_p^{p,i-p+1}$. Similarly, using \eqref{eq:Evanishing}, the fact that $\bsE_2^{s,\bullet} = 0$ for $1 \leq s \leq p-2$, and the fact that $\bsE \Rightarrow 0$ by the exactness of the complex ${Kz_p}^{(r)}$, it follows that there exist isomorphisms
\begin{equation} \label{eq:bsE2isos}
\bsE_2^{0,i} \cong \bsE_p^{0,i} \cong \bsE_p^{p,i-p+1} \cong \bsE_2^{p,i-p+1} \quad \text{for all $i \leq p^{r+1}+p-3$.}
\end{equation}
In particular, \eqref{eq:E2pmE20m}, \eqref{eq:E2isos}, and \eqref{eq:bsE2isos} imply that
\begin{equation} \label{eq:chainofisos}
\bsE_2^{0,i} \cong \bsE_2^{p,i-p+1} \cong E_2^{0,i-p+1} \cong E_2^{p,i-2p+2} \quad \text{for all $i \leq p^{r+1}+p-3$.}
\end{equation}
Trivially, $\bsE_2^{0,i} = 0 = E_2^{0,i}$ if $i < 0$. Then an induction argument using \eqref{eq:chainofisos} and the isomorphism $\bsE_2^{0,\bullet} \cong E_2^{p,\bullet}$ shows that $E_2^{p,i} = 0$ for all $i \leq p^{r+1}+p-3$. In particular, $E_2^{p,i} = 0$ for $i = p^{r+1}-p$ and $i=p^{r+1}-p+1$. This implies that $E_2^{0,p^{r+1}}$ is the only nonzero term of total degree $p^{r+1}$ in the $E_2$-page of its spectral sequence, and that $E_2^{0,p^{r+1}}$ consists of permanent cycles. Since $[E_\infty]^{p^{r+1}} \cong k$ by the calculations in the previous paragraph, we conclude then that $E_2^{0,p^{r+1}} \cong k$. Finally, write $(\bsc_r)^p$ for the $p$-th power of $\bsc_r$ in $S^p(V)$. Then $(\bsc_r)^p$ is in the kernel of the de Rham differential $S^p(V) \rightarrow S^{p-1}(V) \otimes \Lambda^1(V)$. Since $\sigma: \Omega_p(V) \rightarrow E_1$ is a chain map, and since $\sigma_0 = \Theta_S^p$, this implies that $\bsc_r^{\cup p}$ is a nonzero element of  the one-dimensional space $E_2^{0,p^{r+1}}$. Thus, $\bsc_r^{\cup p}$ spans the space of permanent cycles in $E_1^{0,p^{r+1}}$. This completes the proof.
\end{proof}

We have now completed the proof outlined at the start of this section. Thus:

\begin{theorem} \label{thm:nontrivialrestriction}
For each $r \geq 1$, the classes $\bsc_r$ and $\bsc_r^\Pi$ restrict nontrivially to $GL(m|n)_1$.
\end{theorem}

\begin{remark} \label{rem:ercupproduct}
A version of Proposition \ref{prop:imphistar} also holds for the cohomology classes $\bse_{r+1}$ and $\bse_r^{\cup p}$, i.e., the image of $\bse_{r+1}$ in $\Ext_{\bsp}^{2p^r}(\Gamma_0^{p(r)},S_0^{p(r)})$ is a nonzero scalar multiple of the cup product $\bse_r^{\cup p}$. This can be proved through an argument similar to that given for Proposition \ref{prop:imphistar}, or it can be deduced from \cite[Corollary 5.9]{Franjou:1999} using the fact that $\bse_r|_{\bsvzero} = e_r$ for each $r \geq 1$. More generally, it is natural to expect that analogues of the $\Ext$-group calculations in \cite[\S\S4--5]{Franjou:1999} should hold for the exponential superfunctors defined in \eqref{eq:twistcomponents}.
\end{remark}

\subsection{Cohomological finite generation in positive characteristic} \label{subsec:fgpgeq3}

Our main theorem is:

\begin{theorem} \label{thm:fgpgeq3}
Let $k$ be a field of characteristic $p \geq 3$. Let $G$ be a finite $k$-supergroup scheme, and let $M$ be a finite-dimensional $G$-supermodule. Then $\Hbul(G,k)$ is a finitely-generated $k$-algebra, and $\Hbul(G,M)$ is finitely-generated under the cup product action of $\Hbul(G,k)$.
\end{theorem}

\begin{proof}
By \cite[Remark 5.4.3]{Drupieski:2013b} we may assume that $k$ is perfect, and by \cite[Theorem 5.3.3]{Drupieski:2013b} we may assume that $G$ is infinitesimal. Then by \cite[Theorem 5.4.2]{Drupieski:2013b}, it suffices to show for $m,n,r \geq 1$ that, up to rescaling, $e_r^{m,n}:=(\bse_r + \bse_r^\Pi)|_{GL(m|n)}$ and $c_r^{m,n}:=(\bsc_r+\bsc_r^\Pi)|_{GL(m|n)}$ satisfy \ref{eq:ertoG1} and \ref{eq:crtoG1}.\footnote{The assumption $m,n \not\equiv 0 \mod p$ in \cite[Conjecture 5.4.1]{Drupieski:2013b}, and hence also in \cite[Theorem 5.4.2]{Drupieski:2013b}, turns out to be unnecessary, though it should be assumed that $m,n,r \geq 1$.} Since \ref{eq:ertoG1} and \ref{eq:crtoG1} are conditions in terms of superalgebra homomorphisms, it suffices by multiplicativity and linearity to consider $\bse_r$, $\bse_r^\Pi$, $\bsc_r$, and $\bsc_r^\Pi$ separately.

Set $G = GL(m|n)$, and let $\g_{+1}$ and $\g_{-1}$ be the subalgebras of $\g = \glmn$ defined in \eqref{eq:Zgradingglmn}. As in \cite[\S2]{Brundan:2003a}, let $T$ be the maximal torus in $G$ consisting of diagonal matrices, and let $\ve_i: T \rightarrow \G_m$ be the homomorphism that picks out the $i$-th entry of a diagonal matrix. Then the character group $X(T)$ of $T$ is the free abelian group generated by $\ve_1,\ldots,\ve_{m+n}$, and each rational $G$-supermodule $M$ decomposes into a direct sum of $T$-weight spaces. In particular, the matrix unit $e_{ij} \in \glmn$ is of weight $\ve_i - \ve_j$. Now $(\g_{+1}^*)^{(r)}$ is the irreducible $G$-supermodule of highest weight $p^r(\ve_{m+1}-\ve_m)$, and $(\g_{-1}^*)^{(r)}$ is the irreducible $G$-supermodule of highest weight $p^r(\ve_1-\ve_{m+n})$. Since $\bsc_r$ and $\bsc_r^\Pi$ each restrict nontrivially to $G_1$, it follows that the induced $G$-super\-module homomorphisms $\bsc_r|_{G_1} : (\g_{+1}^*)^{(r)} \rightarrow \opH^{p^r}(G_1,k)$ and
$\bsc_r^\Pi|_{G_1} : (\g_{-1}^*)^{(r)} \rightarrow \opH^{p^r}(G_1,k)$ are injections.

Write $\Phi_\odd$ for the set of weights of $T$ in $\gone$. We claim that if $v$ is a vector in the $E_0$-page of \eqref{eq:Mayspecseq} of total degree $p^r$ and weight $p^r\alpha$ for some $\alpha \in \Phi_\odd$, then $v$ is in the image of the homomorphism
\begin{equation} \label{eq:powertoedge}
(\gone^*)^{(r)} \hookrightarrow S(\gone^*) \hookrightarrow \bsl^{p^r}(\g^*) \cong E_0^{0,p^r}
\end{equation}
where the first map raises elements to the $p^r$-power. To see this, suppose without loss of generality that $\alpha$ is a weight of $T$ in $\g_{+1}^*$, say, $\alpha = \ve_{m+j} - \ve_i$ for some $1 \leq i \leq m$ and $1 \leq j \leq n$, so that $v$ is of weight $p^r(\ve_{m+j} - \ve_i) = (-p^r)\ve_i + p^r \ve_{m+j}$. Next observe that $\bsl(\g^*) \cong \Lambda(\gzero^*) \gotimes S(\gone^*)$ and $S(\gone^*) \cong S(\g_{+1}^*) \otimes S(\g_{-1}^*)$ as $T$-algebras. If $w$ is a vector of weight $\sum_{s=1}^{m+n} a_s \ve_s$ in either $\Lambda(\gzero^*)$ or $S(\gzero^*(2))^{(1)}$, then $\sum_{s=1}^m a_s = 0$ and $\sum_{s=1}^n a_{m+s} = 0$. On the other hand, if $w$ is a weight vector in $S^c(\g_{+1}^*)$, then $\sum_{s=1}^m a_s = -c$ and $\sum_{s=1}^n a_{m+s} = c$, while if $w$ is a weight vector in $S^c(\g_{-1}^*)$, then $\sum_{s=1}^m a_s = c$ and $\sum_{s=1}^n a_{m+s} = -c$. Combining these observations, it follows first that $v$ must be a weight vector in $S^{p^r}(\g_{+1}^*)$, and then that $v$ must be in the image of \eqref{eq:powertoedge}. Now by the irreducibility of $(\g_{+1}^*)^{(r)}$ and $(\g_{-1}^*)^{(r)}$ as $G$-supermodules and by dimension comparison, it follows that the composition of the map $(\bsc_r+\bsc_r^\Pi)|_{G_1}: (\gone^*)^{(r)} \rightarrow \opH^{p^r}(G_1,k)$ with the vertical edge map of \eqref{eq:Mayspecseq} must have the same image as \eqref{eq:powertoedge}, and hence that $c_r^{m,n}:=(\bsc_r+\bsc_r^\Pi)|_G$ satisfies \ref{eq:crtoG1}.

Next we show by induction on $r$ that, up to rescaling, \ref{eq:ertoG1} is satisfied by $e_r^{m,n} := (\bse_r+\bse_r^\Pi)|_G$. The inductive step is handled by an argument like that in the second paragraph of Section \ref{subsec:restrictioncr}, using Remark \ref{rem:ercupproduct} instead of Proposition \ref{prop:imphistar}, so it suffices to consider the case $r=1$.

Recall the discussion from Section \ref{subsec:Xg}. We will work with the homogeneous basis for $\glmn$ of the matrix units $\set{e_{ij}: 1 \leq i,j \leq m+n}$. Then $d_1: X_1(\g) \rightarrow X_0(\g)$ and $d_2: X_2(\g) \rightarrow X_1(\g)$ satisfy
\begin{gather*}
d_1(1 \otimes e_{ij} \otimes 1) = e_{ij} \otimes 1 \otimes 1, \\
d_2(1 \otimes (e_{ij}\cdot e_{st}) \otimes 1) = e_{ij} \otimes e_{st} \otimes 1 - (-1)^{\ol{e_{ij}} \cdot \ol{e_{st}}} e_{st} \otimes e_{ij} \otimes 1 - 1 \otimes [e_{ij},e_{st}] \otimes 1, \quad \text{and} \\
d_2(1 \otimes 1 \otimes e_{ij}) = (e_{ij})^{p-1} \otimes e_{ij} \otimes 1 - 1 \otimes e_{ij}^{[p]} \otimes 1 \quad \text{for $\ol{e_{ij}} = \ol{0}$.}
\end{gather*}
Here $z \mapsto z^{[p]}$ is the map that sends $z \in \gzero$ to its associative $p$-th matrix power.

Now given $e_{ij} \in \g$, let $e_{ij}^*$ be the corresponding dual basis element, and given $e_{ij} \in \gzero$, let $g_{ij} \in \Hom_{V(\g)}(X_2(\g),k)$ be the homomorphism that is naturally dual to $1 \otimes 1 \otimes e_{ij}$. Then the $g_{ij}$ are cocycles in $\Hom_{V(\g)}(X_2(\g),k)$ by \cite[Lemma 3.5.2]{Drupieski:2013b}, and the proof of \cite[Proposition 3.5.3]{Drupieski:2013b} shows (modulo a reindexing of the spectral sequence) that the horizontal edge map of \eqref{eq:Mayspecseq} sends $(e_{ij}^*)^{(1)} \in (\gzero^*)^{(1)} \cong E_0^{2,0}$ to the cohomology class of $g_{ij}$. Then to finish the proof, it suffices to show that $(\bse_1 + \bse_1^\Pi)|_{G_1} : (\gzero^*)^{(1)} \rightarrow \opH^2(G_1,k)$ also sends $(e_{ij}^*)^{(1)}$ to the cohomology class of $g_{ij}$.

Recall that $\bse_1$ is the extension class in $\bsp$ of the exact sequence \eqref{eq:e1extension}. Then $\bse_1(k^{m|n})$ is made into an exact sequence of restricted $\g$-supermodules by having $\g$ act trivially on $k^{m|0(1)}$, and by giving $\bss^p(k^{m|n})$ and $\bsg^p(k^{m|n})$ the $\g$-supermodule structures induced by the natural action of $\g$ on $(k^{m|n})^{\otimes p}$. As in the proof of Lemma \ref{lem:c1restriction}, the cohomology class $\bse_1|_{G_1} \in \Ext_{G_1}^2(k^{m|0(1)},k^{m|0(1)})$ can be described by lifting the identity map $k^{m|0(1)} \rightarrow k^{m|0(1)}$ to a chain map
\begin{equation} \label{eq:e1chainmap}
\vcenter{\xymatrix{
\cdots \ar@{->}[r] & X_2(\g) \otimes k^{m|0(1)} \ar@{->}[r] \ar@{->}[d]^{\varphi_2} & X_1(\g) \otimes k^{m|0(1)} \ar@{->}[r] \ar@{->}[d]^{\varphi_1} & X_0(\g) \otimes k^{m|0(1)} \ar@{->}[r] \ar@{->}[d]^{\varphi_0} & k^{m|0(1)} \ar@{=}[d] \\
 & k^{m|0(1)} \ar@{->}[r] & \bss^p(k^{m|n}) \ar@{->}[r] & \bsg^p(k^{m|n}) \ar@{->}[r] & k^{m|0(1)},
}}
\end{equation}
and then taking the cohomology class of the map $\varphi_2$.

As in the proof of Lemma \ref{lem:c1restriction}, let $x_1,\ldots,x_m,y_1,\ldots,y_n$ denote the standard homogeneous basis for $k^{m|n}$. Then the reader can check that the following formulas uniquely define a chain homomorphism $\varphi: X(\g) \otimes k^{m|0(1)} \rightarrow \bse_1(k^{m|n})$ lifting the identity map on $k^{m|0(1)}$:
\begin{itemize}
\item $\varphi_0((1 \otimes 1 \otimes 1) \otimes x_s^{(1)}) = \gamma_p(x_s)$,
\item $\varphi_1((1 \otimes e_{ij} \otimes 1) \otimes x_s^{(1)}) = \delta_{j,s} \cdot \frac{1}{(p-1)!} \cdot x_i \cdot x_s^{p-1}$ for $e_{ij} \in \g$,
\item $\varphi_2((1 \otimes 1 \otimes e_{ij}^{(1)}) \otimes x_s^{(1)}) = \delta_{j,s} \cdot x_i^{(1)}$ for $e_{ij} \in \gzero$, and
\item $\varphi_2((1 \otimes z \otimes 1) \otimes x_s^{(1)}) = 0$ for all $z \in \bsa^2(\g)$.
\end{itemize}
This implies that if $e_{ij}^* \in \g_m^*$, then $\bse_1|_{G_1} : (\g_m^*)^{(1)} \rightarrow \opH^2(G_1,k)$ sends ${e_{ij}^*}^{(1)}$ to the cohomology class of $g_{ij}$. Similarly, if $e_{ij}^* \in \g_n^*$, then $\bse_1^\Pi|_{G_1} : (\g_n^*)^{(1)} \rightarrow \opH^2(G_1,k)$ sends ${e_{ij}^*}^{(1)}$ to the class of $g_{ij}$. Thus, $(\bse_1+\bse_1^\Pi)|_{G_1} : (\gzero^*)^{(1)} \rightarrow \opH^2(G_1,k)$ is equal to the corresponding edge map of \eqref{eq:Mayspecseq}, and extending multiplicatively, we conclude for $r=1$ that \ref{eq:ertoG1} is satisfied by $e_1^{m,n} := (\bse_1+\bse_1^\Pi)|_G$.
\end{proof}

Applying the equivalence between finite supergroup schemes over $k$ and finite-dimensional cocommutative Hopf superalgebras over $k$, we immediately get:

\begin{corollary} \label{cor:fgpgeq3}
Let $k$ be a field of characteristic $p \geq 3$, let $A$ be a finite-dimensional cocommutative Hopf superalgebra over $k$, and let $M$ be a finite-dimensional $A$-supermodule. Then the cohomology ring $\Hbul(A,k)$ is a finitely-generated $k$-algebra, and $\Hbul(A,M)$ is finitely-generated under the cup product action of $\Hbul(A,k)$.
\end{corollary}

Using the observations from the proof of Theorem \ref{thm:fgpgeq3}, we can now prove Lemma \ref{lem:differential}.

\begin{proof}[Proof of \ref{lem:differential}]
We prove that if $s > 0$ and $s \equiv 0 \mod 2p^{r-1}$, then the differential must be an isomorphism of one-dimensional vector spaces. First, by the discussion preceding Proposition \ref{prop:basesj=1}, which reinterprets the indicated differential in terms of right multiplication by the extension class $\bse_r$, the claim of the lemma is equivalent to the claim that $(\bse_r)^\ell \neq 0$ for all $\ell \geq 0$. The observations in Step 3 of the proof of Theorem \ref{thm:quadruplecalculation} already imply that $(\bse_r)^\ell \neq 0$ for $0 \leq \ell < p$. Furthermore, the isomorphism \eqref{eq:2prshifteven}, whose validity does not depend on Lemma \ref{lem:differential}, implies by induction that if in addition $(\bse_r)^p \neq 0$, then $(\bse_r)^\ell \neq 0$ for all $\ell \geq 0$. So it suffices to show that $(\bse_r)^p \neq 0$.

Let $m,n \geq 1$, and set $G = GL(m|n)_1$, the first Frobenius kernel of $GL(m|n)$. We will show that $(\bse_r)^p \neq 0$ by proving that $(\bse_r)^p|_G = (\bse_r|_G)^p \neq 0$. As discussed in Section \ref{subsec:recollections}, the restriction $\bse_r|_G$ determines a cohomology class
\[
\bse_r|_G \in \opH^{2p^{r-1}}(G,\g_m^{(r)}) \cong \opH^{2p^{r-1}}(G,k) \otimes \g_m^{(r)} \cong \Hom_k((\g_m^*)^{(r)},\opH^{2p^{r-1}}(G,k)).
\]
Moreover, the proof of Theorem \ref{thm:fgpgeq3} asserts that, up to a nonzero scalar factor, the induced map $\bse_r|_G: (\g_m^*)^{(r)} \rightarrow \opH^{2p^{r-1}}(G,k)$ factors through the composite of the $p^{r-1}$-power map $(\g_m^*)^{(r)} \rightarrow S(\g_m^*)^{(1)}$ and the horizontal edge map $\Phi: S(\gzero^*)^{(1)} \rightarrow \Hbul(G,k)$ of \eqref{eq:Mayspecseq}. Let $\{e_{ij}: 1 \leq i,j \leq m+n \}$ be the basis of matrix units for $\g$, and let $\{e_{ij}^*: 1 \leq i,j \leq m+n \}$ be the dual basis. Given $e_{ij} \in \gzero$, set $X_{ij} = \Phi(e_{ij}^*)$. Then considering $\bse_r|_G$ as an element of $\opH^{2p^{r-1}}(G,k) \otimes \g_m^{(r)}$, we can write
\[
\bse_r|_G = \lambda \cdot \left( \sum_{1 \leq i,j \leq m} X_{ij}^{p^{r-1}} \otimes e_{ij}^{(r)} \right)
\]
for some nonzero $\lambda \in k$. Now it follows that $(\bse_r)^p|_G = (\bse_r|_G)^p \in \opH^{2p^r}(G,k) \otimes \g_m^{(r)}$ is given by
\[
\lambda^p \cdot \left( \sum_{1 \leq i,j \leq m} X_{ij}^{p^{r-1}} \otimes e_{ij}^{(r)} \right)^p = \lambda^p \cdot \left( \sum_{\substack{1 \leq i,j \leq m \\ 1 \leq i_1,\ldots,i_{p-1} \leq m}} X_{i,i_1}^{p^{r-1}} X_{i_1,i_2}^{p^{r-1}} \cdots X_{i_{p-1},j}^{p^{r-1}} \otimes e_{ij}^{(r)} \right).
\]
So to prove that $(\bse_r|_G)^p \neq 0$, it suffices to show for some $1 \leq i,j \leq m$ that $(Z_{ij}^p)^{p^{r-1}} \neq 0$, where
\[
Z_{ij}^p := \sum_{1 \leq i_1,\ldots,i_{p-1} \leq m} X_{i,i_1} X_{i_1,i_2} \cdots X_{i_{p-1},j} \in \opH^{2p}(G,k).
\]
Write $\Phi^{-1}(Z_{ij}^p)$ to denote the polynomial in $S^p(\g_m^*)$ obtained by replacing each $X_{uv}$ in $Z_{ij}^p$ with the original coordinate function $e_{uv}^*$. Then $\Phi^{-1}(Z_{ij}^p)$ is the natural polynomial function on $\g_m$ defining the condition that the $ij$-coordinate of the $p$-th power of a matrix in $\g_m$ should be equal to $0$.

We show that $(Z_{mm}^p)^{p^{r-1}} \neq 0$ by showing that its restriction to the cohomology of a particular restricted Lie sub-superalgebra of $\glmn$ is nonzero. Specifically, let $x:= e_{m,m}+e_{m+1,m+1}$ and let $y:=e_{m+1,m}+e_{m,m+1}$. Then $x$ and $y$ span a two-dimensional restricted Lie sub-superalgebra $\mathfrak{s}$ of $\glmn$ with $[x,y] = 0$, $x^{[p]} = x$, and $[y,y] = 2x$. Moreover, one can check that the restriction map $S(\glmn_{\ol{0}}^*) \rightarrow S(\mathfrak{s}_{\ol{0}}^*)$ sends the polynomial $\Phi^{-1}(Z_{mm}^p)$ to $(x^*)^p \in S^p(\mathfrak{s}_{\ol{0}}^*)$, where $x^* \in \mathfrak{s}_{\ol{0}}^*$ denotes the basis vector dual to $x$. Now the fact that $(Z_{mm}^p)^{p^{r-1}}$ restricts to a nontrivial class in $\Hbul(V(\mathfrak{s}),k)$ follows from the observation at the end of Example \ref{ex:twodim}.
\end{proof}

\subsection{Implications for the cohomology of \texorpdfstring{$GL(m|n)$}{GL(m|n)}} \label{subsec:implications}

Recall from \cite[Corollary 3.13]{Friedlander:1997} that if $T$ and $T'$ are homogeneous strict polynomial functors of degree $d$, and if $n \geq d$, then evaluation on the vector space $k^n$ induces an isomorphism $\Ext_{\cp_d}^\bullet(T,T') \cong \Ext_{GL_n}^\bullet(T(k^n),T'(k^n))$. Taking $d = 0$, the category $\cp_0$ identifies with the semisimple category of finite-dimensional $k$-vector spaces, and $GL_n$ acts trivially on $T(k^n)$ for each $T \in \cp_0$. Then a special case of the previous isomorphism is the well-known fact that $\Hom_{GL_n}(k,k) \cong k$, but $\Ext_{GL_n}^i(k,k) = 0$ for all $i > 0$. The next proposition implies that there is no analogue of \cite[Corollary 3.13]{Friedlander:1997} for $\bsp_d$ and $GL(m|n)$.

\begin{proposition} \label{prop:GLmncohomologyring}
Let $m$ and $n$ be positive integers. Then $\Ext_{GL(m|n)}^2(k,k) \neq 0$.
\end{proposition}

\begin{proof}
Set $G = GL(m|n)$, and set $\g = \Lie(G) = \glmn$. Write $\gzero = \g_m \oplus \g_n$ as in \eqref{eq:Zgradingglmn}, and let $\tr_m: \g_m \rightarrow k$ and $\tr_n: \g_n \rightarrow k$ be the usual trace functions on $\g_m$ and $\g_n$, respectively. We consider $\tr_m$ and $\tr_n$ as elements of $\g^*$ in the obvious way. Then the supertrace function $\str : \glmn \rightarrow k$ is defined by $\str = \tr_m - \tr_n$. By \cite[Corollary 5.2.3]{Drupieski:2013b}, the $E_2$-page of the spectral sequence \eqref{eq:Mayspecseq} has the form $E_2^{i,j} = \opH^j(\g,k) \otimes S^i(\gzero^*(2))^{(1)}$, where $\Hbul(\g,k)$ denotes the ordinary Lie superalgebra cohomology of $\g$. Immediate calculation using the Koszul resolution for $\g$ \cite[\S3.1]{Drupieski:2013b} shows that $\opH^1(\g,k) \cong (\g/[\g,\g])^*$ is spanned by the supertrace function. Now one can check that the differential $d_2: E_2^{0,1} \rightarrow E_2^{2,0}$ maps $E_2^{0,1} \cong \opH^1(\g,k)$ onto the subspace of $E_2^{2,0} \cong (\gzero^*)^{(1)}$ spanned by $\str^{(1)}$.\footnote{This can be checked, for example, by considering the construction of the spectral sequence \eqref{eq:Mayspecseq} via the resolution $X(\g)$, and then looking at the low degree terms in $X(\g)$.} This implies that the horizontal edge map of \eqref{eq:Mayspecseq} induces an injection $(\gzero^* / (k \cdot \str))^{(1)} \hookrightarrow \opH^2(G_1,k)$. In particular, the image of ${\tr_m}^{(1)} \in (\gzero^*)^{(1)}$ under the edge map $(\gzero^*)^{(1)} \rightarrow \opH^2(G_1,k)$ is nonzero.

Now as in the proof of \cite[Lemma 1.4]{Friedlander:1997}, consider the commutative diagram
\[
\xymatrix{
\Ext_{G}^2(k,\g_m^{(1)}) \ar@{->}[r]^-{z \mapsto z|_{G_1}} \ar@{->}[d]^{(\tr_m^{(1)})_*} & \Ext_{G_1}^2(k,\g_m^{(1)}) \ar@{->}[r]^-{\sim} & \Hom_k((\g_m^*)^{(1)},\opH^2(G_1,k)) \ar@{->}[d]^{\circ \tr_m^{(1)}} \\
\Ext_{G}^2(k,k) \ar@{->}[rr]^-{z \mapsto z|_{G_1}} & & \Ext_{G_1}^2(k,k).
}
\]
The proof of Theorem \ref{thm:fgpgeq3} and the observation at the end of the previous paragraph imply that the image of $\bse_1|_G$ across the top row and right-hand column of this diagram is nonzero. Then the image of $\bse_1|_G$ down the left-hand column and across the bottom row of the diagram must also be nonzero. In particular, $({\tr_m}^{(1)})_*(\bse_1|_G)$ must be a nonzero class in $\Ext_G^2(k,k)$.
\end{proof}

\begin{problem}
Compute the rational cohomology ring $\Hbul(GL(m|n),k) = \Ext_{GL(m|n)}^\bullet(k,k)$.
\end{problem}

\begin{remark} \label{rem:nontrivialext}
It follows from a result of Kujawa \cite[Lemma 3.6]{Kujawa:2006a} that $\Ext_{GL(m|n)}^1(L,L) = 0$ for each irreducible rational $GL(m|n)$-supermodule $L$. On the other hand, Brundan and Kleshchev have shown for the supergroup $Q(n)$ that $\opH^1(Q(n),k) \cong \Pi(k)$ \cite[Corollary 7.8]{Brundan:2003}.
\end{remark}

\subsection{Cohomological finite generation in characteristic zero} \label{subsec:fgp=0}

In this section we prove cohom\-ological finite-generation for finite-dimensional cocommutative Hopf superalgebras over fields of characteristic zero, using a theorem of Kostant describing the structure of these algebras. The next result is Theorem 3.3 of \cite{Kostant:1977} (it is evident from the context that the theorem as published contains a typo: the word `commutative' should be the word `cocommutative'). Given a group $G$, we write $kG$ for the group algebra of $G$ over $k$, i.e., the ring of $k$-linear combinations of elements of $G$, in which multiplication is induced by multiplication in $G$.

\begin{theorem}[Kostant] \label{thm:Kostant}
Let $k$ be an algebraically closed field of characteristic zero, and let $A$ be a cocommutative Hopf super\-algebra over $k$. Let $G$ be the group of group-like elements in $A$, let $\g$ be the Lie superalgebra of primitive elements in $A$, and let $U(\g)$ be the universal enveloping superalgebra of $\g$. Then there exists an isomorphism of Hopf superalgebras $A \cong kG \# U(\g)$, where the smash product is taken with respect to the homomorphism $\pi: G \rightarrow GL(\g)$ defined by $\pi(g)(x) = gxg^{-1}$.
\end{theorem}

Since $U(\g)$ is infinite-dimensional whenever $\gzero \neq 0$, and since any purely odd Lie superalgebra is automatically abelian, we immediately get:

\begin{corollary} \label{cor:Kostant}
Let $k$ be an algebraically closed field of characteristic zero, and let $A$ be a finite-dimensional cocommutative Hopf superalgebra over $k$. Then there exists a finite group $G$, a finite-dimensional odd superspace $V$, and a representation $\pi: G \rightarrow GL(V)$ such that $A$ is isomorphic as a Hopf superalgebra to the smash product $kG \# \Lambda(V)$ formed with respect to $\pi$.
\end{corollary}

Now we prove the characteristic zero analogue of Corollary \ref{cor:fgpgeq3}.

\begin{theorem} \label{thm:fgp=0}
Let $k$ be a field of characteristic zero, let $A$ be a finite-dimensional cocommutative Hopf superalgebra over $k$, and let $M$ be a finite-dimensional $A$-supermodule. Then $\Hbul(A,k)$ is a finitely-generated $k$-algebra, and $\Hbul(A,M)$ is a finitely-generated $\Hbul(A,k)$-module.
\end{theorem}

\begin{proof}
Since $\Hbul(A,k) \otimes \ol{k} \cong \Hbul(A \otimes \ol{k}, \ol{k})$ as $\ol{k}$-algebras, and since $\Hbul(A,M) \otimes \ol{k} \cong \Hbul(A \otimes \ol{k},M \otimes \ol{k})$ as modules under the previous isomorphism, we may assume that $k$ is algebraically closed. Then by Corollary \ref{cor:Kostant}, there exists a finite group $G$, a finite-dimensional odd superspace $V$, and a representation $\pi: G \rightarrow GL(V)$ such that $A$ is isomorphic as a Hopf superalgebra to the smash product $kG \# \Lambda(V)$ formed with respect to $\pi$. Since the exterior algebra $\Lambda(V)$ is a normal subalgebra of $A$, and since the Hopf superalgebra quotient $A//\Lambda(V)$ is isomorphic to the group algebra $kG$, we get for each $A$-supermodule $M$ a Lyndon--Hochschild--Serre spectral sequence
\[
E_2^{i,j}(M) = \opH^i(kG, \opH^j(\Lambda(V),M)) \Rightarrow \opH^{i+j}(A,M).
\]
The group algebra $kG$ is semisimple by Maschke's Theorem, so $E_2^{i,j}(M) = 0$ for $i > 0$. Then the spectral sequence collapses at the $E_2$-page, yielding $\Hbul(A,M) \cong \Hbul(\Lambda(V),M)^G$. Next, the ring $\Hbul(\Lambda(V),k)$ is isomorphic as a graded $k$-algebra to the symmetric algebra $S(V^*)$. This identification is an isomorphism of $kG$-modules, where the $kG$-module structure on $S(V^*)$ is induced by the contragredient representation of $G$ on $V^*$. Then by Noether's finiteness theorem, the cohomology ring $\Hbul(A,k) \cong S(V^*)^G$ is a finitely-generated $k$-algebra, and $S(V^*)$ is finitely-generated as a module over $S(V^*)^G$. Then to prove that $\Hbul(A,M)$ is a finitely-generated $\Hbul(A,k)$-module, it suffices to show that $\Hbul(\Lambda(V),M)$ is a finitely-generated $\Hbul(\Lambda(V),k)$-supermodule.

Consider $V$ as an abelian Lie superalgebra. Then $\Lambda(V)$ is the enveloping superalgebra of $V$, and $\Hbul(\Lambda(V),M)$ is the Lie superalgebra cohomology group $\Hbul(V,M)$ studied in \cite[\S\S3.1--3.2]{Drupieski:2013b}. Let $C^\bullet(V,M)$ be the cochain complex defined in \cite[\S3.2]{Drupieski:2013b} whose cohomology is equal to $\Hbul(V,M)$. Then the cup product makes $C^\bullet(V,M)$ into a finitely-generated differential graded (either left or right) $C^\bullet(V,k)$-super\-module; see the discussion after \cite[Remark 3.2.2]{Drupieski:2013b} and Footnote 2 in \cite[\S2.6]{Drupieski:2013b}. Since $V$ is abelian, the differntial on $C^\bullet(V,k)$ is trivial, and we get $\Hbul(V,k) \cong C^\bullet(V,k) \cong S(V^*)$. Passing to cohomology, this implies that $\Hbul(V,M)$ is finitely-generated under the (left or right) cup product action of $\Hbul(V,k) \cong C^\bullet(V,k)$.
\end{proof}

\makeatletter
\renewcommand*{\@biblabel}[1]{\hfill#1.}
\makeatother

\bibliographystyle{eprintamsplain}
\bibliography{CFG}

\begin{figure}[p]
\includegraphics{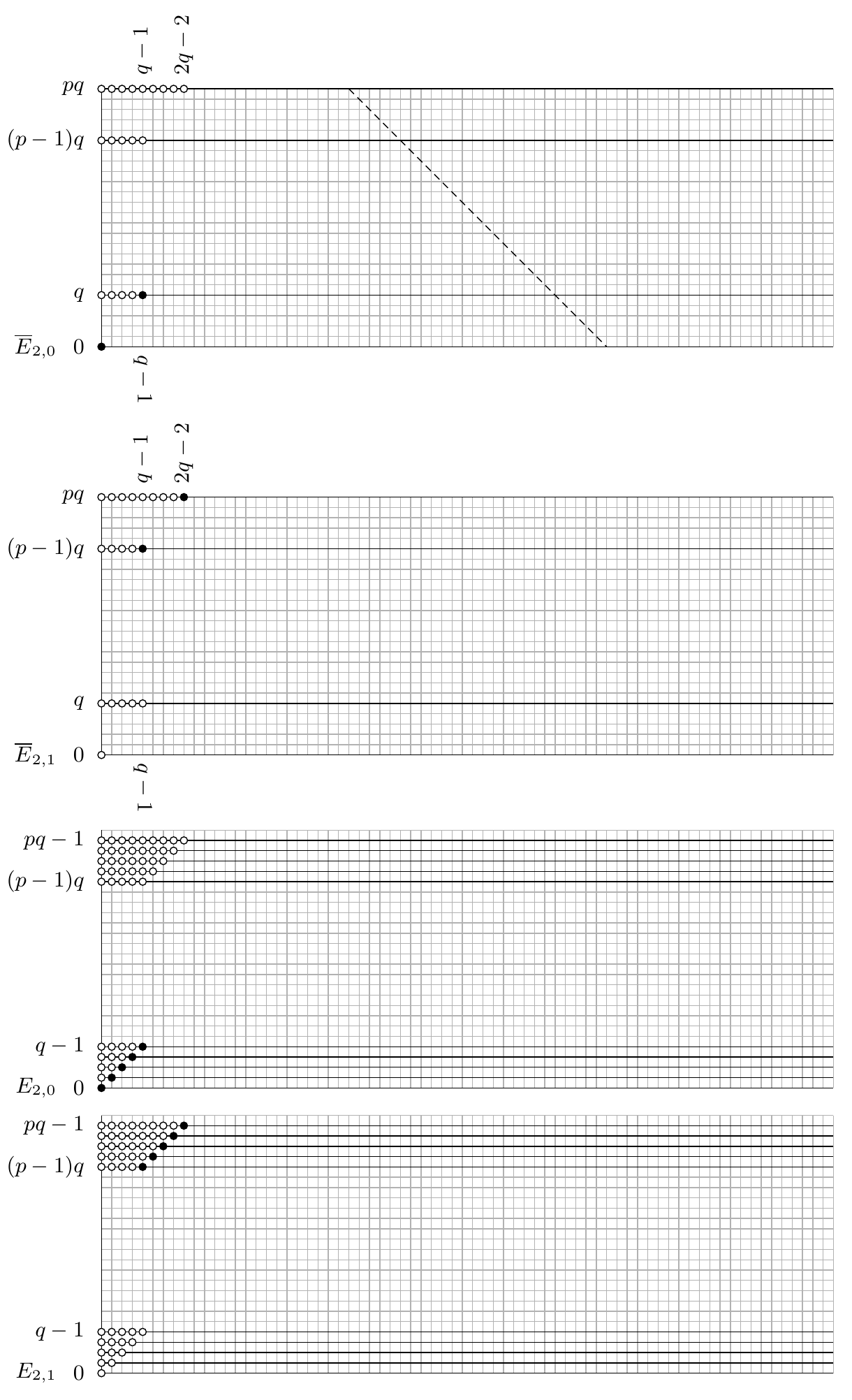}
\caption{End of Step 1 of the proof of Theorem \ref{thm:quadruplecalculation}.} \label{fig:step1}
\end{figure}

\begin{figure}[p]
\includegraphics{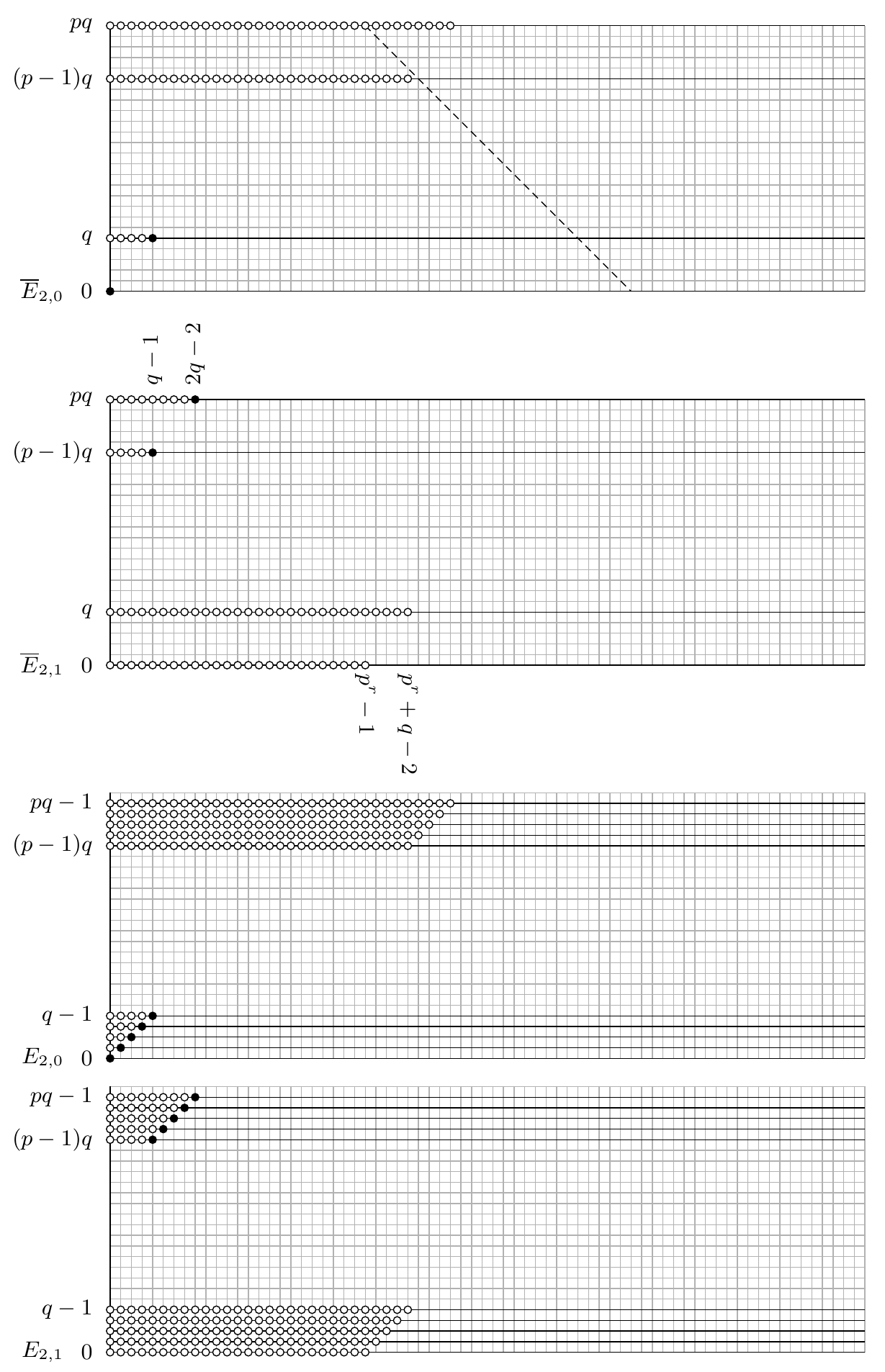}
\caption{End of Step 2 of the proof of Theorem \ref{thm:quadruplecalculation}.} \label{fig:step2}
\end{figure}

\begin{figure}[p]
\includegraphics{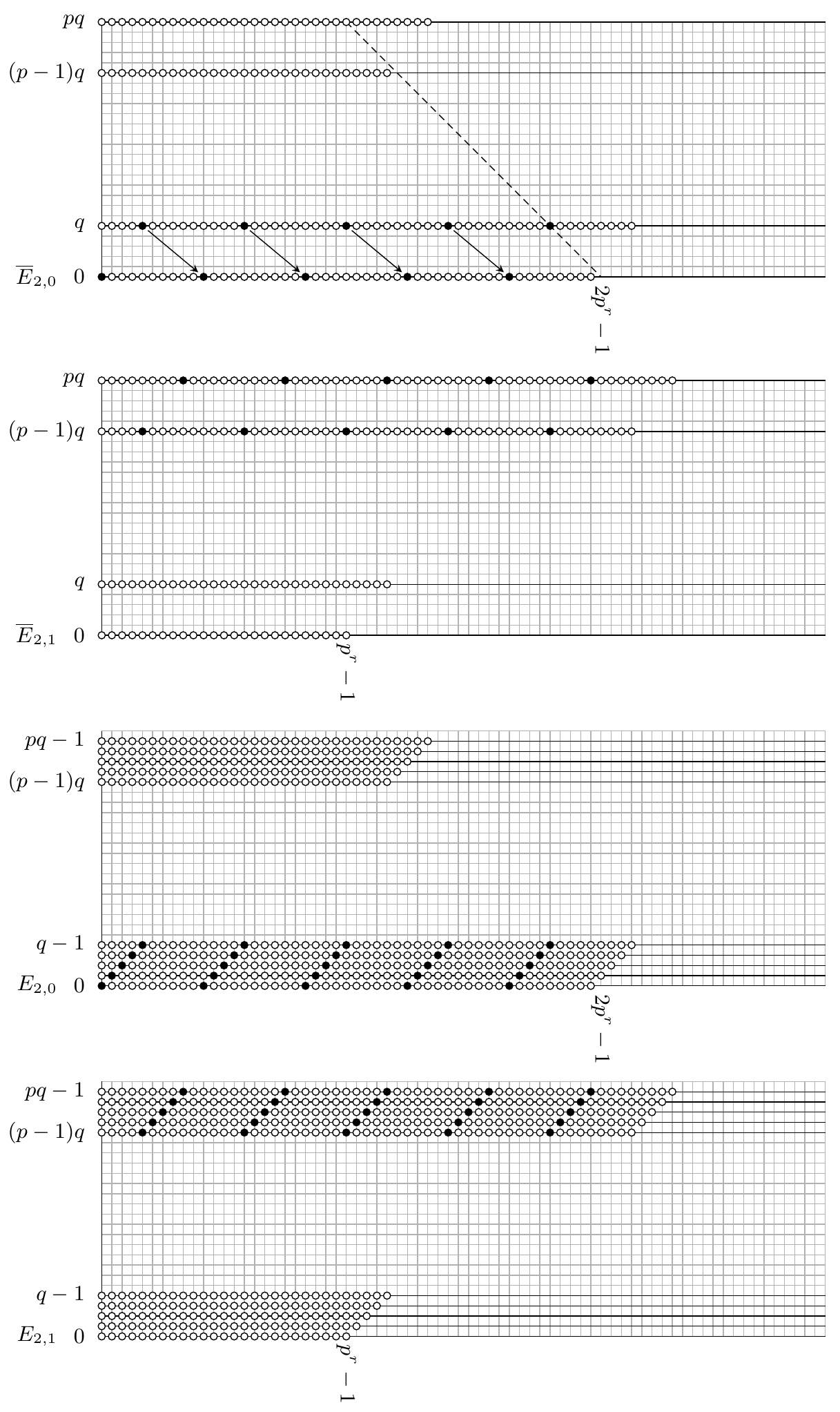}
\caption{End of Step 3 of the proof of Theorem \ref{thm:quadruplecalculation}.} \label{fig:step3}
\end{figure}

\begin{figure}[p]
\includegraphics{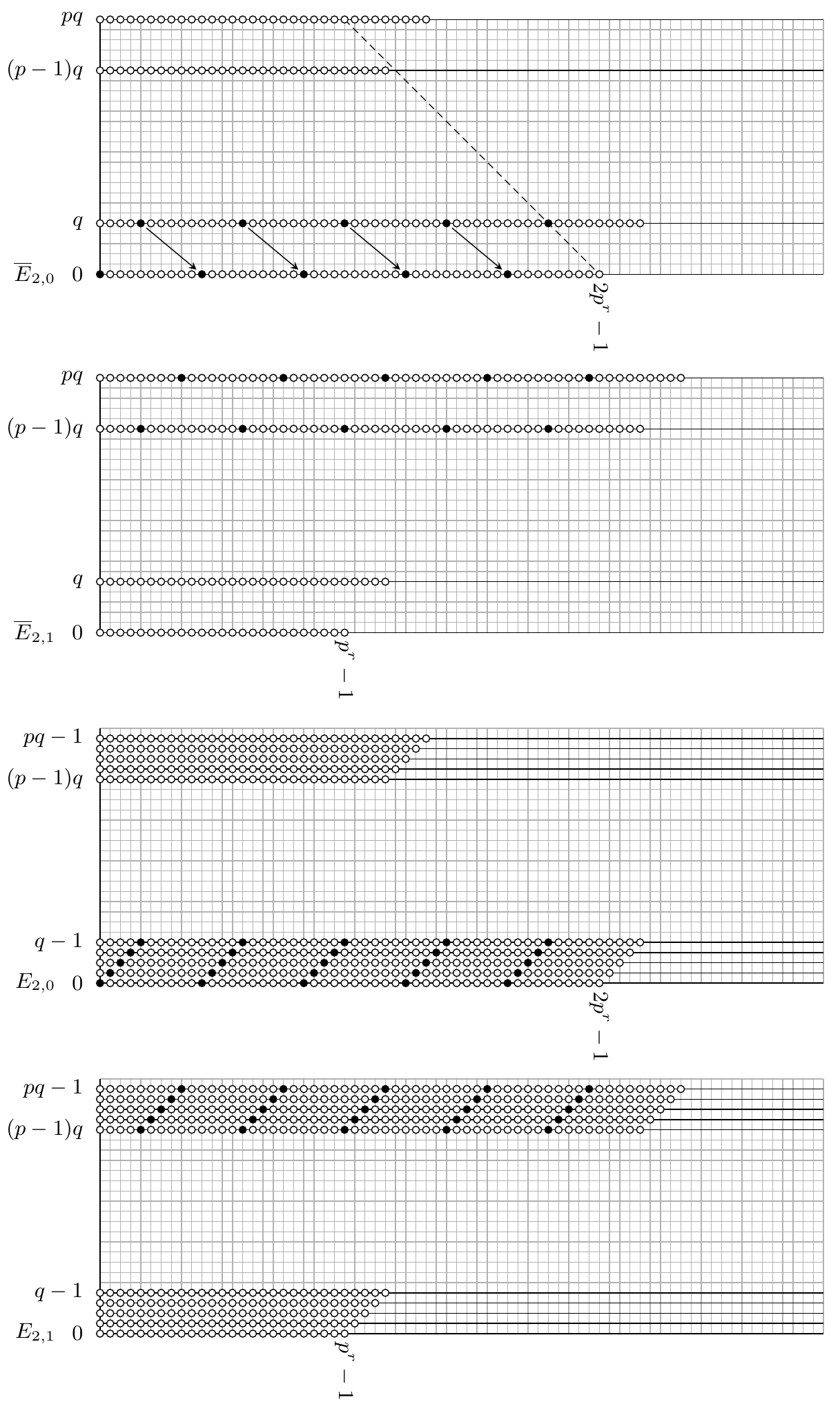}
\caption{End of Step 4 of the proof of Theorem \ref{thm:quadruplecalculation}.} \label{fig:step4}
\end{figure}

\begin{figure}[p]
\includegraphics{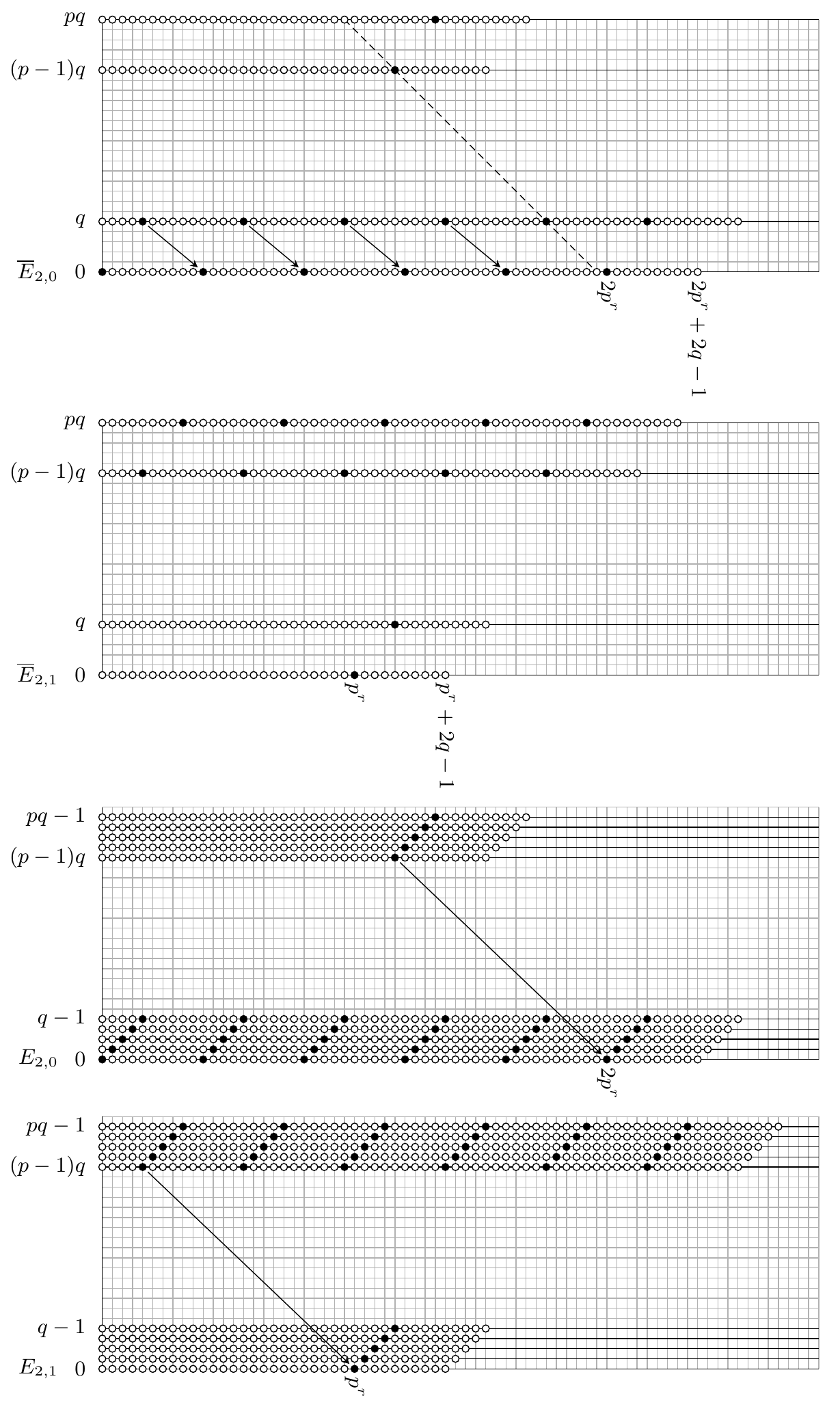}
\caption{End of Step 5 of the proof of Theorem \ref{thm:quadruplecalculation}.} \label{fig:step5}
\end{figure}

\begin{figure}[p]
\includegraphics{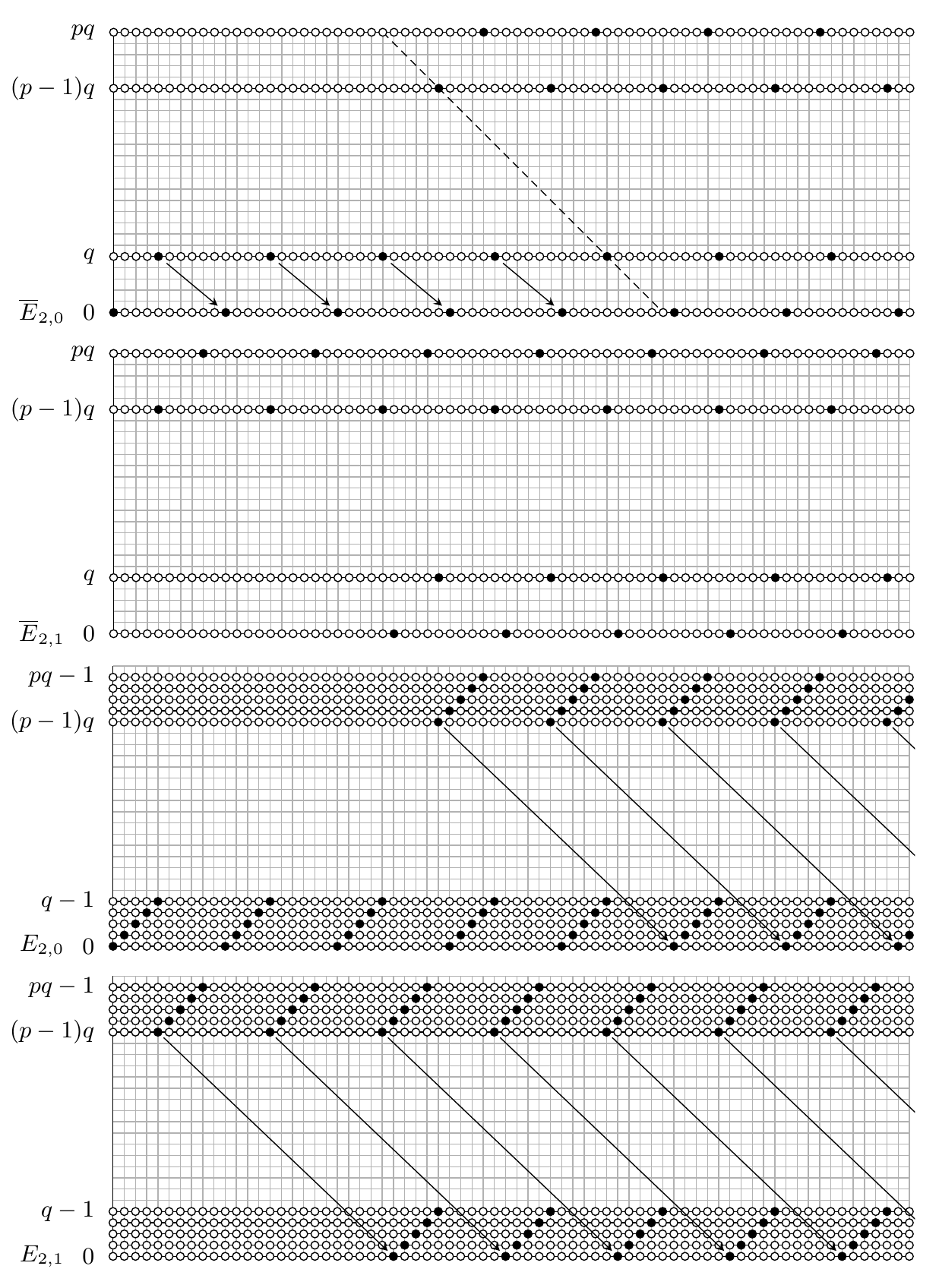}
\caption{End of the proof of Theorem \ref{thm:quadruplecalculation}.} \label{fig:endsteps}
\end{figure}

\end{document}